\pdfoutput=1
\documentclass{amsart}

\usepackage[left=1.25in, right=1.25in, top=1.5in, bottom=1.5in]{geometry}
\usepackage{amsmath, amssymb, amsthm, tikz-cd}
\usepackage{amscd}
\usepackage{mathtools}
\usepackage{enumitem}
\usetikzlibrary{fit}
\usepackage[colorlinks=true, linkcolor=black, citecolor=black, urlcolor=black]{hyperref}
\usepackage{cite}

\newtheorem{theorem}{Theorem}[section]
\newtheorem{proposition}[theorem]{Proposition}
\newtheorem{lemma}[theorem]{Lemma}
\newtheorem{corollary}[theorem]{Corollary}

\newtheorem{thmletter}{Theorem}

\makeatletter
\let\c@equation\c@theorem

\makeatother

\theoremstyle{definition}
\newtheorem{definition}[theorem]{Definition}
\newtheorem{construction}[theorem]{Construction}
\newtheorem{convention}[theorem]{Convention}
\newtheorem{example}[theorem]{Example}

\theoremstyle{remark}
\newtheorem{remark}[theorem]{Remark}

\DeclareMathOperator{\Spf}{Spf}
\DeclareMathOperator{\Spec}{Spec}

\DeclareMathOperator{\Perf}{Perf}

\DeclareMathOperator{\Map}{Map}
\DeclareMathOperator{\CAlg}{CAlg}
\DeclareMathOperator{\colim}{colim}
\DeclareMathOperator{\fib}{fib}

\DeclareMathOperator{\RHom}{RHom}
\DeclareMathOperator{\cofib}{cofib}
\DeclareMathOperator{\Rees}{Rees}
\DeclareMathOperator{\Ksz}{Kos}
\DeclareMathOperator{\Fil}{Fil}

\DeclareMathOperator{\Hom}{Hom}

\DeclareMathOperator{\Bock}{Bock}
\DeclareMathOperator{\gr}{gr}

\newcommand{\Gm}{\mathbb G_m}

\newcommand{\Zp}{\mathbb Z_p}

\newcommand{\cO}{\mathcal O}
\newcommand{\Dual}{\mathbb D}

\newcommand{\cH}{\mathcal H}

\newcommand{\cN}{\mathcal N}

\newcommand{\NN}{\mathcal{NN}}

\newcommand{\syn}{\mathrm{Syn}}
\newcommand{\Nyg}{\mathcal N}

\newcommand{\HW}{\mathrm{HW}}
\newcommand{\dR}{\mathrm{dR}}
\newcommand{\dRW}{\mathrm{dRW}}
\newcommand{\HT}{\mathrm{HT}}
\newcommand{\crys}{\mathrm{crys}}
\newcommand{\et}{\mathrm{\acute et}}

\newcommand{\comp}{\mathrm{comp}}

\newcommand{\RHomgr}{\operatorname{RHom}^{\mathrm{gr}}}
\newcommand{\Dqc}{D_{\mathrm{qc}}}

\newcommand{\WOm}{W\Omega}

\begin{document}

\title{The stacky approach to de Rham--Witt cohomology}
\author{Yuanning Zhang}
\address{Department of Mathematics, Northwestern University}
\email{yuanningzhang2026@u.northwestern.edu}
\date{}
\hypersetup{
  pdftitle={The stacky approach to de Rham--Witt cohomology},
  pdfauthor={Yuanning Zhang}
}

\begin{abstract}
Let $k$ be a perfect field of characteristic $p>0$.  Ekedahl proved an
equivalence of triangulated categories between bounded coherent
complexes over the Raynaud ring and complexes of $F$-gauges.  We give a
geometric proof, independent of his proof, that the de Rham--Witt
realization is a symmetric monoidal equivalence of stable
$\infty$-categories between perfect $F$-gauges and bounded coherent
complexes over the Raynaud ring with Ekedahl's completed star product,
as expected by Illusie.  The geometry comes from topological
restriction homology: $\mathrm{TR}^r(k;\mathbf Z_p)^{hS^1}$ is an
iterated fibre product of copies of
$\mathrm{TC}^-(k;\mathbf Z_p)$, and taking even stacks before and after
the fibre product gives the genuine and the naive $r$-Nygaard stacks
$k^{\NN_r}$ and $k^{\Nyg_r}$.  Pushforward from $k^{\NN_r}$ to
$k^{\Nyg_r}$ realizes Andriopoulos' $r$-Nygaard
filtrations, which interpolate between the Nygaard filtration at $r=1$
and the Hodge--Witt filtration at $r=\infty$.  On $k^{\NN_r}$, which
is a chain of Nygaard charts, restriction and Frobenius come from the
inclusions of sub-chains,
Verschiebung from duality, and the differential as a Bockstein
operator, and the relations of the star product become projection
formulas.  As applications, we deduce Ekedahl's K\"unneth formula and
Poincar\'e duality for de Rham--Witt cohomology from their counterparts
for $F$-gauge cohomology, and show that the special-value formulas of
Illusie and of Mondal over finite fields are equivalent.
\end{abstract}

\maketitle

\setcounter{tocdepth}{2}
\tableofcontents

\section{Introduction}
\label{sec:intro}

Let $k$ be a perfect field of characteristic $p>0$, and let
$W=W(k)$, with Frobenius $\sigma$.  For a smooth proper $k$-scheme
$X$, the operators $F$, $V$ and $d$ on the de Rham--Witt complex make
$R\Gamma(X,W\Omega^\bullet_{X/k})$ a bounded complex of modules over
the Raynaud ring $R$ with coherent cohomology
\cite{Illusie79,IllusieRaynaud83} (Definition~\ref{def:raynaud-ring}).
In the 1980s Ekedahl introduced a tensor product of $R$-modules, the
star product $M\star N$, for which de Rham--Witt cohomology satisfies a K\"unneth formula
\cite{Ekedahl85}, and a dualizing functor, for which it satisfies
Poincar\'e duality \cite{Ekedahl84}.  He also proved an equivalence
\[
 D_c^b(R)\simeq D_c^b(F\text{-gauges})
\]
between bounded coherent complexes of $R$-modules and of $F$-gauges
\cite[Theorem~II.5.3]{Ekedahl86Diagonal}.
Ekedahl notes that Fontaine had introduced $F$-gauges independently
\cite[Introduction]{Ekedahl86Diagonal}, and Fontaine and Jannsen
developed them in a preprint of 2013 \cite{FontaineJannsen}.

In recent years $F$-gauges have acquired a geometric meaning.  Drinfeld
and Bhatt--Lurie introduced the prismatization
\cite{Drinfeld,BhattLurie}.  Drinfeld and Bhatt studied $F$-gauges as
quasi-coherent complexes on the syntomic point $k^{\syn}$, on whose
loci crystalline, de Rham and Hodge cohomology arise as pullbacks
\cite{Drinfeld,BhattFGauges}.  The left-hand side of Ekedahl's
equivalence, the de Rham--Witt side, has so far been missing from this
picture: neither the operators $F$, $V$ and $d$ of the de Rham--Witt
complex nor Ekedahl's multiplicative structures, which are defined by
explicit formulas, have been interpreted on $k^{\syn}$.

Ekedahl's equivalence is proved in triangulated categories, by
constructing an adjoint and checking the adjunction maps.  In recent
notes \cite[\S6]{IllusieZeta}, Illusie sketches how it compares the
de Rham--Witt and the $F$-gauge approaches to special values of zeta
functions over finite fields.  For this he needs an enhancement to
stable $\infty$-categories and compatibility with the symmetric
monoidal structures and with Breuil--Kisin twists on both sides, which
he describes as expected, and he remarks that ``it would be nice to
find a simpler, more conceptual proof''.  We construct a
\emph{de Rham--Witt realization}
\[
 T_{\dRW}:\Dqc(k^{\syn})\longrightarrow\widehat D(R),
\]
which sends the $F$-gauge cohomology $\pi_{X,*}\cO_{X^{\syn}}$ of a
smooth proper $k$-scheme $X$ to $R\Gamma(X,W\Omega^\bullet_{X/k})$, with
its operators $F$, $V$ and $d$
(Corollary~\ref{cor:geometric-raynaud-realization}), and prove that it
is a symmetric monoidal equivalence, as Illusie expects.

\begin{thmletter}[Ekedahl's derived equivalence;
Theorems~\ref{thm:presentable-equivalence}
and~\ref{thm:symmetric-monoidal-equivalence},
Corollary~\ref{thm:gauge-raynaud-equivalence}]
\label{thm:intro-equivalence}
The de Rham--Witt realization is a symmetric monoidal equivalence of
stable $\infty$-categories
\[
 T_{\dRW}:\bigl(\Dqc(k^{\syn}),\otimes\bigr)
 \xrightarrow{\ \sim\ }
 \bigl(\widehat D(R),\widehat\star^{L}\bigr),
\]
which restricts to an equivalence
$\Perf(k^{\syn})\simeq D_c^b(R)$.  On $D_c^b(R)$ its inverse is given
by Ekedahl's functor $S$.
\end{thmletter}

Here $D_c^b(R)$ consists of the bounded complexes of graded
$R$-modules with coherent cohomology
(Definition~\ref{def:coherent-complexes}), and
$\widehat D(R)$ of the complexes $M$ that are complete,
that is, $M\simeq R\varprojlim_r(R_r\otimes_R^{L}M)$ for the finite
reductions $R_r=R/(V^rR+dV^rR)$.  The completed star product
$\widehat\star^{L}$ is Ekedahl's (Definition~\ref{def:derived-star}).
The equivalence also identifies Breuil--Kisin twists with the twists
$M\{a\}=M(a)[-a]$ of $M\in D(R)$
(Corollary~\ref{lem:Koszul-BK-shift}).  The proof is geometric and
does not use Ekedahl's proof of his equivalence.  From his work we use
his functor $S$ from complexes of $R$-modules to $F$-gauges
(Construction~\ref{con:ladder}) and his results on $R$-modules, notably
on the star product and on his dualizing functor
\cite{Ekedahl84,Ekedahl85}.  We use the derived star product in the
form constructed by Li and Yang \cite[\S3.2]{LiYang}.

\medskip
\noindent\textbf{Geometry of $r$-Nygaard filtrations.}
A gauge on the Nygaard point $k^{\Nyg}$ is a filtered object whose
transition maps $u$ and $t$ satisfy $ut=p$.  To obtain $F$, $V$ and
$d$ at Witt length $r$, we were guided
by Andriopoulos' description of the $r$-Nygaard filtration through
$\mathrm{TR}$ \cite{AndriopoulosTR}, by the even filtration of Hahn,
Raksit and Wilson \cite{HRW}, and by calculations of $\mathrm{TR}^r$
shown to us by Noah Riggenbach.  These sources suggest the
\emph{naive $r$-Nygaard stack}
\[
 k^{\Nyg_r}=\Spf\bigl((W[u,t]/(ut-p^r))^\wedge_p\bigr)/\Gm,
 \qquad |u|=1,\quad |t|=-1,
\]
whose coordinate ring has an $A_{r-1}$-singularity.  We resolve it by
the \emph{genuine $r$-Nygaard stack} $k^{\NN_r}$, a chain of $r$
Nygaard charts glued along Frobenius.  The contraction
$g_r:k^{\NN_r}\to k^{\Nyg_r}$ comes from the minimal resolution of the
singularity and is crepant (Proposition~\ref{prop:contraction-crepant}).
Both stacks also have a conceptual description through the even
stacks $\mathrm{Spev}$ of \cite{DevalapurkarThesis}, in their
$p$-complete variant.  The spectrum
$\mathrm{TR}^r(k;\mathbf Z_p)^{hS^1}$ is an iterated fibre product of $r$
copies of $\mathrm{TC}^-(k;\mathbf Z_p)$ over $\mathrm{TP}(k;\mathbf Z_p)$
\cite{AndriopoulosTR}, and taking even stacks before and after the
fibre product gives (Remark~\ref{rem:TR-motivation})
\[
 \begin{aligned}
 k^{\NN_r}&\simeq\operatorname{coeq}\Bigl(\coprod_{i=0}^{r-2}\mathrm{Spev}\bigl(\mathrm{TP}(k;\mathbf Z_p)\bigr)
 \rightrightarrows\coprod_{i=0}^{r-1}\mathrm{Spev}\bigl(\mathrm{TC}^-(k;\mathbf Z_p)\bigr)\Bigr)\\
 &\xrightarrow{\ g_r\ }\mathrm{Spev}\bigl(\mathrm{TR}^r(k;\mathbf Z_p)^{hS^1}\bigr)\simeq k^{\Nyg_r}.
 \end{aligned}
\]
Together with the map $f_r$ to the syntomic point, $g_r$ gives a
correspondence
\[
\begin{tikzcd}[column sep=small]
 & k^{\NN_r} \arrow[dl, "f_r"'] \arrow[dr, "g_r"] & \\
 k^{\syn} & & k^{\Nyg_r}
\end{tikzcd}
\]
For a smooth $k$-scheme $X$ with $F$-gauge cohomology
$\cH_{\syn}(X)=\pi_{X,*}\cO_{X^{\syn}}$, we call
$\cH_{\Nyg_r}(X):=g_{r,*}f_r^*\cH_{\syn}(X)$ its \emph{$r$-Nygaard
realization}.  Its restrictions to the loci of $k^{\Nyg_r}$ recover the
classical cohomology theories of $X$, now at Witt length $r$.

\begin{thmletter}[Specializations of the $r$-Nygaard realization;
Theorem~\ref{thm:r-nygaard-specializations}]
\label{thm:intro-realization}
Let $X$ be a smooth quasi-compact $k$-scheme and $r\geq1$.  Up to
Frobenius twists of the coefficients:
\begin{enumerate}
\item $\cH_{\Nyg_r}(X)$ is the Rees object of the $r$-Nygaard filtration
 on $R\Gamma_{\crys}(X/W)$;
\item on the de Rham locus $t\neq0$ it is $R\Gamma_{\crys}(X/W)$, and on
 the Hodge--Tate locus $u\neq0$ it is $R\Gamma_\Delta(X/W)$;
\item on the divisor $u=0$ it is the Hodge filtration on
 $R\Gamma(X,W_r\Omega^\bullet_{X/k})$;
\item on the divisor $t=0$ it is the conjugate filtration on
 $R\Gamma_{\crys}(X/W_r)$, with graded pieces
 $R\Gamma(X,W_r\Omega^i_{X/k})[-i]$;
\item at the origin $u=t=0$, its weight-$i$ piece is
 $R\Gamma(X,W_r\Omega^i_{X/k})[-i]$.
\end{enumerate}
\end{thmletter}

For $r=1$ these are Bhatt's comparisons for the Nygaard point
\cite[Theorem~3.3.5]{BhattFGauges}.  For general $r$, the $r$-Nygaard
filtration in (1) and its graded pieces in (4) are due to Andriopoulos
\cite[Theorem~1.8 and Corollaries~5.14--5.15]{AndriopoulosTR}, and his
characterization of the filtration enters our proof.  The point of
Theorem~\ref{thm:intro-realization} is that all of these theories are
restrictions of a single object, built geometrically from the
correspondence.

The chain of length $r$ sits in the chain of length $r+1$ as the right
and as the left sub-chain.  These two inclusions, with the coefficient
Frobenius, are the restriction and Frobenius maps $\mathsf R_r$ and
$\mathsf F_r$, and together with their analogues on the naive $r$-Nygaard stacks
$k^{\Nyg_r}$ they are compatible with the correspondences
(Proposition~\ref{prop:level-maps-compatible}):
\[
\begin{tikzcd}[column sep=large, row sep=large]
 k^{\syn} \arrow[d, "\Phi_{\syn}"']
 & k^{\NN_r} \arrow[l, "f_r"'] \arrow[r, "g_r"]
   \arrow[d, shift right=0.6ex, "\mathsf R_r"']
   \arrow[d, shift left=0.6ex, "\mathsf F_r"]
 & k^{\Nyg_r}
   \arrow[d, shift right=0.6ex, "\mathsf R_r"']
   \arrow[d, shift left=0.6ex, "\mathsf F_r"]\\
 k^{\syn}
 & k^{\NN_{r+1}} \arrow[l, "f_{r+1}"] \arrow[r, "g_{r+1}"']
 & k^{\Nyg_{r+1}}.
\end{tikzcd}
\]
They induce the restriction and Frobenius maps between the $r$-Nygaard
filtrations at adjacent levels.

At $r=\infty$, the limit of the
$r$-Nygaard filtrations along restriction is the $\infty$-Nygaard
filtration of Andriopoulos \cite[\S5.1]{AndriopoulosTR}, and it is the
Hodge--Witt filtration $R\Gamma(X,W\Omega^{\geq\bullet}_{X/k})$.  So the
$r$-Nygaard filtrations interpolate between the Nygaard filtration at
$r=1$ and the Hodge--Witt filtration at $r=\infty$.  Geometrically, along the restriction maps $\mathsf R_r$ the naive
$r$-Nygaard stacks $k^{\Nyg_r}$ have colimit $k^{\Nyg_\infty}\simeq\mathbb A^1_{W,t}/\Gm$,
the genuine ones $k^{\NN_r}$ have as colimit a chain $k^{\NN_\infty}$ that is
infinite in one direction, and pushforward along the colimit
correspondence
\[
\begin{tikzcd}[column sep=small]
 & k^{\NN_\infty} \arrow[dl, "f_\infty"'] \arrow[dr, "g_\infty"] & \\
 k^{\syn} & & \mathbb A^1_{W,t}/\Gm
\end{tikzcd}
\]
recovers the Hodge--Witt filtration (Theorem~\ref{thm:hodge-witt}).

Item (3) exhibits the de Rham--Witt complex of level $r$ as a
filtered object, whose differential is recovered from the boundary maps
of the filtration.  To do this for an arbitrary $F$-gauge $E$, we use the
\emph{$r$-Hodge stack}
\[
 k^{\mathrm{Hodge}_r}:=k^{\NN_r}\times^{R}_{k^{\Nyg_r}}B\mathbb G_{m,W_r},
\]
the derived preimage under $g_r$ of the origin $u=t=0$, with its maps
$e_r$ to $k^{\syn}$ and $g'_r$ to $B\mathbb G_{m,W_r}$.  The \emph{finite
de Rham--Witt realization} is
\[
 T_{\dRW_r}(E)\simeq(\sigma^{1-r})^*\bigl((g'_r)_*e_r^*E\bigr)[\bullet]
 \in D(\mathrm{GrMod}_{W_r[d]/d^2})
\]
(Definition~\ref{def:finite-extraction} and Lemma~\ref{lem:T-r-Hodge}),
where $[\bullet]$ is the regrading and $d$ comes from the boundary maps of
the Hodge filtration on the divisor $u=0$.  The map $e_r$ is finite and
oriented, $e_r^!\simeq e_r^*[-2]$ (Lemma~\ref{lem:e-finite-oriented}), so
$T_{\dRW_r}$ is corepresented by self-dual perfect $F$-gauges, twists of
$e_{r,*}\cO$ (Proposition~\ref{thm:oriented-finite-extraction}).
Restriction and Frobenius are pullbacks along $\mathsf R_r$ and
$\mathsf F_r$, and Verschiebung, which goes from level $r$ to level
$r+1$, is the transpose of Frobenius for this duality
(Construction~\ref{def:intrinsic-geometric-V}).  On a chart, Frobenius
multiplies the Koszul generator of the equation $t$ by $p$, Verschiebung
multiplies the unit by $p$, and $d$ sends the generator to the unit.
All relations of the Raynaud ring, such as $FV=VF=p$ and $FdV=d$, follow
from these three formulas (Remark~\ref{rem:local-raynaud-relations}).
Taking the limit over $r$ gives the de Rham--Witt realization
$T_{\dRW}$ of Theorem~\ref{thm:intro-equivalence}, with its $R$-action
(\S\ref{sec:representing-objects}).

\medskip
\noindent\textbf{Where the star product comes from.}
For Dieudonn\'e modules, the box product
\cite{Goerss,HopkinsLurie,AntieauNikolaus} is the tensor product
subject to the relations
\[
 F(x\boxtimes y)=Fx\boxtimes Fy,\qquad
 Vx\boxtimes y=V(x\boxtimes Fy),\qquad
 x\boxtimes Vy=V(Fx\boxtimes y),
\]
and Ekedahl's star product (Definition~\ref{def:star-product}) adds
the differential $d$, subject to the Leibniz rule.  In Ekedahl's work
these relations are part of the definition.  In our approach each of
them is a basic property of a geometric operation.  Frobenius is a
pullback, hence multiplicative.  Verschiebung is the trace along the Frobenius map
$\mathsf F_r:k^{\mathrm{Hodge}_r}\to k^{\mathrm{Hodge}_{r+1}}$ of $r$-Hodge stacks, which is finite, and the two
relations for $V$ are its projection formula.  The differential is the boundary map of a Koszul complex,
hence a derivation (Lemma~\ref{lem:V-cup-projection}).  So the cup
products on the finite realizations $T_{\dRW_r}$ satisfy the defining
relations of the star product, and its universal property gives the
comparison map of Theorem~\ref{thm:intro-equivalence}.

This also resolves an apparent tension between the two
sides.\footnote{We thank Vadim Vologodsky for discussions of this
question.}  A Dieudonn\'e module, viewed as an $R$-module in internal
degree zero, corresponds to an $F$-gauge with Hodge--Tate weights in
$[0,1]$.  The tensor product of two such $F$-gauges has Hodge--Tate
weights in $[0,2]$, whereas the box product of two Dieudonn\'e modules
is again a Dieudonn\'e module.  The box product therefore cannot be the tensor product of
$F$-gauges, while the star product is, by
Theorem~\ref{thm:intro-equivalence}.  For two Dieudonn\'e modules the
star product agrees with the box product in internal degree zero
(Proposition~\ref{prop:star-box-degree-zero}), but it can also be
nonzero in internal degree one (Example~\ref{ex:star-box-dV}), and this
is where Hodge--Tate weight $2$ appears.  So there are not two competing
tensor products, and there is no need for a separate stack whose
quasi-coherent sheaves are complexes of $R$-modules.  The syntomic point
$k^{\syn}$ is already such a stack, and the $r$-Nygaard stacks
$k^{\NN_r}$ and $k^{\Nyg_r}$ are correspondences that read off $F$, $V$ and $d$.

\medskip
\noindent\textbf{Why there is no K\"unneth formula at finite level.}
For full Witt vectors, tensor products become box
products: the $V$-completed box product satisfies
$W(A)\widehat\boxtimes W(B)\cong W(A\otimes B)$ for rings $A$
and $B$ \cite[Theorem~4.16]{AntieauNikolaus}, and Ekedahl's star
product extends this to de Rham--Witt complexes.  At finite level
$r\geq2$ there is no such statement.  Already
$W_2(k[x])\otimes_{W_2(k)}W_2(k[y])\to W_2(k[x,y])$ is not surjective,
and Ekedahl observed that reduction to level $r$ is monoidal only for
$r=1$ \cite[Chapter~I]{Ekedahl85}.  Correspondingly, for $r\geq2$ the
finite realization $T_{\dRW_r}$, as a functor to graded
$W_r[d]/d^2$-complexes with the graded tensor product over $W_r$, is
lax but not strong monoidal (Remark~\ref{rem:finite-level-not-monoidal}).
Our construction explains this.  The functor $T_{\dRW_r}$ is a pullback
followed by a pushforward through the $r$-Hodge stack:
\[
 \begin{tikzcd}[column sep=small, ampersand replacement=\&]
  \& k^{\mathrm{Hodge}_r} \arrow[dl, "e_r"'] \arrow[dr, "g'_r"] \& \\
  k^{\syn} \& \& B\mathbb G_{m,W_r}.
 \end{tikzcd}
\]
The pullback $e_r^*$ is symmetric monoidal, but the pushforward
$(g'_r)_*$ is only lax symmetric monoidal, so for $r\geq2$ the cup
products $T_{\dRW_r}(E)\otimes T_{\dRW_r}(E')\to T_{\dRW_r}(E\otimes E')$
are in general not equivalences.  For $r=1$ the map $g'_1$ is the
identity of $B\mathbb G_{m,k}$, so $T_{\dRW_1}$ is just the Hodge
pullback and is symmetric monoidal.  This is the K\"unneth formula for
Hodge cohomology.  At $r=\infty$, the de Rham--Witt realization $T_{\dRW}$
is symmetric monoidal again, now for
the completed star product $\widehat\star^{L}$
(Theorem~\ref{thm:intro-equivalence}).  So the
realizations are monoidal exactly at the two ends, $r=1$ and
$r=\infty$, and the proof of Theorem~\ref{thm:intro-equivalence} uses
both: the comparison
map is an equivalence because its reduction to level one is, by
derived Nakayama.

\medskip
\noindent\textbf{K\"unneth formula and Poincar\'e duality.}
On the side of $F$-gauges, the K\"unneth formula for $F$-gauge
cohomology can be checked at the Hodge point, and Poincar\'e duality
for $F$-gauge cohomology holds by \cite[Theorem~4.5.3]{BhattFGauges}.
Theorem~\ref{thm:intro-equivalence} carries them to de Rham--Witt
cohomology.

\begin{thmletter}[K\"unneth formula and Poincar\'e duality;
Corollary~\ref{cor:geometric-kunneth}, Theorem~\ref{cor:duality-comparison}
and Corollary~\ref{thm:geometric-Ekedahl-duality}]
\label{thm:intro-duality}
Let $X$ and $Y$ be smooth proper over $k$.  The K\"unneth equivalence
for $F$-gauge cohomology induces an equivalence
\[
 R\Gamma(X,W\Omega^\bullet_{X/k})\,\widehat\star^{L}\,
 R\Gamma(Y,W\Omega^\bullet_{Y/k})
 \simeq R\Gamma(X\times_kY,W\Omega^\bullet_{X\times_kY/k}).
\]
The equivalence of Theorem~\ref{thm:intro-equivalence} is compatible
with duality: $T_{\dRW}(E^\vee)\simeq\Dual_R(T_{\dRW}(E))$ for perfect
$E$.  In particular, if $X$ has pure dimension $N$, Poincar\'e duality
for $F$-gauge cohomology induces an equivalence
\[
 \Dual_R\bigl(R\Gamma(X,W\Omega^\bullet_{X/k})\bigr)(-N)[-N]
 \simeq R\Gamma(X,W\Omega^\bullet_{X/k}),
\]
compatible with the finite dualities for
$R\Gamma(X,W_r\Omega^\bullet_{X/k})$, which exchanges Frobenius and
Verschiebung under transpose.
\end{thmletter}

Here $\Dual_R$ is Ekedahl's dualizing functor \cite[\S III.3]{Ekedahl84}.
The K\"unneth formula and Poincar\'e duality are as in Ekedahl's theorems
\cite[Theorem~III.2.9]{Ekedahl84}, \cite[Theorem~II.1.1]{Ekedahl85},
and here they follow from their counterparts for $F$-gauges.  Li and
Yang extend the K\"unneth formula to smooth quasi-compact varieties in
the derived setting \cite[Theorem~3.36]{LiYang}.  The two
shifts in the duality come from the Breuil--Kisin twist and from the
regrading in $T_{\dRW_r}$.

\medskip
\noindent\textbf{Special values of zeta functions.}
Over a finite field $k=\mathbf F_q$, special values of zeta functions
have been studied $p$-adically in two settings.  Milne and
Ramachandran showed that complexes of $R$-modules are a natural home
for them.  An object $M$ of $D_c^b(R)$ has a zeta function, Hodge
numbers and an absolute $p$-adic cohomology $\RHom_R(W,M\{n\})$, and
they proved a special-value formula for all such objects, assuming that
$q^n$ is at most a simple root of the minimal polynomial of Frobenius
\cite[Theorem~0.1]{MilneRamachandran}.  Mondal
works with $F$-gauges and proves an unconditional formula, using a
stable Bockstein characteristic \cite{MondalZeta}.  In his survey of
Mondal's work, Illusie proves the formula for every $M\in D_c^b(R)$
without this assumption \cite[Theorem~5]{IllusieZeta}, and
observes that a symmetric monoidal $\infty$-categorical version of
Ekedahl's equivalence would identify the invariants on the two sides
\cite[\S6.4]{IllusieZeta}.  Theorem~\ref{thm:intro-equivalence}
provides it.

\begin{thmletter}[Comparison of zeta-value invariants;
Theorem~\ref{thm:zeta-comparison}]
\label{thm:intro-zeta}
Let $M\in D_c^b(R)$, and let $E\in\Perf(k^{\syn})$
correspond to $M$ under Theorem~\ref{thm:intro-equivalence}.  Then $M$
and $E$ have the same zeta function, the same syntomic cohomology
$R\Gamma(-,\mathbf Z_p(n))$ over $k$ and over $\bar k$, the same
stable Bockstein characteristics and the same Hodge numbers.  Hence
the special-value formula of \cite[Theorem~5]{IllusieZeta} for $M$ is
equivalent to that of \cite[Theorem~6.5]{MondalZeta} for $E$.
\end{thmletter}

Each formula therefore implies the other, so
Theorem~\ref{thm:intro-zeta} gives a new proof of each.  For instance,
Mondal's formula also follows from Illusie's, whose proof uses
Ekedahl's d\'evissage of $D_c^b(R)$.

\medskip
\noindent\textbf{Related work.}
This project began as an attempt to understand Ekedahl's papers on
de Rham--Witt cohomology \cite{Ekedahl84,Ekedahl85,Ekedahl86Diagonal}
from a modern point of view, and much of it is a geometric reading of
his constructions.  The star product and the
dualizing functor are his, the functor $S$ of
Construction~\ref{con:ladder} is his complex of $F$-gauges, and
Theorem~\ref{thm:intro-equivalence} recovers his equivalence.
The framework of this paper is that of Bhatt's notes on prismatic
$F$-gauges \cite{BhattFGauges}, which develop $F$-gauges on the syntomic
point, recover crystalline, de Rham and Hodge cohomology from the loci of
the Nygaard point, and formulate Poincar\'e duality for $F$-gauge
cohomology.  Theorem~\ref{thm:intro-realization} extends his comparisons for the
Nygaard point to Witt length $r$, and the de Rham--Witt realization
$T_{\dRW}$ is named after his realizations $T_{\mathrm{dR}}$ and
$T_{\mathrm{crys}}$.  The other main input is Andriopoulos' $r$-Nygaard
filtration on prismatic cohomology \cite{AndriopoulosTR}, obtained by
gluing $r$ copies of the Nygaard filtration along Frobenius.  He
characterized it algebraically and related its filtered and graded
pieces to the motivic filtrations on $(\mathrm{TR}^r)^{hS^1}$ and
$\mathrm{TR}^r$.  The
genuine $r$-Nygaard stack $k^{\NN_r}$ makes this gluing geometric, and
pushforward to $k^{\Nyg_r}$ recovers his filtration
(Theorem~\ref{thm:intro-realization}).  Lahoti and Manam study the moduli stack
$\mathbf Z_p^{\NN}$ of polyfiltered Cartier--Witt divisors, an
$\mathbf N$-indexed chain of copies of the Nygaard stack of
$\mathbf Z_p$ glued along the Hodge--Tate and de Rham maps, and they
mention the present work as a variant in characteristic $p$
\cite[\S1.3]{LahotiManam}.  Yang follows Ekedahl's approach and makes the
equivalence explicit, using bounded projective resolutions at each
finite level \cite{YangCanonical}.  Our earlier papers used Ekedahl's equivalence
at the level of triangulated categories, for explicit computations in
de Rham--Witt cohomology with applications to $p$-primary Brauer groups
\cite{ZhangDominoes} and for the study of
torsion $F$-gauges \cite{ZhangHN}.  The present paper studies the
equivalence itself.

\medskip
\noindent\textbf{Acknowledgements.} This project began during a visit to
Ben Antieau at the Institute for Advanced Study in the spring of 2024.
The author thanks his advisor, Ben Antieau, for his guidance, and for
discussing this project with him over several years.  The author
learned d\'ecalage, Koszul duality and the box product from
conversations with him, and these ideas run through the paper.  The author is
especially
grateful to Noah Riggenbach for showing him calculations around
$\mathrm{TR}^r$, which inspired the constructions of $k^{\Nyg_r}$ and
$k^{\NN_r}$.  A first version of the construction, the stack now called
the $\infty$-Hodge stack $k^{\mathrm{Hodge}_\infty}$
(Construction~\ref{con:infinite-nygaard-stacks}), was presented under
the name $k^{\mathrm{TR}}$ at Antieau's group meeting in January 2025.  The notation
$k^{\NN_r}$ was inspired by the notation $\mathbf Z_p^{\NN}$ of Lahoti
and Manam \cite{LahotiManam}.  The author thanks Chris Dodd, Elden Elmanto, Daigo Ito, Shizhang Li, Kirill Magidson, Deven Manam, Josh
Mundinger, Sasha Petrov, Arpon Raksit, Vladimir Shein, Longke Tang,
Vadim Vologodsky and Yuan Yang for helpful conversations.  This material is based
upon work supported by the National Science Foundation under Grant
No.~DMS-1928930 while the author was in residence at the Simons Laufer
Mathematical Sciences Institute in Berkeley, California, during the Fall
2026 semester, for the program \emph{Motivic Homotopy Theory: Connections
and Applications}.  The author is also supported by the Simons
Collaboration on Perfection in Algebra, Geometry, and Topology.

\medskip
\noindent\textbf{AI disclosure.} The author used AI tools, namely
GPT-6 Astra and Claude Opus 5.5, to explore the compatibility with
duality in \S\ref{sec:geometric-ekedahl-duality}
and the comparison of special values in \S\ref{sec:zeta}.  Apart from
these, the mathematical ideas of this paper are due to the author.  AI
assistants were also used to review arguments, to help write up proofs
and expository passages from the author's ideas, and to proofread and
edit the manuscript.  All content they produced has been reviewed,
edited, and checked for accuracy by the author, who is responsible for
the paper as it stands.

\medskip
\phantomsection\label{sec:conventions}
\noindent\textbf{Notation and conventions.}
Throughout, $k$ is a perfect field of characteristic $p>0$,
$W=W(k)$, $W_r=W/p^r$, and $\sigma$ is the Frobenius of $W$.  We write
$R\Gamma_{\crys}(X/W_r):=R\Gamma_{\crys}(X/W)/^{L}p^r$, and
$R\Gamma_{\crys}$ for $R\Gamma_{\crys}(X/W)$ inside formulas.
Stacks are classical $p$-adic formal fpqc stacks, except for the
derived $r$-Hodge stacks $k^{\mathrm{Hodge}_r}$ of \S\ref{sec:hodge-stack}, which are locally
$\Gm$-quotients of spectra of $1$-truncated animated rings.
Categories are stable $\infty$-categories with cohomological indexing,
all pullbacks and pushforwards are derived, and $\RHom$ is the derived
mapping complex.  An object is \emph{static} if its cohomology sits in
degree zero, in each internal degree if it is graded.  For a morphism
$f$ of affine formal schemes or of their $\Gm$-quotients, $f^*$ also
denotes the map of graded coordinate rings, and we describe $f$ by
$f^*$ on coordinates.

For a graded ring $\Lambda$, the objects of $D(\mathrm{GrMod}_\Lambda)$
are \emph{graded complexes}, and $W_r[d]/d^2$ is graded with $d$ in
degree one.  For the Raynaud ring $R$ (Definition~\ref{def:raynaud-ring}) we
write $D(R):=D(\mathrm{GrMod}_R)$, and $\widehat D(R)\subset D(R)$ for
the complete objects (Definition~\ref{def:raynaud-completion}).  We write $E$ for gauges and $F$-gauges and $M$ for graded
complexes, such as complexes of $R$-modules.  Both gradings are written
as superscripts: $E^n$ is the weight-$n$ component of a gauge $E$, and
$M^n$ is the internal-degree-$n$ component of $M$.  Each such
component is itself a complex, and for a gauge or graded complex $N$ we
write $N^{n,q}$ for the term of $N^n$ in cohomological degree $q$.  So
the first index is always the weight or internal degree and the second
the cohomological degree, as in $h^{p,q}=\dim H^q(\Omega^p)$.  We write $|x|$ for the weight of a
homogeneous element $x$.  Internal shifts, cohomological shifts and
Breuil--Kisin twists are
\[
 M(a)^{n,q}=M^{n+a,q},\qquad M[a]^{n,q}=M^{n,q+a},\qquad M\{a\}:=M(a)[-a],
\]
so $R(-n)$ is free on a generator of internal degree $n$.  On
complexes of $R$-modules the natural shifts are $(a)$ and $[b]$, and
the Breuil--Kisin twist $\{a\}=(a)[-a]$ is diagonal.  On $F$-gauges the
natural shifts are $\{a\}$ and $[b]$, and the internal shift $(a)$ of
$R$-modules corresponds to the diagonal $\{a\}[a]$, since regrading
turns the weight shift of a gauge into $(a)[-a]$
(Convention~\ref{conv:shearing}).  On
$R\Gamma(X,W\Omega^\bullet)$ the internal degree is the form degree, and
the simple functor $\mathbf s$ of Definition~\ref{def:simple-functor}
combines it with the complex degree.  For $f\in\Lambda$
homogeneous of weight $w$ we put $M/^{L}f:=\cofib(f:M(-w)\to M)$, and
iterate for sequences.  The Koszul model
$\Ksz_\Lambda(f_1,\ldots,f_m)$ of $\Lambda/^{L}(f_1,\ldots,f_m)$ sits in
degrees $[-m,0]$, with a generator $\varepsilon$ of degree $-1$ and
weight $|f_j|$ for each $f_j$.

For a $W$-complex $P$, the Frobenius twist
$(\sigma^a)^*P:=W\otimes_{W,\sigma^a}P$ is $P$ with $b\in W$ acting by
$\sigma^{-a}(b)$.  A $\sigma$-semilinear $F$ and a
$\sigma^{-1}$-semilinear $V$ become $W$-linear maps
$F:(\sigma^a)^*P\to(\sigma^{a-1})^*P$ and
$V:(\sigma^{a-1})^*P\to(\sigma^a)^*P$.  Restriction and extension of
scalars along the Frobenius are $\sigma_*:=(\sigma^{-1})^*$ and
$\sigma^*$, mutually inverse since $k$ is perfect, and a
$\sigma$-semilinear map $M\to N$ is a $W$-linear map $M\to\sigma_*N$.  We
use $\sigma_*$ only for the sources and targets of such maps.
Dieudonn\'e modules are always $V$-complete, which for finite length
means that $V$ is nilpotent.

\section{The geometry of \texorpdfstring{$r$}{r}-Nygaard filtrations}
\label{sec:expanded-construction}

This section constructs the geometry behind the $r$-Nygaard
filtration.  In \S\ref{sec:syntomic-point} we review Bhatt's filtered
prismatization $X^{\Nyg}$ and syntomification $X^{\syn}$ of a smooth
$k$-scheme $X$.  In \S\ref{sec:genuine-naive} we recall Andriopoulos'
$r$-Nygaard filtration and construct two stacks, the genuine
$r$-Nygaard stack $k^{\NN_r}$ and the naive $r$-Nygaard stack
$k^{\Nyg_r}$.  The main construction is a correspondence
\[
\begin{tikzcd}[column sep=small]
 & k^{\NN_r} \arrow[dl, "f_r"'] \arrow[dr, "g_r"] & \\
 k^{\syn} & & k^{\Nyg_r}
\end{tikzcd}
\]
for each level $r\geq1$, together with restriction and Frobenius maps
from level $r$ to level $r+1$ that are compatible with it.  In
\S\ref{sec:pullback-pushforward} we study the geometry of this
correspondence and compare the $r$-Nygaard realization
$\cH_{\Nyg_r}(X):=g_{r,*}f_r^*\cH_{\syn}(X)$ with cohomology theories
of $X$.  It is the Rees object of the $r$-Nygaard filtration
(Proposition~\ref{prop:r-nygaard-comparison}), and its restrictions
to the open and closed loci of $k^{\Nyg_r}$ recover cohomology
theories of $X$ such as crystalline cohomology and the de Rham--Witt
cohomology of level $r$ with its Hodge and conjugate filtrations
(Theorem~\ref{thm:r-nygaard-specializations}).  The associated cochain
complex of the Hodge filtration then recovers the de Rham--Witt complex
of level $r$ with its differential, and applied to an arbitrary $F$-gauge it defines the
finite de Rham--Witt realization $T_{\dRW_r}$
(Definition~\ref{def:finite-extraction}).

\subsection{Nygaard stack and syntomic stack}
\label{sec:syntomic-point}

This subsection recalls Bhatt's construction of the filtered
prismatization $X^{\Nyg}$ and the syntomification $X^{\syn}$ of a
smooth $k$-scheme $X$ \cite[\S\S3.3--3.4 and 4.1]{BhattFGauges}.

Filtrations enter through the Rees construction.  A decreasing
filtration $\Fil^\bullet M$ has Rees object
$\Rees(\Fil^\bullet M):=\bigoplus_n\Fil^nM\,t^{-n}$ with $|t|=-1$, on
which $t$ acts by the transition maps.  This identifies filtered
complexes with $\Dqc(\mathbb A^1_t/\Gm)$, and complete ones with its
derived $t$-complete objects \cite[Proposition~2.2.6]{BhattFGauges}.
Increasing filtrations correspond likewise to $\Dqc(\mathbb A^1_u/\Gm)$
with $|u|=1$.  We use these identifications without comment.

Write $\Fil_p^nW=p^{\max(n,0)}W$ for the $p$-adic filtration on $W$.
Its Rees algebra is the graded algebra
\[
 \Rees(\Fil_p^\bullet W)
 =\bigoplus_{n\in\mathbb Z}p^{\max(n,0)}W\,t^{-n}
 \simeq W[u,t]/(ut-p),\qquad u=pt^{-1}.
\]
Here $|u|=1$ and $|t|=-1$.

\begin{definition}[The Nygaard point and gauges]
\label{def:nygaard-point}
\label{def:gauge-ends}
The \emph{Nygaard point} is the completed Rees stack
\[
 k^{\Nyg}:=
 \Spf\bigl((W[u,t]/(ut-p))^{\wedge}_p\bigr)/\Gm .
\]
A \emph{gauge} is an object of $\Dqc(k^{\Nyg})$, that is, a derived
$p$-complete graded module $E$ over the coordinate ring of $k^{\Nyg}$.
In terms of its weight components $E^n$, multiplication by $u$ and
$t$ gives a diagram
\[
\begin{tikzcd}[column sep=small, ampersand replacement=\&]
 \cdots \arrow[r, shift left=0.6ex, "t"] \& E^{n+1} \arrow[l, shift left=0.6ex, "u"] \arrow[r, shift left=0.6ex, "t"] \& E^{n} \arrow[l, shift left=0.6ex, "u"] \arrow[r, shift left=0.6ex, "t"] \& E^{n-1} \arrow[l, shift left=0.6ex, "u"] \arrow[r, shift left=0.6ex, "t"] \& \cdots \arrow[l, shift left=0.6ex, "u"]
\end{tikzcd}
\]
in which $u$ raises the weight, $t$ lowers it, and both composites
$ut$ and $tu$ are multiplication by $p$.  Put
\[
 E^{\infty}:=\bigl(\colim(E^n\xrightarrow{u}E^{n+1}\xrightarrow{u}\cdots)\bigr)^\wedge_p,
 \qquad
 E^{-\infty}:=\bigl(\colim(E^n\xrightarrow{t}E^{n-1}\xrightarrow{t}\cdots)\bigr)^\wedge_p,
\]
the derived $p$-completed colimits along $u$ and along $t$.  In
pictures we place $E^{\infty}$ at the left of this row and
$E^{-\infty}$ at the right.
\end{definition}

To pass from the point $k^{\Nyg}$ to a smooth $k$-scheme, we use
Bhatt's ring stack $\mathbb G_a^{\Nyg}$.

\begin{construction}[The ring stack $\mathbb G_a^{\Nyg}$]
\label{con:nygaard-ring-stack}
Let $L$ be the tautological line bundle on $k^{\Nyg}$, with sections
$u:\cO\to L$ and $t:L\to\cO$ whose composite $t\circ u$ is
multiplication by $p$.
Write $\mathbb W$ for the ring scheme of $p$-typical Witt vectors
and $F_{\mathbb W}$ for its Frobenius, both base changed to
$k^{\Nyg}$.  The divided-power envelope $\mathbb G_a^\sharp$ of the
origin in $\mathbb G_a$ identifies with $\ker(F_{\mathbb W})$, so
\[
 0\longrightarrow\mathbb G_a^\sharp\longrightarrow\mathbb W
 \xrightarrow{F_{\mathbb W}}F_{\mathbb W,*}\mathbb W
 \longrightarrow0
\]
is exact for the fpqc topology
\cite[Lemma~2.6.1 and Remark~2.6.2]{BhattFGauges}.
Here $F_{\mathbb W,*}\mathbb W$ is $\mathbb W$ with its module
structure restricted along $F_{\mathbb W}$.

Let $\mathbf V(L)$ be the vector bundle scheme of $L$, the affine
$k^{\Nyg}$-scheme with $\mathbf V(L)(S)=L(S)$ for a test scheme
$S\to k^{\Nyg}$, and let $\mathbf V(L)^\sharp$ be the divided-power
envelope of its zero section.  Thus $\mathbf V(\cO)=\mathbb G_a$, and the section
$u:\cO\to L$ induces a map $u^\sharp:\mathbb G_a^\sharp\to\mathbf V(L)^\sharp$.
Pushing out the preceding extension along $u^\sharp$ gives a group
scheme $M_u$.  Since $t^\sharp u^\sharp=p$, multiplication by $p$ on
$\mathbb W$ and the map $t^\sharp$ induce a map
$d_{u,t}:M_u\to\mathbb W$, and these fit into the diagram
\[
\begin{tikzcd}[column sep=huge, row sep=large]
0 \arrow[r] & \mathbb G_a^\sharp \arrow[r] \arrow[d,"u^\sharp"]
    \arrow[dd, bend right=40, dashed, "p=t^\sharp u^\sharp"', pos=0.2]
  & \mathbb W \arrow[r,"F_{\mathbb W}"] \arrow[d]
    \arrow[dd, bend right=40, dashed, "p"', pos=0.2]
  & F_{\mathbb W,*}\mathbb W \arrow[r] \arrow[d,equal]
    \arrow[dd, bend right=40, dashed, "p"', pos=0.2] & 0 \\
0 \arrow[r] & \mathbf V(L)^\sharp \arrow[r] \arrow[d,"t^\sharp"]
  & M_u \arrow[r] \arrow[d,"d_{u,t}"]
  & F_{\mathbb W,*}\mathbb W \arrow[r] \arrow[d,"p"] & 0 \\
0 \arrow[r] & \mathbb G_a^\sharp \arrow[r]
  & \mathbb W \arrow[r,"F_{\mathbb W}"]
  & F_{\mathbb W,*}\mathbb W \arrow[r] & 0.
\end{tikzcd}
\]
The map $d_{u,t}$ is a quasi-ideal: it is $\mathbb W$-linear and
satisfies $d_{u,t}(x)y=d_{u,t}(y)x$.  The ring stack
\[
 \mathbb G_a^{\Nyg}
 :=\operatorname{Cone}(M_u\xrightarrow{d_{u,t}}\mathbb W)
\]
is the stack quotient of $\mathbb W$ by the additive action of
$M_u$ through $d_{u,t}$, with the ring structure induced by the
quasi-ideal.  The middle column of the diagram gives a map
$\mathbb W/p\to\mathbb G_a^{\Nyg}$.  Together
with the canonical map $k\simeq W(k)/p\to\mathbb W/p$, it
makes $\mathbb G_a^{\Nyg}$ an animated $k$-algebra stack.
\end{construction}

\begin{definition}[Transmutation and filtered prismatization]
\label{def:filtered-prismatization}
Let $S$ be a stack and $\mathcal A$ an animated $k$-algebra stack on
$S$-schemes.  The \emph{transmutation} of a $k$-scheme $X$ along
$\mathcal A$ is the $S$-stack $X^{\mathcal A}$ with
\[
 X^{\mathcal A}(\Spec(R)\to S)
 :=X(\mathcal A(R))
 =\Map_k\bigl(\Spec(\mathcal A(R)),X\bigr).
\]
Put $\mathbb G_a^{\Delta}:=\mathbb W/p$, a $k$-algebra stack on
$p$-nilpotent $W$-algebras.  For a smooth $k$-scheme $X$, the
\emph{prismatization} $X^\Delta\to k^\Delta:=\Spf(W)$ and the
\emph{filtered prismatization} $X^{\Nyg}\to k^{\Nyg}$ are
\[
 X^\Delta(R)=\Map_k\bigl(\Spec(\mathbb W(R)/p),X\bigr),
 \qquad
 X^{\Nyg}(\Spec(R)\to k^{\Nyg})
 =\Map_k\bigl(\Spec(\mathbb G_a^{\Nyg}(R)),X\bigr).
\]
\end{definition}

Write $\varphi:k^\Delta\xrightarrow{\sim}k^\Delta$ for the
Witt-vector Frobenius.  The two opens $(k^{\Nyg})_{u\neq0}$ and
$(k^{\Nyg})_{t\neq0}$ are each isomorphic to $k^\Delta$, and they are
disjoint, since their intersection is the locus $p\neq0$, which is
empty in the $p$-adic formal spectrum.

\begin{construction}[The Hodge--Tate and de Rham pullbacks]
\label{conv:frobenius-gluing}
\label{con:HT-dR-pullbacks}
We call $(k^{\Nyg})_{u\neq0}$ the \emph{Hodge--Tate locus} and
$(k^{\Nyg})_{t\neq0}$ the \emph{de Rham locus} of $k^{\Nyg}$, with
inclusions
\[
 j_{\HT}:k^\Delta\simeq(k^{\Nyg})_{u\neq0}\hookrightarrow k^{\Nyg},
 \qquad
 \widetilde j_{\dR}:k^\Delta\simeq(k^{\Nyg})_{t\neq0}\hookrightarrow k^{\Nyg},
 \qquad
 j_{\dR}:=\widetilde j_{\dR}\circ\varphi^{-1}.
\]
We glue along the ordered pair $(j_{\HT},j_{\dR})$.  Thus
$j_{\HT}$ is $W$-linear and the Frobenius twist lies on the de Rham
locus.

Explicitly, $(k^{\Nyg})_{u\neq0}=\Spf(W[u^{\pm1}]^\wedge_p)/\Gm$, since
$t=pu^{-1}$ there, and
$\Dqc((k^{\Nyg})_{u\neq0})\simeq D_{p\text{-}\comp}(W)$
by taking weight zero, and similarly for $(k^{\Nyg})_{t\neq0}$.  Inverting $u$
or $t$ identifies every weight of $E$ with the colimit along $u$
or along $t$, so for a gauge $E$,
\[
 j_{\HT}^*E\simeq E^{\infty},
 \qquad
 \widetilde j_{\dR}^{\,*}E\simeq E^{-\infty}.
\]
We call these the \emph{Hodge--Tate pullback} and the \emph{de Rham
pullback} of $E$.

The ring stack has canonical identifications on the two loci,
\[
 j_{\HT}^*\mathbb G_a^{\Nyg}
 \simeq\mathbb W/p
 \simeq j_{\dR}^*\mathbb G_a^{\Nyg}
\]
of animated $k$-algebra stacks, the second using the Frobenius
twist in $j_{\dR}$.  Transmutation gives open immersions
$j_{\HT},j_{\dR}:X^\Delta\hookrightarrow X^{\Nyg}$ over the two
maps just defined.
\end{construction}

\begin{construction}[The Hodge-filtered and conjugate-filtered de Rham pullbacks]
\label{con:closed-loci}
The zero loci of $u$ and $t$ are the Cartier divisors
\[
 i_{\dR,+}:\mathbb A^1_{k,t}/\Gm=(k^{\Nyg})_{u=0}\hookrightarrow k^{\Nyg},
 \qquad
 i_{\dR,c}:\mathbb A^1_{k,u}/\Gm=(k^{\Nyg})_{t=0}\hookrightarrow k^{\Nyg},
\]
which meet transversally along the origin
$i_{\mathrm{Hodge}}:B\mathbb G_{m,k}=V(u,t)\hookrightarrow k^{\Nyg}$.  For a gauge
$E$ we call $i_{\dR,+}^*E$ its \emph{Hodge-filtered de Rham pullback},
$i_{\dR,c}^*E$ its \emph{conjugate-filtered de Rham pullback}, and
$i_{\mathrm{Hodge}}^*E$ its \emph{Hodge pullback}.  For a smooth
$k$-scheme $X$, base change of $X^{\Nyg}$ along these immersions
gives the Hodge-filtered de Rham stack, the Frobenius-twisted
conjugate-filtered de Rham stack, and the Hodge stack
\[
 X^{\dR,+}:=(X^{\Nyg})_{u=0},
 \qquad
 X^{\dR,c}:=(X^{\Nyg})_{t=0},
 \qquad
 X^{\mathrm{Hodge}}:=(X^{\Nyg})_{u=t=0}
\]
\cite[Remark~3.4.2]{BhattFGauges}.  Together with the two open loci
$X^\Delta$, these stratify $X^{\Nyg}$.
\end{construction}

The filtered prismatization recovers the classical cohomology
theories of $X$ on these loci.

\begin{theorem}[{Specializations of $X^{\Nyg}$;
\cite[Theorem~3.3.5]{BhattFGauges}}]
\label{thm:ordinary-input}
Let $X$ be a smooth quasi-compact $k$-scheme, and put
$\cH_{\Nyg}(X):=\pi_{X,*}\cO_{X^{\Nyg}}\in\Dqc(k^{\Nyg})$.
\begin{enumerate}
\item $\cH_{\Nyg}(X)$ is the Nygaard filtration $\Fil_{\Nyg}^{\geq\bullet}$ on
 $R\Gamma_{\crys}(X/W)=\varphi^*R\Gamma_\Delta(X/W)$.
\item On $(k^{\Nyg})_{t\neq0}$, the de Rham pullback is
 $\cH_{\Nyg}(X)^{-\infty}\simeq R\Gamma_{\crys}(X/W)$.
\item On $(k^{\Nyg})_{u\neq0}$, the Hodge--Tate pullback is
 $\cH_{\Nyg}(X)^{\infty}\simeq R\Gamma_\Delta(X/W)=(\sigma^{-1})^*R\Gamma_{\crys}(X/W)$.
\item On $(k^{\Nyg})_{u=0}\simeq\mathbb A^1_{k,t}/\Gm$, the
 Hodge-filtered de Rham pullback $i_{\dR,+}^*\cH_{\Nyg}(X)$ is the
 Hodge filtration
 $\Fil_H^\bullet R\Gamma(X,\Omega^\bullet_{X/k})$.
\item On $(k^{\Nyg})_{t=0}\simeq\mathbb A^1_{k,u}/\Gm$, the
 conjugate-filtered de Rham pullback $i_{\dR,c}^*\cH_{\Nyg}(X)$ is
 $(\sigma^{-1})^*$ of the conjugate filtration
 $\Fil^{\mathrm{conj}}_\bullet R\Gamma(X,\Omega^\bullet_{X/k})$.
\item On $(k^{\Nyg})_{u=t=0}=B\mathbb G_{m,k}$, the Hodge pullback
 has weight-$i$ piece
 $(i_{\mathrm{Hodge}}^*\cH_{\Nyg}(X))^i\simeq R\Gamma(X,\Omega^i_{X/k})[-i]$.
\end{enumerate}
\end{theorem}

\begin{definition}[Syntomification and $F$-gauges]
\label{def:syntomification}
\label{def:syntomic-point}
\label{def:prismatic-F-gauge}
For a smooth $k$-scheme $X$, the \emph{syntomification} $X^{\syn}$
is the pushout
\[
\begin{CD}
 X^\Delta\amalg X^\Delta @>{(j_{\HT},\,j_{\dR})}>>
 X^{\Nyg}\\
 @VVV @VV{j_{\Nyg}}V\\
 X^\Delta @>{j_\Delta}>>
 X^{\syn}.
\end{CD}
\]
The construction is functorial in $X$.  For $X=\Spec(k)$ it gives
the \emph{syntomic point} $k^{\syn}$, and functoriality gives the
structure map $\pi_X:X^{\syn}\to k^{\syn}$.  An \emph{$F$-gauge} is
an object of $\Dqc(k^{\syn})$, and the \emph{$F$-gauge cohomology}
of $X$ is
\[
 \cH_{\syn}(X):=\pi_{X,*}\cO_{X^{\syn}}\in\CAlg(\Dqc(k^{\syn})).
\]
\end{definition}

All maps in the square are \'etale, and $j_{\Nyg}$ is a surjective
local isomorphism, hence an \'etale cover of $k^{\syn}$.  Functoriality
gives
$X^{\Nyg}\simeq X^{\syn}\times_{k^{\syn}}k^{\Nyg}$
\cite[Remark~4.1.4]{BhattFGauges}, so base change along $j_{\Nyg}$
gives $j_{\Nyg}^*\cH_{\syn}(X)\simeq\cH_{\Nyg}(X)$. The identifications on the two loci descend $\mathbb G_a^{\Nyg}$
to an animated $k$-algebra stack $\mathbb G_a^{\syn}$ on $k^{\syn}$,
and $X^{\syn}$ is alternatively the transmutation of $X$ along
$\mathbb G_a^{\syn}$ \cite[Remark~4.1.6]{BhattFGauges}.

By descent along this gluing, an $F$-gauge $E$ is a gauge
$j_{\Nyg}^*E$ together with an isomorphism
$j_{\HT}^*E\xrightarrow{\sim}j_{\dR}^*E$ on $k^\Delta$, its
\emph{descent datum}.  By Construction~\ref{con:HT-dR-pullbacks},
the descent datum is equivalently an isomorphism of derived
$p$-complete $W$-modules $\varphi^*E^{\infty}\xrightarrow{\sim}E^{-\infty}$
between the Hodge--Tate and de Rham pullbacks.  We abuse notation
and write $E$ also for the underlying gauge $j_{\Nyg}^*E$.

\begin{definition}[{Breuil--Kisin twists; \cite[Definition~4.3.4]{BhattFGauges}}]
\label{def:BK-twist}
The \emph{Breuil--Kisin twist} $\cO\{-1\}\in\operatorname{Pic}(k^{\syn})$
is the line bundle whose underlying $F$-crystal is $W$ with Frobenius
structure $\varphi^*W\xrightarrow{\ p\ }W$.  Its restriction to
$k^{\Nyg}$ is the dual $L^{\vee}$ of the tautological line bundle
$L$ of Construction~\ref{con:nygaard-ring-stack}, so that
$\cO\{1\}|_{k^{\Nyg}}=L$.  Its restriction to $k^\Delta=\Spf(W)$ is a
free $W$-module $We$ of rank one, with Frobenius structure
$\varphi^*(We)\to We$, $e\mapsto pe$.
For $E\in\Dqc(k^{\syn})$ and $m\in\mathbb Z$ put
$E\{m\}:=E\otimes\cO\{-1\}^{\otimes(-m)}$, and likewise on
$k^{\Nyg}$ and $k^\Delta$.  On $k^\Delta$ we trivialize $\cO\{m\}$ by
the basis $e^{\otimes(-m)}$, so that $\cO\{m\}|_{k^\Delta}=W$ with
Frobenius structure $p^{-m}\sigma$.
\end{definition}

\begin{remark}[Geometry of the syntomic point]
\label{rem:two-points-of-syn}
The two open loci of $k^{\Nyg}$ have the same image in $k^{\syn}$,
namely the open substack $\Spf(W)$, and the origin $V(u,t)$ maps to
the closed substack $B\mathbb G_{m,k}$.  The two divisors $u=0$ and
$t=0$ give two specializations from the open substack to the closed
one, so $k^{\syn}$ is not separated.  Because of the Frobenius in
the gluing, $k^{\syn}$ is a $\Zp$-stack and not a $W$-stack
\cite[Remark~4.1.2 and Warning~4.1.3]{BhattFGauges}.
\end{remark}

\begin{remark}[Coefficient Frobenius]
\label{rem:coefficient-Frobenius-syn}
Since $k$ is perfect, its absolute Frobenius
$F_k:\Spec(k)\xrightarrow{\sim}\Spec(k)$ is an automorphism.
Functoriality of syntomification \cite[\S6.1]{BhattFGauges} gives
an automorphism of $\Zp$-stacks $\Phi_{\syn}:=(F_k)^{\syn}:k^{\syn}\xrightarrow{\sim}k^{\syn}$. On the Nygaard point $k^{\Nyg}$ the induced automorphism
$\Phi_{\Nyg}:=(F_k)^{\Nyg}$ applies $\sigma$ to the coefficients
and fixes $u$ and $t$.  So pullback along $\Phi_{\syn}$ is the Frobenius
twist on underlying gauges, and for an $F$-gauge $E$ we write
$\sigma^*E:=\Phi_{\syn}^*E$.
\end{remark}
\subsection{Genuine and naive \texorpdfstring{$r$}{r}-Nygaard stacks}
\label{sec:genuine-naive}

Fix $r\geq1$ and let $X$ be a smooth $k$-scheme.  By
Theorem~\ref{thm:ordinary-input}, the Nygaard filtration on
$R\Gamma_{\crys}(X/W)$ is the filtration whose Hodge-filtered de Rham
pullback is $R\Gamma(X,\Omega^\bullet_{X/k})$ with its Hodge filtration.
Andriopoulos defined a filtration on the same object that plays this
role for $W_r\Omega^\bullet_{X/k}$, obtained by iterating the Nygaard
filtration $r$ times along the divided Frobenius.  We recall his
definition and results in the notation of Theorem~\ref{thm:ordinary-input}.

\begin{definition}[{The $r$-Nygaard filtration;
\cite[Theorem~1.1 and Definition~5.1]{AndriopoulosTR}}]
\label{def:r-nygaard-filtration}
Let $X$ be a smooth $k$-scheme, and let
$\iota_m,\varphi_m:\Fil^{\geq m}_{\Nyg}R\Gamma_{\crys}(X/W)\to R\Gamma_{\crys}(X/W)$
be the inclusion and the divided Frobenius $\varphi_m=\varphi/p^m$ of
the Nygaard filtration, the latter $\sigma$-semilinear.  The
\emph{$r$-Nygaard filtration} on $R\Gamma_{\crys}(X/W)$ has $m$-th
filtered piece
\[
 \Fil^{\geq m}_{\Nyg_r}R\Gamma_{\crys}
 :=\Fil^{\geq m}_{\Nyg}R\Gamma_{\crys}
 \times_{R\Gamma_{\crys}}\cdots
 \times_{R\Gamma_{\crys}}
 \Fil^{\geq m}_{\Nyg}R\Gamma_{\crys},
\]
with $r$ factors, where each fibre product is formed with $\iota_m$
on its left factor and $\varphi_m$ on its right factor.
\end{definition}

At $r=1$ this is the Nygaard filtration.  Andriopoulos writes the
$m$-th filtered piece with a Breuil--Kisin twist $\{m\}$, which is
where the divided Frobenius lands canonically.  Over the crystalline
prism we have trivialized the twist by the basis $e^{\otimes(-m)}$
(Definition~\ref{def:BK-twist}), on which Frobenius acts by
$p^{-m}\sigma$, so $\varphi_m$ is literally $\varphi/p^m$.  Since $\varphi_m$ is only semilinear, the fibre products are formed in $\Zp$-modules, and
making them $W$-linear requires a Frobenius twist of the
coefficients at each step.

\begin{construction}[{Operators on the $r$-Nygaard filtration;
\cite[Theorem~1.1 and Definition~5.1]{AndriopoulosTR}}]
\label{con:r-nygaard-operators}
For $r\geq1$, the $r$-Nygaard filtrations at levels $r$ and $r+1$
are related by two maps
\[
 \pi_r,F_r:\Fil^{\geq m}_{\Nyg_{r+1}}R\Gamma_{\crys}
 \longrightarrow\Fil^{\geq m}_{\Nyg_r}R\Gamma_{\crys}.
\]
Here $\pi_r$ forgets the leftmost factor of the iterated fibre product
and is called \emph{restriction}, and $F_r$ forgets the rightmost factor
and is called \emph{Frobenius}.  There are also two maps to the
underlying object,
\[
 \iota:\Fil^{\geq m}_{\Nyg_r}R\Gamma_{\crys}
 \lhook\joinrel\longrightarrow R\Gamma_{\crys},
 \qquad
 \varphi_{r,m}:\Fil^{\geq m}_{\Nyg_r}R\Gamma_{\crys}
 \longrightarrow R\Gamma_{\crys},
\]
where the \emph{inclusion} $\iota$ is the projection to the rightmost
factor followed by $\iota_m$ and the \emph{$r$-divided Frobenius}
$\varphi_{r,m}$ is the projection to the leftmost factor followed by
$\varphi_m$.  The map $\varphi_{r,m}$ is
$\sigma^r$-semilinear, and $\varphi^r=p^{rm}\varphi_{r,m}$ on
$\Fil^{\geq m}_{\Nyg_r}R\Gamma_{\crys}$.  These maps satisfy
$\iota\circ\pi_r=\iota$ and $\varphi_{r,m}\circ F_r=\varphi_{r+1,m}$.
Verschiebung is constructed later, on the finite de Rham--Witt
realizations, by duality
(Construction~\ref{def:intrinsic-geometric-V}).
\end{construction}

The $r$-Nygaard filtration admits a characterization of the same
form as Bhatt's characterization of the Nygaard filtration
\cite[Theorem~3.2.1]{BhattFGauges}, with $p^r$ in place of $p$ and
$W_r\Omega^\bullet_{X/k}$ in place of $\Omega^\bullet_{X/k}$.

\begin{theorem}[{Characterization of the $r$-Nygaard filtration;
\cite[Theorem~1.8 and Corollary~5.14]{AndriopoulosTR}}]
\label{thm:andriopoulos}
Let $X$ be a smooth affine $k$-scheme.  The $r$-fold crystalline Frobenius
refines to a $\sigma^r$-semilinear filtered map
\[
 \varphi^r:\Fil^{\geq\bullet}_{\Nyg_r}R\Gamma_{\crys}
 \longrightarrow p^{r\bullet}R\Gamma_{\crys},
\]
and the $r$-Nygaard filtration is uniquely characterized by the
following two properties.
\begin{enumerate}
\item $\Fil^{\geq\bullet}_{\Nyg_r}R\Gamma_{\crys}(X/W)$ is complete as a
 filtered object.
\item The map $\varphi^r$ induces isomorphisms
 \[
  \gr^i_{\Nyg_r}R\Gamma_{\crys}
  \simeq\tau^{\leq i}\gr^i_{p^{r\bullet}}R\Gamma_{\crys}
  \simeq\tau^{\leq i}R\Gamma(X,W_r\Omega^\bullet_{X/k}).
 \]
\end{enumerate}
Equivalently, $\varphi^r$ identifies the $r$-Nygaard filtration with
the Beilinson-connective cover of the $p^r$-adic filtration on
$R\Gamma_{\crys}(X/W)$.
\end{theorem}

The affineness hypothesis is needed because the truncations
$\tau^{\leq i}$ do not commute with $R\Gamma(X,-)$.  For $X$
smooth and proper, $r=1$ and $i=0$, the Cartier isomorphism gives
$\gr^0_{\Nyg}R\Gamma_{\crys}(X/W)\simeq R\Gamma(X,\cO_X)$ up to
a Frobenius twist, which has $H^1\neq0$ in general, whereas
$\tau^{\leq0}R\Gamma_{\dR}(X/k)=H^0_{\dR}(X/k)$.  For general
$X$ the truncations must be taken sheafwise, and
Corollary~\ref{cor:bridge} gives the resulting global statement.

\begin{example}[The $r$-Nygaard filtration on $W$ itself]
\label{ex:r-nygaard-on-W}
For $X=\Spec k$ the Nygaard filtration on $W$ is the $p$-adic
filtration \cite[Example~3.2.4]{BhattFGauges}.  For $m\geq0$ the map
$\varphi_m=\sigma/p^m$ carries $p^mW$ onto $W$, and for $m\leq0$ the
inclusion $\iota_m$ is the identity.  So the fibre products are
discrete, and
$\Fil^{\geq m}_{\Nyg_r}W=p^{rm}W$ for $m\geq0$: the $r$-Nygaard
filtration on $W$ is the $p^r$-adic filtration, with
$\varphi_{r,m}=\sigma^r/p^{rm}$.
\end{example}

\begin{example}[The de Rham--Witt model]
\label{ex:r-nygaard-dRW}
Let $X=\Spec S$ be a smooth affine $k$-scheme, and let
$\WOm^\bullet_{S/k}$ be its de Rham--Witt complex, a $p$-torsion-free
model of $R\Gamma_{\crys}(X/W)$ with $\varphi=p^nF$ in degree $n$
\cite{Illusie79}.  Bhatt--Lurie--Mathew
\cite[Construction~8.1.1 and Proposition~8.2.1]{BLM21} represent the
Nygaard filtration by the subcomplexes with $n$-th term
$p^{m-n-1}V\WOm^n_{S/k}$ for $n<m$ and $\WOm^n_{S/k}$ for $n\geq m$.
By Theorem~\ref{thm:andriopoulos}, the $m$-th piece of the
$r$-Nygaard filtration is the preimage under $\varphi^r$ of the $m$-th
piece of the connective cover.  In degree $n\geq m$ there is no
condition, and in degree $n<m$ the condition is
$F^rx\in p^{r(m-n)}\WOm^n_{S/k}$.  Since $F$ is injective and
$p^cy=F(p^{c-1}Vy)$, applying this $r$ times gives
\[
 \bigl(\Fil^{\geq m}_{\Nyg_r}\WOm^\bullet_{S/k}\bigr)^n=
 \begin{cases}
  p^{r(m-n-1)}V^r\WOm^n_{S/k}, & n<m,\\[2pt]
  \WOm^n_{S/k}, & n\geq m.
 \end{cases}
\]
\end{example}

The Nygaard point $k^{\Nyg}$ is the Rees stack of the $p$-adic
filtration on $W$, and the Nygaard filtration is the gauge
$\cH_{\Nyg}(X)$ on it.  By Example~\ref{ex:r-nygaard-on-W} the
$r$-Nygaard filtration on $W$ is the $p^r$-adic filtration, so the
first candidate is the Rees stack of the $p^r$-adic filtration.

\begin{definition}[Naive \(r\)-Nygaard stack]
\label{def:naive-nygaard-stack}
The \emph{naive $r$-Nygaard stack} is the Rees stack of the
$p^r$-adic filtration on $W$,
\[
 k^{\Nyg_r}:=\Spf\bigl((W[u,t]/(ut-p^r))^\wedge_p\bigr)/\Gm,
\]
where $|u|=1$ and $|t|=-1$.
\end{definition}

Replacing $p$ by $p^r$ in the Nygaard point thus gives the naive
$r$-Nygaard stack.  At $r=1$ it is $k^{\Nyg}$.

\begin{remark}[The naive $r$-Nygaard stack and its $A_{r-1}$-singularity]
\label{rem:naive-defects}
Let $r\geq2$.  The naive $r$-Nygaard stack $k^{\Nyg_r}$ has two
drawbacks, both avoided by the genuine $r$-Nygaard stack $k^{\NN_r}$
introduced below.  First, $k^{\Nyg_r}$ is not regular: at the origin of
$W[u,t]/(ut-p^r)$ the element $ut=p^r$ lies in the square of the
maximal ideal $(p,u,t)$, so $k^{\Nyg_r}$ has an
$A_{r-1}$-singularity there.  Second, it is unclear how to construct
a $k$-algebra stack $\mathbb G_a^{\Nyg_r}$ over $k^{\Nyg_r}$ whose
transmutation would attach a stack $X^{\Nyg_r}$ to every smooth
$k$-scheme $X$, as $\mathbb G_a^{\Nyg}$ does in
Definition~\ref{def:filtered-prismatization}.  The first guess,
$\mathbb W/p^r$ in place of $\mathbb W/p$, is a
$\mathbf Z/p^r$-algebra stack and not a $k$-algebra stack, and
evaluating a $k$-scheme on it gives the empty
stack.\footnote{This was pointed out to the author by Vadim
Vologodsky.  Drinfeld \cite[\S1.6.4]{DrinfeldBT} records the same
fact: the $r$-prismatization defined by
$\operatorname{Cone}(p^r:\mathbb W\to\mathbb W)$ is empty on every
scheme over $\mathbf Z/p^{r-1}$.}
\end{remark}

The toric resolution of $ut=z^r$ is a chain of $r$ copies of $ut=z$,
and base change along $z\mapsto p$ turns it into a chain of $r$ copies of the Nygaard point.  We call this chain the \emph{genuine $r$-Nygaard
stack}.  As in equivariant homotopy theory, where $\mathrm{TR}^r$ is
built from genuine fixed points (Remark~\ref{rem:TR-motivation}),
``genuine'' is opposed to ``naive''.  The
genuine $r$-Nygaard stack $k^{\NN_r}$ is regular,
it maps properly onto $k^{\Nyg_r}$, and the Frobenius
gluing of adjacent Nygaard points, with the de Rham locus of each
glued to the Hodge--Tate locus of the next, is exactly what encodes
the iterated fibre product of Definition~\ref{def:r-nygaard-filtration},
with the coefficient twists that make its steps $W$-linear.

\begin{definition}[Genuine \(r\)-Nygaard stack]
\label{def:adjacent-transition}
For $i\geq0$, let $W^{(i)}:=(\sigma^i)^*W=W\otimes_{W,\sigma^i}W$,
and write $a^{(i)}:=1\otimes a$ for the copy of $a\in W$ in the
second factor.  Thus
$W^{(i)}$ is a copy of $W$ whose structural map is
$a\mapsto\sigma^{-i}(a)^{(i)}$.  The \emph{$i$-th Nygaard
chart} $k^{\Nyg}_i$ and its chart ring $B_i$ are
\[
 B_i:=\bigl(W^{(i)}[u_i,t_i]/(u_it_i-p)\bigr)^{\wedge}_p,
 \qquad k^{\Nyg}_i:=\Spf(B_i)/\Gm,
\]
where $|u_i|=1$ and $|t_i|=-1$, with structural map
$\lambda_i:W\to B_i$, $\lambda_i(a)=\sigma^{-i}(a)^{(i)}$.
The transition map between adjacent charts,
$\alpha_i:(k^{\Nyg}_i)_{t_i\neq0}\xrightarrow{\sim}(k^{\Nyg}_{i+1})_{u_{i+1}\neq0}$,
is the isomorphism given on coordinates by
\begin{equation}\label{eq:overlap-gluing}
 a^{(i+1)}\longmapsto\sigma(a)^{(i)},
 \qquad
 u_{i+1}\longmapsto t_i^{-1},
 \qquad
 t_{i+1}\longmapsto p\,t_i.
\end{equation}
The \emph{genuine $r$-Nygaard stack} is the chain
\[
 k^{\NN_r}:=k^{\Nyg}_0\cup_{\alpha_0}k^{\Nyg}_1\cup_{\alpha_1}\cdots
 \cup_{\alpha_{r-2}}k^{\Nyg}_{r-1}.
\]
\end{definition}

The formulas \eqref{eq:overlap-gluing} preserve weights and the
relation $u_it_i=p$, and they are invertible, with inverse
$a^{(i)}\mapsto\sigma^{-1}(a)^{(i+1)}$,
$t_i\mapsto u_{i+1}^{-1}$, $u_i\mapsto p\,u_{i+1}$.
Moreover $\alpha_i^*\lambda_{i+1}=\lambda_i$, so each transition
map is $W$-linear.  Since the two open loci of a chart are disjoint
(\S\ref{sec:syntomic-point}), only adjacent charts meet, along
$k^{\Nyg}_i\cap k^{\Nyg}_{i+1}=(k^{\Nyg}_i)_{t_i\neq0}\simeq(k^{\Nyg}_{i+1})_{u_{i+1}\neq0}$,
and there is no cocycle condition to check.

Next we construct the contraction map $k^{\NN_r}\to k^{\Nyg_r}$,
the analogue of the resolution of an $A_{r-1}$-singularity.

\begin{proposition}[The contraction map \(g_r\)]
\label{prop:contraction-map}
There is a $W$-linear morphism $g_r:k^{\NN_r}\longrightarrow k^{\Nyg_r}$ whose restriction $g_{r,i}:k^{\Nyg}_i\to k^{\Nyg_r}$ is given on coordinates by
\[
 a\longmapsto\lambda_i(a)=\sigma^{-i}(a)^{(i)},
 \qquad
 u\longmapsto p^iu_i,
 \qquad
 t\longmapsto p^{r-1-i}t_i
\]
for $a\in W$ and $0\leq i<r$.
\end{proposition}

\begin{proof}
The formulas preserve weights and send $ut$ to
$p^{r-1}u_it_i=p^r$, so they define maps on the charts.
Their coefficient maps agree on overlaps because
$\alpha_i^*\lambda_{i+1}=\lambda_i$.
For the coordinates,
\[
 \begin{aligned}
 \alpha_i^*(p^{i+1}u_{i+1})
   &=p^{i+1}t_i^{-1}=p^iu_i,\\
 \alpha_i^*(p^{r-2-i}t_{i+1})
   &=p^{r-1-i}t_i.
 \end{aligned}
\]
Hence the chart maps glue to $g_r$.
\end{proof}

Resolving the singularity is not the only purpose of this
construction.  What makes the genuine $r$-Nygaard stack $k^{\NN_r}$ useful is
that its charts are glued by the same Frobenius that glues the two
loci of $k^{\Nyg}$ into $k^{\syn}$.  So the chart inclusions
$k^{\Nyg}_i\to k^{\Nyg}\to k^{\syn}$ agree on the overlaps and glue
to a map from $k^{\NN_r}$ to the syntomic point.

\begin{proposition}[The syntomic projection \(f_r\)]
\label{prop:syntomic-projection}
For $0\leq i<r$, let $f_{r,i}:k^{\Nyg}_i\to k^{\Nyg}$ be the map
given on coordinates by $a\mapsto a^{(i)}$, $u\mapsto u_i$,
$t\mapsto t_i$.  It is $\sigma^i$-semilinear over $W$, since
$f_{r,i}^*a=\lambda_i(\sigma^i(a))$.  The maps $j_{\Nyg}\circ f_{r,i}$
glue to a canonical morphism $f_r:k^{\NN_r}\to k^{\syn}$.
\end{proposition}

\begin{proof}
Taking weight zero identifies the overlap
$(k^{\Nyg}_i)_{t_i\neq0}=\Spf\bigl(W^{(i)}[t_i^{\pm1}]\bigr)^{\wedge}_p/\Gm$
with $\Spf(W^{(i)})$.  Let $y:\Spf(W^{(i)})\to k^\Delta$ be the map
$y^*(a)=\sigma(a)^{(i)}$.  The map $f_{r,i}$ sends $t$ to $t_i$, so it
lands in the de Rham locus, and after trivializing $t$ and $t_i$ its
coefficient map is $a\mapsto a^{(i)}$.  Hence
$f_{r,i}=\widetilde j_{\dR}\circ y'$ with $y'^*(a)=a^{(i)}$.  Since
$\widetilde j_{\dR}=j_{\dR}\circ\varphi$ and $\varphi\circ y'=y$, we
get $f_{r,i}=j_{\dR}\circ y$.  The map
$f_{r,i+1}\circ\alpha_i$ sends $u$ to $t_i^{-1}$, so it lands in the
Hodge--Tate locus, and after trivializing $u$ and $t_i$ its coefficient
map is $a\mapsto a^{(i+1)}\mapsto\sigma(a)^{(i)}$.  As $j_{\HT}$ is
$W$-linear, $f_{r,i+1}\circ\alpha_i=j_{\HT}\circ y$.  The gluing
$j_{\HT}\simeq j_{\dR}$ in $k^{\syn}$
(Construction~\ref{conv:frobenius-gluing}) therefore identifies
$j_{\Nyg}\circ f_{r,i}$ with $j_{\Nyg}\circ f_{r,i+1}\circ\alpha_i$ on
the overlap.  Since chart $i+1$ meets chart $i$ and chart $i+2$ in its
disjoint loci $u_{i+1}\neq0$ and $t_{i+1}\neq0$, there are no triple
overlaps, and these identifications glue the maps
$j_{\Nyg}\circ f_{r,i}$ to $f_r$.
\end{proof}

The next lemma describes the syntomic projection $f_r$ over the
\'etale cover $j_{\Nyg}:k^{\Nyg}\to k^{\syn}$.

\begin{lemma}[The syntomic projection over the Nygaard point]
\label{lem:syntomic-projection-fibre}
There is an equivalence over $k^{\Nyg}$
\[
 k^{\NN_r}\times_{k^{\syn}}k^{\Nyg}\simeq
 \coprod_{i=0}^{r-1}k^{\Nyg}_i\amalg(k^{\Nyg}_0)_{u_0\neq0}
 \amalg(k^{\Nyg}_{r-1})_{t_{r-1}\neq0},
\]
where $k^{\Nyg}_i$ maps to $k^{\Nyg}$ by $f_{r,i}$, and the two end
summands map isomorphically onto the de Rham and the Hodge--Tate locus
of $k^{\Nyg}$, respectively.  In particular, $f_r$ is finite on every
closed substack of $k^{\NN_r}$ that is disjoint from the two ends
$(k^{\Nyg}_0)_{u_0\neq0}$ and $(k^{\Nyg}_{r-1})_{t_{r-1}\neq0}$.
\end{lemma}

\begin{proof}
As $k^{\syn}$ is the pushout of $j_{\HT},j_{\dR}:k^\Delta\to k^{\Nyg}$
along disjoint images, $k^{\Nyg}\times_{k^{\syn}}k^{\Nyg}\simeq
k^{\Nyg}\amalg k^\Delta\amalg k^\Delta$.  Over the chart $k^{\Nyg}_i$
this adds copies of $(k^{\Nyg}_i)_{u_i\neq0}$ and
$(k^{\Nyg}_i)_{t_i\neq0}$, which the gluing of
Proposition~\ref{prop:syntomic-projection} identifies with open parts
of the adjacent charts, except at the two ends.  This gives the
decomposition.  A closed substack $Z$ disjoint from the ends therefore
has $Z\times_{k^{\syn}}k^{\Nyg}\simeq\coprod_i(Z\cap k^{\Nyg}_i)$, a
disjoint union of closed substacks of $k^{\Nyg}$, so $Z\to k^{\syn}$ is
finite, as finiteness is \'etale local on the target.
\end{proof}

We have now constructed, for each $r\geq1$, the correspondence
\[
\begin{tikzcd}[row sep=large, column sep=large]
 & k^{\NN_r} \arrow[dl, "f_r"'] \arrow[dr, "g_r"] & \\
 k^{\syn} & & k^{\Nyg_r}.
\end{tikzcd}
\]
Quasi-coherent complexes on $k^{\NN_r}$ and on $k^{\Nyg_r}$ have
concrete descriptions, which we name after the stacks.

\begin{definition}[Genuine and naive $r$-gauges]
\label{def:r-gauges}
A \emph{naive $r$-gauge} is an object $E$ of $\Dqc(k^{\Nyg_r})$.
As in Definition~\ref{def:nygaard-point}, it has maps
$u:E^n\to E^{n+1}$ and $t:E^{n+1}\to E^n$, and now their composites
are multiplication by $p^r$.  A \emph{genuine $r$-gauge} is an
object $\mathcal E$ of $\Dqc(k^{\NN_r})$.  Since only adjacent charts
meet, it is a chain
$(E_0,\ldots,E_{r-1})$ of gauges $E_i\in\Dqc(k^{\Nyg}_i)$ together
with identifications
\[
 \delta_i:E_i|_{t_i\neq0}\xrightarrow{\ \sim\ }
 \alpha_i^*\bigl(E_{i+1}|_{u_{i+1}\neq0}\bigr),
 \qquad 0\leq i\leq r-2,
\]
that is, of the de Rham pullback $E_i^{-\infty}$ with the
Hodge--Tate pullback $E_{i+1}^{\infty}$ along the transition map.
At $r=1$ both notions are gauges.  When no confusion with genuine
$r$-gauges can arise, we refer to naive $r$-gauges as just
$r$-gauges.
\end{definition}

The genuine $r$-Nygaard stack $k^{\NN_r}$ also has a version for a smooth
$k$-scheme $X$, obtained by base change from the syntomification.

\begin{definition}[Genuine $r$-Nygaard prismatization]
\label{def:genuine-r-nygaard-prismatization}
For a smooth $k$-scheme $X$, the \emph{genuine $r$-Nygaard
prismatization} of $X$ is
\[
 X^{\NN_r}:=X^{\syn}\times_{k^{\syn}}k^{\NN_r},
\]
with structure map $\pi_X:X^{\NN_r}\to k^{\NN_r}$.
\end{definition}

Since $X^{\Nyg}\simeq X^{\syn}\times_{k^{\syn}}k^{\Nyg}$ and
$f_{r,i}$ is an isomorphism onto $k^{\Nyg}$, the restriction of
$X^{\NN_r}$ to the chart $k^{\Nyg}_i$ is a copy of $X^{\Nyg}$ with
coefficients twisted to $W^{(i)}$, and $X^{\NN_r}$ is a chain of
$r$ copies of the filtered prismatization $X^{\Nyg}$ glued along
Frobenius.  At $r=1$ it is $X^{\Nyg}$.  We do not know a
corresponding construction over the naive $r$-Nygaard stack
$k^{\Nyg_r}$, for the reason
given in Remark~\ref{rem:naive-defects}.

The two legs play different roles.  The syntomic projection $f_r$
is built from the Frobenius gluing, and pulling back along it turns
an $F$-gauge $E$ into a genuine $r$-gauge.  On the $i$-th chart the
pullback is $f_{r,i}^*E$, which is the underlying gauge of $E$ with
its coefficients twisted to $W^{(i)}$.  On the overlap
$(k^{\Nyg}_i)_{t_i\neq0}$ the two chart maps $f_{r,i}$ and
$f_{r,i+1}\circ\alpha_i$ land in the de Rham and the Hodge--Tate
locus of $k^{\Nyg}$ and have the same composite to $k^{\syn}$, so
the two restrictions of the pullback are identified by the descent
datum of $E$.  This identification is $\delta_i$.  The contraction $g_r$ is built from the
$p^r$-adic filtration, and pushing forward along it turns a genuine
$r$-gauge into an $r$-gauge, a graded module over
$\Rees(p^{r\bullet}W)$, which is the shape of an $r$-Nygaard
filtered object.  At $r=1$ the correspondence collapses:
$g_1=\mathrm{id}_{k^{\Nyg}}$ and $f_1=j_{\Nyg}$.

Finally, we construct two families of maps on the genuine and on
the naive $r$-Nygaard stacks $k^{\NN_r}$ and $k^{\Nyg_r}$ between adjacent levels $r$ and $r+1$.
A chain of length $r$ sits in one of length $r+1$ as either the
right or the left sub-chain, and the two inclusions, with the
coefficient Frobenius, are called the restriction map $\mathsf R_r$
and the Frobenius map $\mathsf F_r$.  We use
the same names for the corresponding maps on the naive stacks
$k^{\Nyg_r}$, and
the source distinguishes the two.

\begin{definition}[Restriction and Frobenius]
\label{def:level-maps}
The maps
$\mathsf R_r,\mathsf F_r:k^{\NN_r}\to k^{\NN_{r+1}}$
are defined as follows.
Restriction sends $k^{\Nyg}_i$ to $k^{\Nyg}_{i+1}$, and is given
on coordinates by
\[
 a^{(i+1)}\longmapsto\sigma(a)^{(i)},\qquad
 u_{i+1}\longmapsto u_i,\qquad t_{i+1}\longmapsto t_i.
\]
Frobenius sends $k^{\Nyg}_i$ to $k^{\Nyg}_i$, and is given on
coordinates by
\[
 a^{(i)}\longmapsto\sigma(a)^{(i)},\qquad
 u_i\longmapsto u_i,\qquad t_i\longmapsto t_i.
\]
These maps respect the transition maps between adjacent charts and hence glue.
Informally, with $k^{\NN_r}$ in the top row, $k^{\NN_{r+1}}$ in the
bottom row and the gluings $\alpha_i$ as horizontal lines, they can be
pictured as follows
\[
\begin{tikzcd}[column sep=small, row sep=large]
 k^{\Nyg}_0 \arrow[r,dash] \arrow[d,"\mathsf F_r"'] \arrow[dr,"\mathsf R_r"]
 & k^{\Nyg}_1 \arrow[r,dash] \arrow[d] \arrow[dr]
 & \cdots \arrow[r,dash]
 & k^{\Nyg}_{r-1} \arrow[d] \arrow[dr] & \\
 k^{\Nyg}_0 \arrow[r,dash]
 & k^{\Nyg}_1 \arrow[r,dash]
 & \cdots \arrow[r,dash]
 & k^{\Nyg}_{r-1} \arrow[r,dash]
 & k^{\Nyg}_r .
\end{tikzcd}
\]
On the naive stacks, define
$\mathsf R_r,\mathsf F_r:k^{\Nyg_r}\to k^{\Nyg_{r+1}}$
on coordinates by
\[
 \mathsf R_r:\quad a\longmapsto a,
 \qquad u\longmapsto pu,
 \qquad t\longmapsto t,
\]
\[
 \mathsf F_r:\quad a\longmapsto\sigma(a),
 \qquad u\longmapsto u,
 \qquad t\longmapsto pt.
\]
Both preserve the defining relation.
On each tower, restriction is $W$-linear and Frobenius is
$\sigma$-semilinear.
\end{definition}

\begin{proposition}[Compatibility of $\mathsf R_r$ and $\mathsf F_r$ with the correspondence]
\label{prop:level-maps-compatible}
The following two squares commute, for restriction and for
Frobenius:
\[
\begin{tikzcd}[column sep=large, row sep=large]
 k^{\NN_r} \arrow[r, shift left, "\mathsf R_r"] \arrow[r, shift right, "\mathsf F_r"'] \arrow[d, "g_r"']
 & k^{\NN_{r+1}} \arrow[d, "g_{r+1}"]\\
 k^{\Nyg_r} \arrow[r, shift left, "\mathsf R_r"] \arrow[r, shift right, "\mathsf F_r"']
 & k^{\Nyg_{r+1}},
\end{tikzcd}
\qquad
\begin{tikzcd}[column sep=large, row sep=large]
 k^{\NN_r} \arrow[r, shift left, "\mathsf R_r"] \arrow[r, shift right, "\mathsf F_r"'] \arrow[d, "f_r"']
 & k^{\NN_{r+1}} \arrow[d, "f_{r+1}"]\\
 k^{\syn} \arrow[r, "\Phi_{\syn}"]
 & k^{\syn}.
\end{tikzcd}
\]
In each square the upper horizontal arrows form one commutative
square and the lower horizontal arrows another.
\end{proposition}

\begin{proof}
On $k^{\Nyg}_i$, the two composites in the restriction square send
$(u,t)$ to $(p^{i+1}u_i,p^{r-1-i}t_i)$, and those in the
Frobenius square send it to $(p^iu_i,p^{r-i}t_i)$.
Their coefficient maps are $\lambda_i$ and
$\lambda_i\circ\sigma$, respectively.
This proves that the two squares over the naive stacks $k^{\Nyg_r}$ commute.
For the squares over $k^{\syn}$, both $\mathsf R_r$ and $\mathsf F_r$ apply $\sigma$
to coefficients and preserve the local Nygaard coordinates,
so the chartwise composites are $\Phi_{\Nyg}\circ f_{r,i}$.
These equalities respect the adjacent gluings and glue to the
asserted ones.
\end{proof}

\begin{remark}[{Variants of even stacks of $\mathrm{TR}^r$}]
\label{rem:TR-motivation}
The stacks $k^{\Nyg_r}$ and $k^{\NN_r}$ and the contraction $g_r$ are
modelled on topological restriction homology.  Write $\mathrm{THH}$,
$\mathrm{TC}^-$, $\mathrm{TP}$ and $\mathrm{TR}^r$ for these invariants
of $k$ with $\mathbf Z_p$-coefficients.  The genuine fixed points
$\mathrm{TR}^r=\mathrm{THH}^{C_{p^{r-1}}}$ and their $S^1$-homotopy
fixed points are iterated fibre products
\[
 \mathrm{TR}^r\simeq
 \mathrm{THH}^{hC_{p^{r-1}}}\times_{\mathrm{THH}^{tC_{p^{r-1}}}}\cdots
 \times_{\mathrm{THH}^{tC_p}}\mathrm{THH},
 \qquad
 (\mathrm{TR}^r)^{hS^1}\simeq
 \mathrm{TC}^-\times_{\mathrm{TP}}\cdots\times_{\mathrm{TP}}\mathrm{TC}^-
\]
along the canonical maps and Frobenius, with factors numbered from $0$
to $r-1$ \cite[\S II.4]{NikolausScholze}, \cite[\S2]{AndriopoulosTR}.
Restriction forgets the first factor.  Frobenius forgets the last, and
on $\mathrm{TR}^r$ it also applies the maps
$\mathrm{THH}^{hC_{p^k}}\to\mathrm{THH}^{hC_{p^{k-1}}}$ to the others.
These maps become identities on $S^1$-homotopy fixed points, so there
restriction and Frobenius forget the first and the last factor, as
$\mathsf R_r$ and $\mathsf F_r$ do on the genuine stacks $k^{\NN_r}$
(Definition~\ref{def:level-maps}).

All rings here are even.  For an even $p$-complete
$\mathbf E_\infty$-ring $A$, we use the $p$-complete variant of the even
stack $\mathrm{Spev}(A)$ \cite[\S2.1]{DevalapurkarThesis}, which
geometrizes the even filtration of Hahn, Raksit and Wilson \cite{HRW}.
Directly from the definition, it is $\Spf(\pi_{2*}A)/\Gm$ with
$\pi_{2n}A$ in weight $n$.  Recall that
$\pi_*\mathrm{TC}^-=W[u,t]/(ut-p)$, with $u$ and $t$ in degrees $2$ and
$-2$, and that the canonical map and Frobenius to $\mathrm{TP}$ invert
$t$ and $u$ \cite[\S IV.4]{NikolausScholze}.  The fibre product is computed degreewise \cite[\S3]{AndriopoulosTR}.
In degree $2n\geq0$ a compatible family $(a_iu^n)_i$ satisfies
$p^na_i=\sigma(a_{i+1})$, and in degree $-2n$ a family $(b_it^n)_i$
satisfies $b_i=p^n\sigma(b_{i+1})$.  So
$\pi_*(\mathrm{TR}^r)^{hS^1}\simeq W[u,t]/(ut-p^r)$, where the $i$-th
projection sends $u$ and $t$ to $p^iu$ and $p^{r-1-i}t$ and acts on $W$
by $\sigma^{-i}$, as $g_r$ does on the $i$-th chart
(Proposition~\ref{prop:contraction-map}).  So taking even stacks
before and after the fibre product gives the genuine and the naive
$r$-Nygaard stacks $k^{\NN_r}$ and $k^{\Nyg_r}$, and $g_r$ is the comparison map:
\[
 k^{\NN_r}\simeq\operatorname{coeq}\Bigl(\coprod_{i=0}^{r-2}\mathrm{Spev}(\mathrm{TP})
 \rightrightarrows\coprod_{i=0}^{r-1}\mathrm{Spev}(\mathrm{TC}^-)\Bigr)
 \xrightarrow{\ g_r\ }\mathrm{Spev}\bigl((\mathrm{TR}^r)^{hS^1}\bigr)\simeq k^{\Nyg_r}.
\]
Here the $i$-th copy of $\mathrm{Spev}(\mathrm{TP})$ maps to the $i$-th
and the $(i+1)$-st chart by the canonical map and by Frobenius.

Modulo $t$ one gets $\mathrm{TR}^r=(\mathrm{TR}^r)^{hS^1}/t$, with
$\pi_*\mathrm{TR}^r\simeq W_r[u]$ \cite[\S3]{AndriopoulosTR}, and its
terms $\mathrm{THH}^{hC_{p^j}}\simeq\mathrm{TC}^-/(p^jt)$ and
$\mathrm{THH}^{tC_{p^j}}\simeq\mathrm{TP}/p^j$ have homotopy rings
$W[u,t]/(ut-p,\,p^jt)$ and $W_j[t^{\pm1}]$
\cite[Lemma~3.1]{RiggenbachTrunc}.  Since $g_r$ sends $t$ to
$p^{r-1-i}t_i$ on the $i$-th chart, for $j=r-1-i$ these are the even
stacks of the preimage of $t=0$ in the $i$-th chart and of its overlap
with the next.  So the comparison map for $\mathrm{TR}^r$ is the
restriction of $g_r$ over the locus $t=0$.

Restriction sends $u$ to $pu$ and fixes $t$
\cite[\S3]{AndriopoulosTR}, so at $r=\infty$ the limit
$\mathrm{TR}=\varprojlim_r\mathrm{TR}^r$ has $\pi_*\mathrm{TR}^{hS^1}\simeq W[t]$ and $\pi_*\mathrm{TR}\simeq W$.
The fibre products become infinite, and their even stacks give the
$\infty$-Nygaard stacks of Construction~\ref{con:infinite-nygaard-stacks},
\[
 \begin{aligned}
 k^{\NN_\infty}&\simeq\operatorname{coeq}\Bigl(\coprod_{j\geq1}\mathrm{Spev}(\mathrm{TP})
 \rightrightarrows\coprod_{j\geq0}\mathrm{Spev}(\mathrm{TC}^-)\Bigr)
 \xrightarrow{\ g_\infty\ }\mathrm{Spev}(\mathrm{TR}^{hS^1})\simeq\mathbb A^1_{W,t}/\Gm,\\
 k^{\mathrm{Hodge}_\infty}&\simeq\operatorname{coeq}\Bigl(\coprod_{j\geq1}\mathrm{Spev}\bigl(\mathrm{THH}^{tC_{p^j}}\bigr)
 \rightrightarrows\coprod_{j\geq0}\mathrm{Spev}\bigl(\mathrm{THH}^{hC_{p^j}}\bigr)\Bigr)
 \xrightarrow{\ g'_\infty\ }\mathrm{Spev}(\mathrm{TR})\simeq B\mathbb G_{m,W},
 \end{aligned}
\]
in which the $j$-th overlap maps to the $j$-th and the $(j-1)$-st chart
(Lemma~\ref{lem:naive-infinite-nygaard} and
Remark~\ref{rem:genuine-infinite-nygaard}).

For a smooth $k$-algebra $A$, Hesselholt identified
$\pi_*\mathrm{TR}^r(A;\mathbf Z_p)$ with $W_r\Omega^\bullet_{A/k}$
tensored with a polynomial algebra on a class of degree $2$
\cite[Theorem~B]{Hesselholt96}.  For a smooth $k$-scheme and
$1\leq r\leq\infty$, Andriopoulos identifies the motivic filtrations of
$(\mathrm{TR}^r)^{hS^1}$ and $\mathrm{TR}^r$ with the $r$-Nygaard
filtration and its graded pieces \cite[Theorem~1.5]{AndriopoulosTR}.
Section~\ref{sec:pullback-pushforward} realizes these on $k^{\Nyg_r}$,
and at $r=\infty$ identifies them with the Hodge--Witt filtration
(Theorem~\ref{thm:hodge-witt}).
\end{remark}

\begin{remark}[The syntomic point as a coequalizer]
\label{rem:syntomic-coequalizer}
For every $r\geq1$, the map $f_{r+1}$ exhibits the syntomic point as
the coequalizer of restriction and Frobenius:
\[
 \operatorname{coeq}\Bigl(k^{\NN_r}
 \underset{\mathsf F_r}{\overset{\mathsf R_r}{\rightrightarrows}}
 k^{\NN_{r+1}}\Bigr)\xrightarrow{\ \sim\ }k^{\syn}.
\]
For $r=1$ this reads
$k^{\syn}\simeq\operatorname{coeq}(k^{\Nyg}\rightrightarrows k^{\NN_2})$.
The two charts of $k^{\NN_2}$ are identified, and the gluing between
them becomes the self-gluing of $k^{\Nyg}$ that defines $k^{\syn}$.
This is a stacky analogue of the classical construction of
$\mathrm{TC}$ from the $\mathrm{TR}^r$ by equalizing restriction and
Frobenius \cite[\S II.4]{NikolausScholze}.  On the fibre products of
Remark~\ref{rem:TR-motivation}, restriction and Frobenius forget the
first and the last factor.  So their equalizer identifies all factors
of $(\mathrm{TR}^{r+1})^{hS^1}$, and it is the equalizer of the
canonical map and Frobenius $\mathrm{TC}^-\rightrightarrows\mathrm{TP}$,
that is, $\mathrm{TC}$.  Taking even stacks of the terms turns this
equalizer into the coequalizer above.

One can even extend this to $r=0$, where it becomes the definition of
$k^{\syn}$.  Put
$k^{\NN_0}=k^{\Nyg_0}:=\Spf\bigl(W[u,t]/(ut-1)\bigr)/\Gm\simeq k^\Delta$
and define $\mathsf R_0,\mathsf F_0:k^{\Nyg_0}\to k^{\Nyg}$ by the
formulas of Definition~\ref{def:level-maps}.  Then
$\mathsf R_0=\widetilde j_{\dR}$ and
$\mathsf F_0=j_{\HT}\circ\varphi$, so their coequalizer is
$k^{\syn}$, with quotient map $f_1=j_{\Nyg}$.  Correspondingly,
$\pi_*\mathrm{TP}=W[t^{\pm1}]$ gives $\mathrm{Spev}(\mathrm{TP})=k^\Delta$,
and the equalizer of $\mathrm{TC}^-\rightrightarrows\mathrm{TP}$ is the
case $r=0$ of the equalizer above.  So the coequalizer of restriction
and Frobenius is the same at every level.

To prove the claim, recall that $k^{\NN_{r+1}}$ is the colimit of its
charts glued along their overlaps.  On each chart and overlap of
$k^{\NN_r}$, the maps $\mathsf R_r$ and $\mathsf F_r$ are isomorphisms
onto the $(i+1)$-st and the $i$-th piece of $k^{\NN_{r+1}}$, and
$\mathsf F_r\circ\mathsf R_r^{-1}$ is the identity in coordinates.  As
colimits commute, the coequalizer therefore glues all charts of
$k^{\NN_{r+1}}$ into one copy of $k^{\Nyg}$, and all transition maps
$\alpha_i$ into one self-gluing
$a\mapsto\sigma(a)$, $u\mapsto t^{-1}$, $t\mapsto pt$ from the de Rham
to the Hodge--Tate locus.  By the proof of
Proposition~\ref{prop:syntomic-projection}, this is the gluing that
defines $k^{\syn}$, and the resulting equivalence is induced by
$f_{r+1}$, which coequalizes $\mathsf R_r$ and $\mathsf F_r$ by
Proposition~\ref{prop:level-maps-compatible}.
\end{remark}

\subsection{The \texorpdfstring{$r$}{r}-Nygaard comparison}
\label{sec:pullback-pushforward}
\label{sec:r-nygaard-specializations}
\label{sec:r-nygaard-comparison}
\label{sec:iterated-r-nygaard}

In this subsection we use the correspondence of
\S\ref{sec:genuine-naive} to realize the $r$-Nygaard filtration
from a given $F$-gauge, via $E\mapsto g_{r,*}f_r^*E$.

We begin with the pullback along $f_r$, which is straightforward.

\begin{construction}[Formula for $f_r^*$]\label{con:f-pullback}
Let $E$ be an $F$-gauge.  Its pullback $f_r^*E$ is the genuine
$r$-gauge whose $i$-th chart component is
\[
 E_i:=f_{r,i}^*E,
\]
the underlying gauge of $E$ with its coefficients twisted to
$W^{(i)}$, and whose gluing maps
$\delta_i:E_i|_{t_i\neq0}\xrightarrow{\sim}\alpha_i^*(E_{i+1}|_{u_{i+1}\neq0})$
are the descent datum of $E$ pulled back along the common map
$(k^{\Nyg}_i)_{t_i\neq0}\to k^{\syn}$ of
Proposition~\ref{prop:syntomic-projection}.
\end{construction}

So $f_r^*E$ consists of $r$ copies of the gauge $E$, each with its
own Frobenius twist, in which the de Rham pullback
$E_i^{-\infty}$ of each copy is identified with the Hodge--Tate
pullback $E_{i+1}^{\infty}$ of the next by the descent datum of
$E$.  For $r=3$ the picture is
\[
\begin{tikzcd}[column sep=small, row sep=2.4em, ampersand replacement=\&]
 E_0^{\infty} \& \cdots \& E_0^{n+1} \arrow[r, shift left=0.6ex, "t"] \& E_0^{n} \arrow[l, shift left=0.6ex, "u"] \arrow[r, shift left=0.6ex, "t"] \& E_0^{n-1} \arrow[l, shift left=0.6ex, "u"] \arrow[r, shift left=0.6ex, "t"] \& E_0^{n-2} \arrow[l, shift left=0.6ex, "u"] \& \cdots \& E_0^{-\infty} \arrow[dlllllll, "\delta_0"', "\sim"] \\
 E_1^{\infty} \& \cdots \& E_1^{n+1} \arrow[r, shift left=0.6ex, "t"] \& E_1^{n} \arrow[l, shift left=0.6ex, "u"] \arrow[r, shift left=0.6ex, "t"] \& E_1^{n-1} \arrow[l, shift left=0.6ex, "u"] \arrow[r, shift left=0.6ex, "t"] \& E_1^{n-2} \arrow[l, shift left=0.6ex, "u"] \& \cdots \& E_1^{-\infty} \arrow[dlllllll, "\delta_1"', "\sim"] \\
 E_2^{\infty} \& \cdots \& E_2^{n+1} \arrow[r, shift left=0.6ex, "t"] \& E_2^{n} \arrow[l, shift left=0.6ex, "u"] \arrow[r, shift left=0.6ex, "t"] \& E_2^{n-1} \arrow[l, shift left=0.6ex, "u"] \arrow[r, shift left=0.6ex, "t"] \& E_2^{n-2} \arrow[l, shift left=0.6ex, "u"] \& \cdots \& E_2^{-\infty}
\end{tikzcd}
\]
Here each row is the gauge $E_i=f_{r,i}^*E$, which is $E$ read
through $W^{(i)}$.  The unglued terms $E_0^{\infty}$ and $E_{r-1}^{-\infty}$ are the restrictions
of $f_r^*E$ to the open loci $(k^{\Nyg}_0)_{u_0\neq0}$ and
$(k^{\Nyg}_{r-1})_{t_{r-1}\neq0}$ of the chain.

\begin{proposition}[\v{C}ech formula for $g_{r,*}$]
\label{prop:cech-formula-for-Rgr}
Let $\mathcal E=(E_i,\delta_i)$ be a genuine $r$-gauge.  Then
\begin{equation}\label{eq:g-pushforward}
 g_{r,*}\mathcal E\simeq
 \fib\Bigl(
 \bigoplus_{i=0}^{r-1}E_i
 \xrightarrow{\ \partial\ }
 \bigoplus_{i=0}^{r-2}E_i|_{t_i\neq0}
 \Bigr),
 \qquad
 \partial((m_i)_i)
 =\bigl(m_i|_{t_i\neq0}-\delta_i^{-1}\alpha_i^*
 (m_{i+1}|_{u_{i+1}\neq0})\bigr)_{i=0}^{r-2},
\end{equation}
where each summand is a module over the coordinate ring of
$k^{\Nyg_r}$ by restriction of scalars along $g_r$
(Proposition~\ref{prop:contraction-map}).
\end{proposition}

\begin{proof}
Let $j_i:k^{\Nyg}_i\hookrightarrow k^{\NN_r}$ and
$j_{i,i+1}:k^{\Nyg}_i\cap k^{\Nyg}_{i+1}\hookrightarrow k^{\NN_r}$ be
the open immersions.  Since only adjacent charts meet and there are
no triple intersections, Zariski descent gives
\[
 \mathcal E\simeq
 \fib\Bigl(
 \bigoplus_{i=0}^{r-1}j_{i,*}E_i
 \xrightarrow{\ \partial\ }
 \bigoplus_{i=0}^{r-2}j_{i,i+1,*}\bigl(E_i|_{t_i\neq0}\bigr)
 \Bigr),
\]
where the $i$-th component of $\partial$ is the restriction from the
left chart minus the restriction from the right chart identified
through $\delta_i^{-1}\alpha_i^*$, as in \eqref{eq:g-pushforward}.
The functor $g_{r,*}$ is exact, so it preserves this fibre.  The
composite $g_r\circ j_i$ is $g_{r,i}$, and $g_r\circ j_{i,i+1}$ is its
restriction to the overlap.  Both are induced by homogeneous maps of
$p$-complete graded rings (Proposition~\ref{prop:contraction-map}),
so they are affine and their pushforwards are restriction of
scalars.  Hence $g_{r,*}j_{i,*}E_i$ is $E_i$ regarded as a module
over the coordinate ring of $k^{\Nyg_r}$, and likewise for the
overlap terms, which gives \eqref{eq:g-pushforward}.
\end{proof}

In each weight the formula is an $r$-fold fibre product.  Since
$t_i$ is invertible on the overlap, the weight-$n$ component of
$E_i|_{t_i\neq0}$ is the de Rham pullback $E_i^{-\infty}$, and
through $\alpha_i$ and $\delta_i$ it is also the Hodge--Tate pullback
$E_{i+1}^{\infty}$.  Hence
\[
 (g_{r,*}\mathcal E)^n\simeq
 E_0^n\times_{E_0^{-\infty}\simeq E_1^{\infty}}E_1^n
 \times_{E_1^{-\infty}\simeq E_2^{\infty}}\cdots
 \times_{E_{r-2}^{-\infty}\simeq E_{r-1}^{\infty}}E_{r-1}^n,
\]
where $E_i^n\to E_i^{-\infty}$ is the structure map along $t$ and
$E_{i+1}^n\to E_{i+1}^{\infty}$ is the structure map along $u$.  In
the picture after Construction~\ref{con:f-pullback}, this is the
fibre product of the weight-$n$ column, glued along the diagonal
isomorphisms $\delta_i$.

We now apply this calculation to the $F$-gauge cohomology
$\cH_{\syn}(X)$ of a smooth quasi-compact $k$-scheme $X$
(Definition~\ref{def:syntomification}).

\begin{definition}[The $r$-Nygaard realization]
\label{def:r-nygaard-realization}
For a smooth quasi-compact $k$-scheme $X$, the \emph{$r$-Nygaard
realization} of $X$ is the $r$-gauge
\[
 \cH_{\Nyg_r}(X):=g_{r,*}f_r^*\cH_{\syn}(X)
 \in\CAlg\bigl(\Dqc(k^{\Nyg_r})\bigr).
\]
At $r=1$ it is $\cH_{\Nyg}(X)$.
\end{definition}

The map $f_r$ is \'etale, since on each chart it is $j_{\Nyg}$
composed with the isomorphism $f_{r,i}$.  Base change along it
gives $f_r^*\cH_{\syn}(X)\simeq\pi_{X,*}\cO_{X^{\NN_r}}$, so the
$r$-Nygaard realization is the cohomology of the genuine
$r$-Nygaard prismatization pushed forward to the naive stack
$k^{\Nyg_r}$,
$\cH_{\Nyg_r}(X)\simeq g_{r,*}\pi_{X,*}\cO_{X^{\NN_r}}$.

\begin{proposition}[The $r$-Nygaard comparison]
\label{prop:r-nygaard-comparison}
Let $X$ be a smooth quasi-compact $k$-scheme.  The $m$-th filtered piece of
$\cH_{\Nyg_r}(X)$ is the $r$-fold fibre product
\[
 \cH_{\Nyg_r}(X)^m
 \simeq
 \Fil^{\geq m}_{\Nyg}R\Gamma_{\crys}
 \times_{R\Gamma_{\crys}}\cdots\times_{R\Gamma_{\crys}}
 \Fil^{\geq m}_{\Nyg}R\Gamma_{\crys}
\]
of Definition~\ref{def:r-nygaard-filtration}, with the $i$-th factor
read through $W^{(i)}$.  After untwisting coefficients it is the
$m$-th piece of the $r$-Nygaard filtration.  The coordinates act on
the $i$-th factor by
\begin{equation}\label{eq:general-r-nygaard-actions}
 u|_i=p^iu_i,
 \qquad
 t|_i=p^{r-1-i}t_i .
\end{equation}
\end{proposition}

\begin{proof}
By Construction~\ref{con:f-pullback}, the $i$-th row of
$f_r^*\cH_{\syn}(X)$ is $\cH_{\Nyg}(X)$ read through $W^{(i)}$.  By
Theorem~\ref{thm:ordinary-input}(1)--(3), its $m$-th weight is
$\Fil^{\geq m}_{\Nyg}R\Gamma_{\crys}$, and its maps to the de Rham and
Hodge--Tate pullbacks are $\iota_m$ and $\varphi_m$.  The weightwise form of
Proposition~\ref{prop:cech-formula-for-Rgr} therefore gives the
fibre product, and Proposition~\ref{prop:contraction-map} gives the
actions.  All maps are maps of commutative algebra objects in graded
complexes, so the identifications of the filtered pieces are compatible
with products.
\end{proof}

\begin{remark}[Structure sheaf of the contraction]
\label{rem:Rg-structure-sheaf}
For every $r\geq1$, $g_{r,*}\cO_{k^{\NN_r}}\simeq\cO_{k^{\Nyg_r}}$.
Indeed, for $X=\Spec k$ we have $\cH_{\syn}(X)=\cO_{k^{\syn}}$ and
$f_r^*\cO_{k^{\syn}}=\cO_{k^{\NN_r}}$, so
Proposition~\ref{prop:r-nygaard-comparison} identifies the $m$-th
filtered piece of $g_{r,*}\cO_{k^{\NN_r}}$ with the $m$-th piece of
the $r$-Nygaard filtration on $W$.  By
Example~\ref{ex:r-nygaard-on-W} this is $p^{rm}W$ for $m\geq0$ and
$W$ for $m\leq0$, with $u$ acting as $p^r$ and $t$ as the inclusion.
That is the coordinate ring of $k^{\Nyg_r}$, and the Frobenius
twists along the chain do not change it.
\end{remark}

We now identify the specializations of the $r$-Nygaard realization
$\cH_{\Nyg_r}(X)$, as Theorem~\ref{thm:ordinary-input} does for
$\cH_{\Nyg}(X)$.  We begin with the open loci
$(k^{\Nyg_r})_{u\neq0}$ and $(k^{\Nyg_r})_{t\neq0}$, over which the
contraction is an isomorphism.

\begin{lemma}[The contraction over the open loci]
\label{lem:contraction-open-loci}
The preimages under $g_r$ of $(k^{\Nyg_r})_{u\neq0}$ and
$(k^{\Nyg_r})_{t\neq0}$ are $(k^{\Nyg}_0)_{u_0\neq0}$ and
$(k^{\Nyg}_{r-1})_{t_{r-1}\neq0}$, and $g_r$ restricts to
isomorphisms
\[
 (k^{\Nyg}_0)_{u_0\neq0}\xrightarrow{\ \sim\ }(k^{\Nyg_r})_{u\neq0},
 \qquad
 (k^{\Nyg}_{r-1})_{t_{r-1}\neq0}\xrightarrow{\ \sim\ }(k^{\Nyg_r})_{t\neq0},
\]
given on coordinates by $a\mapsto a^{(0)}$, $u\mapsto u_0$, and by
$a\mapsto\sigma^{1-r}(a)^{(r-1)}$, $t\mapsto t_{r-1}$.
\end{lemma}

\begin{proof}
On the $i$-th chart $k^{\Nyg}_i$ of $k^{\NN_r}$, the coordinate $u$ pulls back to $p^iu_i$
(Proposition~\ref{prop:contraction-map}).  Since $p$ is
topologically nilpotent, this is invertible only if $i=0$ and $u_0$
is invertible.  Likewise $t=p^{r-1-i}t_i$ is invertible only on
$(k^{\Nyg}_{r-1})_{t_{r-1}\neq0}$.  On these loci the formulas of
Proposition~\ref{prop:contraction-map} reduce to the stated ones,
which are isomorphisms of the graded coordinate rings
$W[u^{\pm1}]\to W^{(0)}[u_0^{\pm1}]$ and
$W[t^{\pm1}]\to W^{(r-1)}[t_{r-1}^{\pm1}]$.
\end{proof}

Although the lemma is a statement about the map of stacks $g_r$, it
computes the pushforward of every genuine $r$-gauge over the open
loci.  Indeed, by flat base change along the open immersions, for
every genuine $r$-gauge $\mathcal E=(E_i,\delta_i)$ we have
\[
 (g_{r,*}\mathcal E)|_{u\neq0}\simeq E_0|_{u_0\neq0},
 \qquad
 (g_{r,*}\mathcal E)|_{t\neq0}\simeq E_{r-1}|_{t_{r-1}\neq0},
\]
read through the coordinate maps of the lemma.  In the weightwise
fibre product after Proposition~\ref{prop:cech-formula-for-Rgr}, only
the Hodge--Tate pullback $E_0^{\infty}$ of the first row and the
de Rham pullback $E_{r-1}^{-\infty}$ of the last row survive.

\begin{remark}[The interior open loci of the chain]
\label{rem:interior-open-loci}
The chain $k^{\NN_r}$ has $r-1$ more open loci isomorphic to
$k^\Delta$ besides the two of
Lemma~\ref{lem:contraction-open-loci}.  These are the overlaps
$k^{\Nyg}_i\cap k^{\Nyg}_{i+1}$.  On the $i$-th overlap we have
$u_i=pt_i^{-1}$, so both $g_r^*u=p^{i+1}t_i^{-1}$ and
$g_r^*t=p^{r-1-i}t_i$ are divisible by $p$.  Hence $g_r$ sends this
overlap to the $W$-point $(u,t)=(p^{i+1},p^{r-1-i})$ of
$k^{\Nyg_r}$.  This point reduces to the origin modulo $p$ but lies
in neither divisor.  In the naive stack $k^{\Nyg_r}$ all these factorizations of
$p^r$ lie over its $A_{r-1}$-singularity, and the genuine stack $k^{\NN_r}$
separates them into open loci.
\end{remark}

The two open loci cover $k^{\Nyg_r}$ away from the origin
$V(u,t)$.  As at level one (Construction~\ref{con:closed-loci}), we
study the origin through the two divisors $u=0$ and $t=0$ that meet
there.

\begin{construction}[The Hodge-filtered and conjugate-filtered $r$-de Rham--Witt pullbacks]
\label{con:dRW-pullbacks}
For $r\geq1$, the zero loci of $u$ and $t$ in $k^{\Nyg_r}$ are the
Cartier divisors
\[
 i_{\dRW_r,+}:\mathbb A^1_{W_r,t}/\Gm=(k^{\Nyg_r})_{u=0}
 \lhook\joinrel\longrightarrow k^{\Nyg_r},
 \qquad
 i_{\dRW_r,c}:\mathbb A^1_{W_r,u}/\Gm=(k^{\Nyg_r})_{t=0}
 \lhook\joinrel\longrightarrow k^{\Nyg_r},
\]
with $|t|=-1$ and $|u|=1$, and their intersection is the
origin $i_{\mathrm{Hodge}_r}:B\mathbb G_{m,W_r}=V(u,t)\hookrightarrow k^{\Nyg_r}$.
The two divisors are Tor-independent: $(u,t)$ is a regular sequence
in the coordinate ring of $k^{\Nyg_r}$, with quotients $W_r[t]$ and
$W_r[u]$ whose derived tensor product is $W_r$, so their fibre
product is the classical quotient stack $B\mathbb G_{m,W_r}$.  For
an $r$-gauge $E$ we call $i_{\dRW_r,+}^*E$ its \emph{Hodge-filtered
$r$-de Rham--Witt pullback}, $i_{\dRW_r,c}^*E$ its
\emph{conjugate-filtered $r$-de Rham--Witt pullback}, and
$i_{\mathrm{Hodge}_r}^*E$ its \emph{$r$-Hodge pullback}.
\end{construction}

To compute these pullbacks of $\cH_{\Nyg_r}(X)$ we use that,
\'etale locally on $X$, its $r$-Nygaard filtration is the
Beilinson-connective cover of the $p^r$-adic filtration, by
Theorem~\ref{thm:andriopoulos}.  Since $X$ is not affine, we
therefore work at the level of sheaves on $X_{\et}$.  Let
$\mathcal C_{X/W}\in D(X_{\et},W)$ be the sheafified crystalline
complex, represented on an \'etale affine $\Spec S\to X$ by the
$p$-torsion-free complex $W\Omega^\bullet_{S/k}$, and write
$p^{r\bullet}\mathcal C_{X/W}$ for its $p^r$-adic filtration.

\begin{definition}[Filtered d\'ecalage]
\label{def:filtered-decalage}
A filtered object $F^\bullet$ of $D(X_{\et},W)$ is
\emph{Beilinson-connective} if $\gr^iF^\bullet\in D^{\leq i}$ for all
$i\in\mathbb Z$.  The \emph{filtered d\'ecalage} of $\mathcal C_{X/W}$
is the Beilinson-connective cover of its $p^r$-adic filtration,
\[
 \cN^\bullet_{p^r}(\mathcal C_{X/W})
 :={}^{B}\!\tau_{\leq0}\bigl(p^{r\bullet}\mathcal C_{X/W}\bigr).
\]
Multiplication by $p^r$ and the filtration inclusions give maps $u$
and $t$ with $ut=tu=p^r$, which make
$\Rees\bigl(\cN^\bullet_{p^r}(\mathcal C_{X/W})\bigr)$ a graded module
over $W[u,t]/(ut-p^r)$, with $|u|=1$ and $|t|=-1$.
\end{definition}

By the description of Beilinson truncation on associated graded
pieces,
\[
 \gr^i\cN^\bullet_{p^r}(\mathcal C_{X/W})
 \simeq
 \tau^{\leq i}\bigl(\mathcal C_{X/W}/p^r\bigr),
\]
and the underlying unfiltered object is the d\'ecalage
$L\eta_{p^r}\mathcal C_{X/W}$ of \cite[\S6]{BMS18}.

\begin{corollary}[Global comparison with filtered d\'ecalage]
\label{cor:bridge}
Let $X/k$ be smooth and quasi-compact.  For every $m$ there is a canonical
equivalence
\[
 \cH_{\Nyg_r}(X)^m
 \;\simeq\;
 (\sigma^{-1})^*R\Gamma\bigl(X,\,\cN^m_{p^r}(\mathcal C_{X/W})\bigr),
\]
and these assemble into an equivalence of graded algebra objects
\[
 \cH_{\Nyg_r}(X)
 \;\simeq\;
 (\sigma^{-1})^*R\Gamma\bigl(X,\,\Rees\bigl(\cN^\bullet_{p^r}(\mathcal C_{X/W})\bigr)\bigr)
\]
in $\Dqc(k^{\Nyg_r})$, compatible with \eqref{eq:general-r-nygaard-actions}.
\end{corollary}
\begin{proof}
For an \'etale map $U\to X$, let $\cN^m_r(U)$ be the iterated fibre
product of Definition~\ref{def:r-nygaard-filtration} for $U$, with
the $i$-th factor read through $W^{(i)}$.  The Nygaard filtration
satisfies \'etale descent, since $\cH_{\Nyg}(U)=\pi_{U,*}\cO_{U^{\Nyg}}$
does, and $R\Gamma$ commutes with finite limits.  So
$U\mapsto\cN^m_r(U)$ is an \'etale sheaf, whose value on $X$ is
$\cH_{\Nyg_r}(X)^m$ by Proposition~\ref{prop:r-nygaard-comparison}.

On an \'etale affine $U=\Spec S$, Theorem~\ref{thm:andriopoulos}
identifies $\cN^m_r(U)$, through the rightmost factor and
$\varphi^r$, with the $m$-th piece of the Beilinson-connective cover
of the $p^r$-adic filtration on $R\Gamma_{\crys}(S/W)$.  Since the
$W\Omega^n_{S/k}$ have no higher cohomology on $U$, this filtration is
$R\Gamma(U,p^{r\bullet}\mathcal C_{X/W})$.  The cohomology sheaves of
$\mathcal C_{X/W}/p^r$ are the quasi-coherent
$W_r\Omega^q_{X/k}$ \cite[Proposition~2.7.1]{BLM21}, which have no
higher cohomology on $U$ either.  So $R\Gamma(U,-)$ commutes with the
truncations $\tau^{\leq i}(\mathcal C_{X/W}/p^r)$, which are the graded
pieces of the Beilinson-connective cover, and hence with the cover
itself.  Thus the cover is
$R\Gamma\bigl(U,\cN^m_{p^r}(\mathcal C_{X/W})\bigr)$.  The
identification is functorial in $S$, and \'etale affines form a basis
of $X_{\et}$, so the two sheaves agree, and taking sections over $X$
gives the first equivalence.  The Frobenius twist $(\sigma^{-1})^*$
comes from reading the rightmost factor through $\sigma^{-(r-1)}$ and
the $\sigma^r$-semilinearity of $\varphi^r$.

On the rightmost factor the actions
\eqref{eq:general-r-nygaard-actions} are $u=p^{r-1}\cdot p=p^r$ and
the inclusion $t$, which are those on
$\Rees\bigl(\cN^\bullet_{p^r}(\mathcal C_{X/W})\bigr)$.  Since
$\varphi^r$ is multiplicative, the equivalences assemble into the
asserted equivalence of graded algebra objects.
\end{proof}

\begin{definition}[The Bockstein complex]
\label{def:f-Bockstein-complex}
The \emph{Bockstein complex} $\Bock_{p^r}(\mathcal C_{X/W})$ is the
cochain complex of sheaves of $W_r$-modules with terms
$\mathcal H^q(\mathcal C_{X/W}/p^r)$ and differential the connecting
map $\beta$ of $0\to W_r\xrightarrow{p^r}W_{2r}\to W_r\to0$.  If
$x$ lifts a cocycle modulo $p^r$ and $dx=p^ry$, then
$\beta([\bar x])=[\bar y]$.  We give it the stupid, or Hodge,
filtration $\Fil_H^i:=\Bock_{p^r}(\mathcal C_{X/W})^{\geq i}$.
\end{definition}

The Bockstein complex is the de Rham--Witt complex.

\begin{proposition}[Crystalline Bockstein--de Rham--Witt comparison]
\label{prop:crystalline-Bockstein-dRW}
For every $r\geq1$, there is a canonical multiplicative equivalence of
cochain complexes on $X_{\et}$
\[
 \Bock_{p^r}(\mathcal C_{X/W})
 \xrightarrow{\ \sim\ }
 W_r\Omega^\bullet_{X/k}.
\]
It is a $W_r$-linear equivalence onto
$(\sigma^r)^*W_r\Omega^\bullet_{X/k}$, and it carries $\beta$ to the
de Rham--Witt differential.
\end{proposition}

\begin{proof}
The assertion is \'etale local, so let $\Spec S\to X$ be an \'etale
affine.  Since $S$ is smooth over $k$, the de Rham--Witt complex
$M:=W\Omega^\bullet_{S/k}$, which represents $\mathcal C_{X/W}$ on
$\Spec S$, is the saturated de Rham--Witt complex of $S$
\cite[Theorem~4.4.12]{BLM21}.  So $M$ is $p$-torsion-free, $F$ is
injective with $F(M)=\{x:dx\in pM\}$, and
$W_r\Omega^\bullet_{S/k}=M/(\operatorname{im}V^r+\operatorname{im}dV^r)$.
By induction on $r$, $F^r(M)=\{x:dx\in p^rM\}$: if $dx=p^ry$, then
$dy=0$, so $x=Fx'$ and $y=Fy'$, and $dx'=p^{r-1}y'$.  We now follow
\cite[\S2.7]{BLM21}.  The \emph{higher Cartier map}
\[
 C_r^{-1}:W_r\Omega^i_{S/k}\longrightarrow
 H^i\bigl(W\Omega^\bullet_{S/k}/p^r\bigr),
 \qquad
 C_r^{-1}(\omega)=[F^r\widetilde\omega],
\]
where $\widetilde\omega$ is any lift of $\omega$, is well defined since
$F^rV^r=p^r$ and $F^rdV^r=d$, and surjective since the cycles of
$M/p^r$ are represented by $F^r(M)$.  It is injective: if
$F^ry=p^rw+dz$, then applying $V^r$ and using $V^rd=p^r\,dV^r$ gives
$y=V^rw+dV^rz$.  So $C_r^{-1}$ is a $\sigma^r$-semilinear isomorphism.
Since $dF^r=p^rF^rd$, the Bockstein of the class of
$F^r\widetilde\omega$ is the class of $F^r(d\widetilde\omega)$, so
$C_r^{-1}$ carries $d$ to $\beta$.  It is multiplicative because $F$
is, and it is compatible with \'etale localization, so it sheafifies.
Its inverse is the asserted equivalence.
\end{proof}

We can now compute the fibres of the filtered d\'ecalage on the two
divisors.

\begin{proposition}[Fibres of the sheafified $r$-Nygaard Rees object]
\label{prop:Rees-Bockstein}
\label{prop:sheaf-fibres}
Let $X/k$ be smooth.  There are natural multiplicative equivalences
of graded complexes on $X_{\et}$
\[
 \begin{aligned}
 \Rees\bigl(\cN^\bullet_{p^r}(\mathcal C_{X/W})\bigr)|_{u=0}
 &\simeq(\sigma^r)^*\Rees(\Fil_H^\bullet W_r\Omega^\bullet_{X/k}),\\
 \Rees\bigl(\cN^\bullet_{p^r}(\mathcal C_{X/W})\bigr)|_{t=0}
 &\simeq\Rees\bigl(\tau^{\leq\bullet}(\mathcal C_{X/W}/p^r)\bigr),
 \end{aligned}
\]
where both restrictions are derived.  The first is the decreasing
Hodge filtration, with $i$-th filtered piece
$(\sigma^r)^*W_r\Omega^{\geq i}_{X/k}$.  The second is the increasing
canonical filtration, with $i$-th filtered piece
$\tau^{\leq i}(\mathcal C_{X/W}/p^r)$ and $u$ acting by the inclusions
$\tau^{\leq i-1}\to\tau^{\leq i}$.  In both cases the $i$-th graded
piece is $(\sigma^r)^*W_r\Omega^i_{X/k}[-i]$.
\end{proposition}

\begin{proof}
Write $\cN^\bullet=\cN^\bullet_{p^r}(\mathcal C_{X/W})$.  Since $u$
and $t$ are non-zero-divisors in $W[u,t]/(ut-p^r)$, the derived
fibres at $u=0$ and $t=0$ are the cofibres of $u$ and of $t$.

For $u=0$, the cofibres
$\mathcal H^i_u:=\operatorname{cofib}(u:\cN^{i-1}\to\cN^i)$ form a
filtration through the maps induced by $t$, complete because
$\cN^\bullet$ is.  Its $i$-th graded piece is the cofibre of the map
$\gr^{i-1}\cN^\bullet\to\gr^i\cN^\bullet$ induced by $u$.  Before
truncation, $u=p^r$ identifies successive graded pieces of the
$p^r$-adic filtration, so this map is the truncation map
$\tau^{\leq i-1}(\mathcal C_{X/W}/p^r)\to\tau^{\leq i}(\mathcal C_{X/W}/p^r)$.
Hence $\gr^i\mathcal H^\bullet_u\simeq\mathcal H^i(\mathcal C_{X/W}/p^r)[-i]$,
so $\mathcal H^\bullet_u$ lies in the heart of the Beilinson
$t$-structure.  This heart is equivalent to the category of cochain
complexes: the complex has terms $\mathcal H^i(\mathcal C_{X/W}/p^r)$,
its differential is given by the boundary maps $\gr^i\to\gr^{i+1}[1]$,
and the filtration is the stupid filtration
\cite[Theorem~3.21, Definition~3.24 and Example~3.23]{AntieauDecalage}.
Since $u=p^r$ identifies the $p^r$-adic filtration with its shift,
the cofibre of $u$ is the heart truncation of the $p^r$-adic
filtration, that is, its $E_1$-page \cite[\S4]{AntieauDecalage}.  So
the differential is the $d_1$ of the $p^r$-adic filtration, which is
the Bockstein $\beta$ (compare \cite[Proposition~6.12]{BMS18} for
$r=1$).  Thus
$\Rees\bigl(\cN^\bullet_{p^r}(\mathcal C_{X/W})\bigr)|_{u=0}\simeq\Rees(\Fil_H^\bullet \Bock_{p^r}(\mathcal C_{X/W}))$,
and Proposition~\ref{prop:crystalline-Bockstein-dRW} identifies this
with $(\sigma^r)^*\Rees(\Fil_H^\bullet W_r\Omega^\bullet_{X/k})$.

For $t=0$, the $i$-th term of the cofibre of $t$ is
$\gr^i\cN^\bullet\simeq\tau^{\leq i}(\mathcal C_{X/W}/p^r)$, and $u$
acts by the truncation maps identified above.  Its graded pieces are
$\mathcal H^i(\mathcal C_{X/W}/p^r)[-i]$, which
Proposition~\ref{prop:crystalline-Bockstein-dRW} identifies with
$(\sigma^r)^*W_r\Omega^i_{X/k}[-i]$.

All maps above are multiplicative, and the $E_1$-page is a
commutative algebra in graded cochain complexes
\cite[Example~8.11]{AntieauDecalage}.
\end{proof}

We now pass to global sections.  Corollary~\ref{cor:bridge}
identifies $\cH_{\Nyg_r}(X)$ with the $r$-Nygaard filtration, and
Proposition~\ref{prop:sheaf-fibres} computes its restrictions to the
divisors $u=0$ and $t=0$ sheafwise.  Together they give the level-$r$
analogue of Theorem~\ref{thm:ordinary-input}.  Each locus of the naive
$r$-Nygaard stack $k^{\Nyg_r}$ recovers a classical invariant of $X$,
now with $W_r\Omega^\bullet_{X/k}$ in place of $\Omega^\bullet_{X/k}$.
In particular, the two divisors give the Hodge and the conjugate
filtration, which explains the names of the pullbacks in
Construction~\ref{con:dRW-pullbacks}.  Their intersection, the origin
$u=t=0$, gives the cohomology of the sheaves $W_r\Omega^i_{X/k}$, and
it is the starting point of \S\ref{sec:hodge-stack}.

\begin{theorem}[Specializations of $\cH_{\Nyg_r}(X)$]
\label{thm:r-nygaard-specializations}
Let $X$ be a smooth quasi-compact $k$-scheme and $r\geq1$.
\begin{enumerate}
\item $\cH_{\Nyg_r}(X)$ is the Rees object of the $r$-Nygaard
 filtration, with $m$-th filtered piece
 \[
  (\sigma^{-1})^*R\Gamma\bigl(X,\cN^m_{p^r}(\mathcal C_{X/W})\bigr).
 \]
 After coefficient untwisting this is
 $\Fil^{\geq m}_{\Nyg_r}R\Gamma_{\crys}(X/W)$.
\item On $(k^{\Nyg_r})_{t\neq0}$, the de Rham pullback is
 $\cH_{\Nyg_r}(X)^{-\infty}\simeq(\sigma^{r-1})^*R\Gamma_{\crys}(X/W)$.
\item On $(k^{\Nyg_r})_{u\neq0}$, the Hodge--Tate pullback is
 $\cH_{\Nyg_r}(X)^{\infty}\simeq(\sigma^{-1})^*R\Gamma_{\crys}(X/W)=R\Gamma_\Delta(X/W)$.
\item On $(k^{\Nyg_r})_{u=0}=\mathbb A^1_{W_r,t}/\Gm$, the
 Hodge-filtered $r$-de Rham--Witt pullback is
 \[
  i_{\dRW_r,+}^*\cH_{\Nyg_r}(X)
  \simeq(\sigma^{r-1})^*\Rees\bigl(\Fil_H^\bullet R\Gamma(X,W_r\Omega^\bullet_{X/k})\bigr),
 \]
 where $\Fil_H^aR\Gamma(X,W_r\Omega^\bullet_{X/k})
 =R\Gamma(X,W_r\Omega^{\geq a}_{X/k})$ is the Hodge
 filtration.
\item On $(k^{\Nyg_r})_{t=0}=\mathbb A^1_{W_r,u}/\Gm$, the
 conjugate-filtered $r$-de Rham--Witt pullback is
 \[
  i_{\dRW_r,c}^*\cH_{\Nyg_r}(X)
  \simeq(\sigma^{-1})^*\Rees\bigl(\Fil^{\mathrm{conj}}_\bullet R\Gamma_{\crys}(X/W_r)\bigr),
 \]
 where $\Fil^{\mathrm{conj}}_iR\Gamma_{\crys}(X/W_r)
 :=R\Gamma(X,\tau^{\leq i}(\mathcal C_{X/W}/p^r))$ is the conjugate
 filtration, with
 $\gr_i\simeq(\sigma^r)^*R\Gamma(X,W_r\Omega^i_{X/k})[-i]$.  So the
 $i$-th graded piece of the pullback is
 \[
  (\sigma^{r-1})^*R\Gamma(X,W_r\Omega^i_{X/k})[-i]
 \]
 (compare \cite[Corollary~5.15]{AndriopoulosTR}).
\item On $(k^{\Nyg_r})_{u=t=0}=B\mathbb G_{m,W_r}$, the $r$-Hodge
 pullback has weight-$i$ piece
 $(i_{\mathrm{Hodge}_r}^*\cH_{\Nyg_r}(X))^i
 \simeq(\sigma^{r-1})^*R\Gamma(X,W_r\Omega^i_{X/k})[-i]$.
\end{enumerate}
At $r=1$ these are the statements of Theorem~\ref{thm:ordinary-input}.
\end{theorem}

\begin{proof}
Assertion (1) is Proposition~\ref{prop:r-nygaard-comparison}
together with Corollary~\ref{cor:bridge}.

For (2) and (3), apply the formula after
Lemma~\ref{lem:contraction-open-loci} to $\mathcal E=f_r^*\cH_{\syn}(X)$.
The restriction of $\cH_{\Nyg_r}(X)$ to $(k^{\Nyg_r})_{t\neq0}$ is the
restriction of $f_{r,r-1}^*\cH_{\Nyg}(X)$ to
$(k^{\Nyg}_{r-1})_{t_{r-1}\neq0}$, read through the coordinate map
$a\mapsto\sigma^{1-r}(a)^{(r-1)}$.  By
Theorem~\ref{thm:ordinary-input}(2) this is $R\Gamma_{\crys}(X/W)$
with its coefficients twisted by $\sigma^{1-r}$.  On
$(k^{\Nyg_r})_{u\neq0}$ the chart is $k^{\Nyg}_0$, the coordinate map
is $a\mapsto a^{(0)}$, and Theorem~\ref{thm:ordinary-input}(3) gives
$(\sigma^{-1})^*R\Gamma_{\crys}(X/W)$.

For (4) and (5), assertion (1) gives
$\cH_{\Nyg_r}(X)\simeq(\sigma^{-1})^*R\Gamma(X,\Rees\bigl(\cN^\bullet_{p^r}(\mathcal C_{X/W})\bigr))$.
The derived restrictions to $u=0$ and to $t=0$ are the cofibres of
$u$ and of $t$, which commute with $R\Gamma(X,-)$.  So both
equivalences follow from Proposition~\ref{prop:sheaf-fibres} by
applying $(\sigma^{-1})^*R\Gamma(X,-)$, the filtrations being
formed before global sections, and the graded pieces in (5) are
those of the proposition.

For (6), the origin is the transversal intersection of the two
divisors (Construction~\ref{con:dRW-pullbacks}), so
$i_{\mathrm{Hodge}_r}^*$ is the restriction of $i_{\dRW_r,+}^*$ along
the zero section $t=0$ of $\mathbb A^1_{W_r,t}/\Gm$, which takes the
Rees object of a filtration to its associated graded.  Apply this
to (4).
\end{proof}

We now pass to $r=\infty$, the limit over $r$, which Andriopoulos considers
algebraically \cite[\S5.1]{AndriopoulosTR}.

\begin{definition}[The Hodge--Witt and $\infty$-Nygaard filtrations]
\label{def:hodge-witt-filtration}
Let $X$ be a smooth $k$-scheme.  The \emph{Hodge--Witt filtration} on
$R\Gamma_{\crys}(X/W)\simeq R\Gamma(X,W\Omega^\bullet_{X/k})$ is the
filtration by the stupid truncations of the de Rham--Witt complex,
\[
 \Fil^{\geq m}_{\HW}R\Gamma_{\crys}(X/W)
 :=R\Gamma\bigl(X,W\Omega^{\geq m}_{X/k}\bigr).
\]
The \emph{$\infty$-Nygaard filtration} is the limit
\[
 \Fil^{\geq m}_{\Nyg_\infty}R\Gamma_{\crys}(X/W)
 :=R\varprojlim_r\Fil^{\geq m}_{\Nyg_r}R\Gamma_{\crys}(X/W)
\]
along the restriction maps $\pi_r$ of
Construction~\ref{con:r-nygaard-operators}, where each
$\Fil^{\geq m}_{\Nyg_r}R\Gamma_{\crys}$ is a $W$-module through the
inclusion $\iota$.
\end{definition}

\begin{construction}[The $\infty$-Nygaard stacks]
\label{con:infinite-nygaard-stacks}
Let $k^{\Nyg_\infty}$ and $k^{\NN_\infty}$ be the colimits of the naive
and of the genuine $r$-Nygaard stacks along the restriction maps
$\mathsf R_r$ (Definition~\ref{def:level-maps}).  By
Proposition~\ref{prop:level-maps-compatible},
$g_{r+1}\circ\mathsf R_r=\mathsf R_r\circ g_r$, so the contractions induce
a map $g_\infty:k^{\NN_\infty}\to k^{\Nyg_\infty}$.  The same proposition
gives $f_{r+1}\circ\mathsf R_r=\Phi_{\syn}\circ f_r$.  So the maps
$f'_r:=\Phi_{\syn}^{-r}\circ f_r$ satisfy
$f'_{r+1}\circ\mathsf R_r=f'_r$, and they induce
$f_\infty:k^{\NN_\infty}\to k^{\syn}$.  Finally, let
\[
 \nu_r:k^{\Nyg_r}\longrightarrow\mathbb A^1_{W,t}/\Gm
 :=\Spf\bigl(W[t]^\wedge_p\bigr)/\Gm,
 \qquad t\longmapsto t,
\]
be the map that forgets $u$.  Since $\mathsf R_r$ fixes $t$, we have
$\nu_{r+1}\circ\mathsf R_r=\nu_r$, and the coordinates $t$ define a
coordinate on $k^{\Nyg_\infty}$.  Let $(k^{\Nyg_\infty})_{t=0}$ be the locus where $t$ vanishes, and call
the locus $k^{\mathrm{Hodge}_\infty}:=(k^{\NN_\infty})_{t=0}$ where
$g_\infty^*t$ vanishes the \emph{$\infty$-Hodge stack}.  Let
\[
 g'_\infty:k^{\mathrm{Hodge}_\infty}\longrightarrow(k^{\Nyg_\infty})_{t=0},
 \qquad
 e_\infty:k^{\mathrm{Hodge}_\infty}\longrightarrow k^{\syn}
\]
be the restrictions of $g_\infty$ and $f_\infty$.
\end{construction}

\begin{lemma}[The naive $\infty$-Nygaard stack]
\label{lem:naive-infinite-nygaard}
The maps $\nu_r$ induce an equivalence
$k^{\Nyg_\infty}\simeq\mathbb A^1_{W,t}/\Gm$.
\end{lemma}

\begin{proof}
Both sides are $p$-adic formal stacks, so it suffices to compare
$T$-points for rings $T$ in which $p$ is nilpotent.  The colimit is
filtered, so its $T$-points form the colimit of the groupoids of
$T$-points.  A $T$-point of $k^{\Nyg_r}$ is a line bundle $L$ on
$\Spec T$ with sections $u$ of $L$ and $t$ of $L^{-1}$ such that
$ut=p^r$, and $\mathsf R_r$ sends it to $(L,pu,t)$.  Suppose $p^N=0$
in $T$.  Then every $T$-point $(L,t)$ of $\mathbb A^1_{W,t}/\Gm$
lifts to $(L,0,t)$ at level $N$, and every $T$-point of
$k^{\Nyg_r}$ becomes $(L,0,t)$ after $N$ restriction maps.
Isomorphisms between such points are isomorphisms of $(L,t)$.  So the
colimit of the groupoids of $T$-points maps isomorphically to the
groupoid of $T$-points of $\mathbb A^1_{W,t}/\Gm$.
\end{proof}

In particular the Hodge--Tate locus $u\neq0$ disappears in the
limit, and the naive $\infty$-Nygaard stack $k^{\Nyg_\infty}$ is the Rees stack of a
filtration on $W$-modules.

\begin{theorem}[The Hodge--Witt filtration]
\label{thm:hodge-witt}
Let $X$ be a smooth quasi-compact $k$-scheme.
\begin{enumerate}
\item The $\infty$-Nygaard filtration is the Hodge--Witt filtration:
 \[
  \Fil^{\geq m}_{\Nyg_\infty}R\Gamma_{\crys}(X/W)\simeq
  \Fil^{\geq m}_{\HW}R\Gamma_{\crys}(X/W),
 \]
 compatibly with the maps between filtered pieces.
\item Through Lemma~\ref{lem:naive-infinite-nygaard},
 \[
  g_{\infty,*}f_\infty^*\cH_{\syn}(X)\simeq
  (\sigma^{-1})^*\Rees\bigl(\Fil^{\geq\bullet}_{\HW}R\Gamma_{\crys}(X/W)\bigr)
 \]
 in $\Dqc(\mathbb A^1_{W,t}/\Gm)$.  Its restriction to $t\neq0$ is
 $(\sigma^{-1})^*R\Gamma_{\crys}(X/W)$.  Its pullback to
 $(k^{\Nyg_\infty})_{t=0}\simeq B\mathbb G_{m,W}$ is
 $g'_{\infty,*}e_\infty^*\cH_{\syn}(X)$, with weight-$i$ piece
 $(\sigma^{-1})^*R\Gamma(X,W\Omega^i_{X/k})[-i]$.
\end{enumerate}
\end{theorem}

\begin{proof}
For (1) we work \'etale locally.  Let $\Spec S\to X$ be an \'etale
affine and $C=W\Omega^\bullet_{S/k}$, a strict Dieudonn\'e complex
\cite[Definition~4.1.1]{BLM21}.  The filtration
\[
 \Fil^{\geq m}_{\Nyg}C^i=p^{m-i-1}VC^i\quad(i<m),
 \qquad
 \Fil^{\geq m}_{\Nyg}C^i=C^i\quad(i\geq m)
\]
is carried by $p^iF$ into $p^mC$, with graded pieces
$\tau^{\leq m}(C/p)$ \cite[Proposition~8.2.1]{BLM21}.  These
properties characterize the Nygaard filtration
\cite[Theorem~3.2.1]{BhattFGauges}, so $\iota_m$ is the inclusion and
$\varphi_m=p^{i-m}F$ in degree $i$.

In degrees $i<m$ put $T:=p^{m-i-1}V$.  Then
$\Fil^{\geq m}_{\Nyg}C^i=TC^i$ and $\varphi_m\circ T=\mathrm{id}$ by
$FV=p$, while in degrees $i\geq m$ the map $\iota_m$ is the identity.
So the fibre products of Definition~\ref{def:r-nygaard-filtration} are
computed degreewise, the $r$-fold product consists of the
$x\in\Fil^{\geq m}_{\Nyg}C$ with $\varphi_m^k(x)\in\Fil^{\geq m}_{\Nyg}C$
for $0<k<r$, and
\[
 \Fil^{\geq m}_{\Nyg_r}C^i=T^rC^i\quad(i<m),
 \qquad
 \Fil^{\geq m}_{\Nyg_r}C^i=C^i\quad(i\geq m),
\]
with $\pi_r$ the inclusions.

It remains to show $R\varprojlim_rT^rC^i=0$ for $i<m$.  The canonical
filtration $\Fil^sC^i:=V^sC^i+dV^sC^{i-1}$ is complete and separated,
since $C$ is strict \cite[Definition~2.5.4]{BLM21}, and
$T(\Fil^s)\subseteq\Fil^{s+1}$ because $Vd=p\,dV$.  So
$\bigcap_rT^rC^i=0$.  Write a compatible system
in $\varprojlim_rC^i/T^rC^i$ as $c_1$ with corrections
$c_{r+1}-c_r=T^ry_r$.  The series $c=c_1+\sum_{r\geq1}T^ry_r$
converges, and as $T$ is continuous its tails factor as
\[
 c-c_N=\sum_{r\geq N}T^ry_r=T^N\Bigl(\sum_{r\geq N}T^{r-N}y_r\Bigr)
 \in T^NC^i,
\]
so $c$ lifts the system.  Hence $\varprojlim^1$ vanishes as well, and
$R\varprojlim_r\Fil^{\geq m}_{\Nyg_r}C\simeq C^{\geq m}$, functorially
in $S$.  The sheaves $W_n\Omega^i_{X/k}$ are quasi-coherent on
$W_n(X)$ and compatible with \'etale base change
\cite[I, Propositions~1.13.1 and~1.14]{Illusie79}, and the transition
maps between them are surjective.  So $W\Omega^i_{X/k}$ has no higher
cohomology on \'etale affines, and the identification globalizes and
commutes with $R\Gamma(X,-)$.
This proves (1).

For (2), $\Dqc(k^{\NN_\infty})\simeq\varprojlim_r\Dqc(k^{\NN_r})$, and
$f_\infty^*\cH_{\syn}(X)$ is the system $(f_r'^*\cH_{\syn}(X))_r$.
Mapping spaces in the limit are computed levelwise, so
\[
 g_{\infty,*}(\mathcal F_r)_r\simeq R\varprojlim_r\,\nu_{r,*}g_{r,*}\mathcal F_r,
\]
along the units of $\mathsf R_r^*\dashv\mathsf R_{r,*}$, and $\nu_{r,*}$
only forgets the action of $u$.  By
Proposition~\ref{prop:r-nygaard-comparison}, the weight-$m$ part of
$g_{r,*}f_r'^*\cH_{\syn}(X)$ is the fibre product of
Definition~\ref{def:r-nygaard-filtration}, with its last factor read
through $(\sigma^{-1})^*$, since $\Phi_{\syn}^{-r}$ twists
coefficients by $(\sigma^{-r})^*$.  As $\mathsf R_r$ sends the $i$-th
chart to the $(i+1)$-st, the \v{C}ech formula
(Proposition~\ref{prop:cech-formula-for-Rgr}) shows that the transition
maps forget the first factor, so they are the maps $\pi_r$, and (1)
gives the equivalence.  Its restriction to $t\neq0$ and its pullback
to $t=0$ are the underlying object and the associated graded of the
Hodge--Witt filtration.  The pullback to $t=0$ is the cofibre of $t$,
which commutes with the limit, so base change along $t=0$ identifies
it with $g'_{\infty,*}e_\infty^*\cH_{\syn}(X)$.
\end{proof}

\begin{remark}[Ekedahl's explicit description of the Hodge--Witt filtration]
\label{rem:hodge-witt-ekedahl}
The $r$-Nygaard filtration asks that the first $r$ iterates of the
divided Frobenius stay in the Nygaard filtration.  On cohomology,
Ekedahl gives an explicit description of the Hodge--Witt filtration of
a Mazur--Ogus variety by the same kind of condition: for $M=H^n_{\crys}(X/W)$ with Frobenius
$\varphi$, the image of $H^n(X,W\Omega^{\geq i}_{X/k})$ in $M$ is the
largest sub-$F$-crystal on which $p^{-i}\varphi$ is integral,
equivalently $\bigcap_{r\geq0}\varphi^{-r}(p^{ir}M)$
\cite[Theorem~III.4.5]{Ekedahl86Diagonal}.
Theorem~\ref{thm:hodge-witt}(1) is a counterpart for complexes.
\end{remark}

\begin{remark}[The genuine $\infty$-Nygaard stack]
\label{rem:genuine-infinite-nygaard}
Restriction sends the $i$-th chart of $k^{\NN_r}$ to the $(i+1)$-st
chart of $k^{\NN_{r+1}}$, so it preserves $j=r-1-i$.  Thus
$k^{\NN_\infty}$ is a chain of Nygaard charts indexed by $j\geq0$, and
the composite of $g_\infty$ with Lemma~\ref{lem:naive-infinite-nygaard}
sends the $j$-th chart to $\mathbb A^1_{W,t}/\Gm$ by $t\mapsto p^jt_j$.
Weightwise, the \v{C}ech formula for $g_{\infty,*}$ is the infinite
iterated fibre product by which Andriopoulos describes the
$\infty$-Nygaard filtration \cite[\S5.1]{AndriopoulosTR}.  The same
$\mathbf N$-indexed shape appears in mixed characteristic.  Lahoti and
Manam introduce the moduli stack $\mathbf Z_p^{\NN}$ of polyfiltered
Cartier--Witt divisors, an $\mathbf N$-indexed chain of copies of the
Nygaard stack of $\mathbf Z_p$ glued pairwise along the Hodge--Tate and
de Rham maps \cite[\S1.3]{LahotiManam}.
Frobenius commutes with restriction, and
$f'_{r+1}\circ\mathsf F_r=f'_r$ by
Proposition~\ref{prop:level-maps-compatible}.  So the Frobenius maps
induce an endomorphism of $k^{\NN_\infty}$ over $k^{\syn}$, which
moves the $j$-th chart to the $(j+1)$-st and lies over the map
$t\mapsto pt$ of $\mathbb A^1_{W,t}/\Gm$.

The $\infty$-Hodge stack $k^{\mathrm{Hodge}_\infty}$ plays the role at $r=\infty$ of the $r$-Hodge stack $k^{\mathrm{Hodge}_r}$
(Construction~\ref{con:universal-Hodge-extraction}).  It sits in the
cartesian square
\[
\begin{tikzcd}[column sep=large]
 k^{\syn}
 & k^{\NN_\infty} \arrow[l, "f_\infty"'] \arrow[d, "g_\infty"']
 & k^{\mathrm{Hodge}_\infty} \arrow[l] \arrow[d, "g'_\infty"]\\
 & \mathbb A^1_{W,t}/\Gm
 & B\mathbb G_{m,W}, \arrow[l, "t=0"]
\end{tikzcd}
\]
whose bottom arrow is the origin.  On the $j$-th chart,
$g_\infty^*t=p^jt_j$ is a non-zero-divisor, so $k^{\mathrm{Hodge}_\infty}$
is classical.  This is unlike the $r$-Hodge stack $k^{\mathrm{Hodge}_r}$, whose further
equation $p^iu_i$ on the $i$-th chart tends to zero as $i=r-1-j$ grows.
The $j$-th chart has graded coordinate ring
$W[u,t]/(ut-p,\,p^jt)=\pi_*\mathrm{THH}^{hC_{p^j}}$, so $g'_\infty$ is
the comparison map of even stacks for $\mathrm{TR}(k;\mathbf Z_p)$ in
Remark~\ref{rem:TR-motivation}.

Topologically, $k^{\mathrm{Hodge}_\infty}$ is a chain of closed and open
points
\[
 B\mathbb G_{m,k}\;-\;\Spf W_1(k)\;-\;B\mathbb G_{m,k}\;-\;\Spf W_2(k)
 \;-\;B\mathbb G_{m,k}\;-\;\cdots.
\]
The $j$-th copy of $B\mathbb G_{m,k}$ is the origin of the $j$-th
chart, where $\cH_{\syn}(X)$ restricts to Hodge cohomology.  The copy of
$\Spf W_j(k)=\mathrm{Spev}(\mathrm{THH}^{tC_{p^j}})$ is the overlap of
the $(j-1)$-st and the $j$-th chart, where
$\cH_{\syn}(X)$ restricts to $R\Gamma_{\crys}(X/W_j)$
(Theorem~\ref{thm:ordinary-input}, up to Frobenius twist).  By
Theorem~\ref{thm:hodge-witt}(2), the cohomology of the
$W\Omega^i_{X/k}$, which forms the $E_1$-page of the slope spectral
sequence \cite[II, \S3]{Illusie79}, is assembled from these two kinds
of points.
\end{remark}

\label{sec:hodge-koszul}
\label{sec:dRW-plus-component}
We now return to finite level and turn
Theorem~\ref{thm:r-nygaard-specializations} into a functor on
$F$-gauges.  By its item (4), the Hodge-filtered $r$-de Rham--Witt
pullback of $\cH_{\Nyg_r}(X)$ is the Hodge filtration on
$R\Gamma(X,W_r\Omega^\bullet_{X/k})$.  Its associated graded remembers
the terms $W_r\Omega^i_{X/k}$, and its boundary maps remember the
differential, so Koszul duality recovers the complex itself.  Applied
to an arbitrary $F$-gauge, the same construction gives the finite de
Rham--Witt realization $T_{\dRW_r}$
(Definition~\ref{def:finite-extraction}).  We follow Antieau \cite[\S3]{AntieauDecalage} and Raksit
\cite[\S\S3.2--3.3 and~5.1]{RaksitHH} throughout.

\begin{convention}[Regrading]
\label{conv:shearing}
Let $\mathcal C$ be $D(W_r)$ for some $r\geq1$, or the derived
$\infty$-category of sheaves of $W_r$-modules on $X_{\et}$, and let
$N\in\mathrm{Gr}(\mathcal C)$.  For $\mathcal C=D(W_r)$ this means
\[
 N\in\mathrm{Gr}(D(W_r))\simeq D(\mathrm{GrMod}_{W_r})
 \simeq\Dqc(B\mathbb G_{m,W_r}).
\]
The \emph{regraded} object $N[\bullet]\in\mathrm{Gr}(\mathcal C)$ has
components $(N[\bullet])^n=N^n[n]$.  On a representing complex of graded
modules, $(N[\bullet])^{n,q}=N^{n,q+n}$ with weight-$n$ differential
$(-1)^n\partial_{N^n}$.  Regrading is
inverse to Raksit's identification of the heart of the negative
$t$-structure on graded objects with static graded objects,
$\{X^n\}\mapsto\{X^n[-n]\}$.  As there, the symmetric monoidal structure
acquires the Koszul sign rule: on regraded objects the symmetry has
sign $(-1)^{qq'+nn'}$ for bidegrees
$(q,n),(q',n')$, and the comparison
$N[\bullet]\otimes N'[\bullet]\to(N\otimes N')[\bullet]$ has sign
$(-1)^{qn'}$.  With these signs regrading is symmetric monoidal.
It satisfies $(N(a)[b])[\bullet]\simeq
(N[\bullet])(a)[b-a]$ and
$(N^\vee)[\bullet]\simeq(N[\bullet])^\vee$.  For the associated graded of
a filtration, $N[\bullet]$ is Antieau's $\gr^\bullet[\bullet]$.
\end{convention}

\begin{construction}[Koszul duality]
\label{con:enhanced-gr}
\label{def:linear-Koszul}
Let $\Fil^\bullet C$ be a filtered complex of $W_r$-modules, or of
sheaves of $W_r$-modules on $X_{\et}$ for a smooth $k$-scheme $X$,
viewed as $\Rees(\Fil^\bullet C)$ over $\mathbb A^1_{W_r,t}/\Gm$.  Its
associated graded is the pullback along the zero section
$B\mathbb G_{m,W_r}\hookrightarrow\mathbb A^1_{W_r,t}/\Gm$.  In weight
$m$ it is the Koszul complex
$[\Fil^{m+1}C\,\varepsilon_t\xrightarrow{\,t\,}\Fil^mC]$ in degrees
$-1,0$, where the generator $\varepsilon_t$ has weight $-1$.  The fibre
sequences
$\gr^{m+1}(\Fil^\bullet C)\to\Fil^mC/\Fil^{m+2}C\to\gr^m(\Fil^\bullet C)$
give the boundary maps
\[
 \eta:\ \gr^m\bigl(\Fil^\bullet C\bigr)\longrightarrow
 \gr^{m+1}\bigl(\Fil^\bullet C\bigr)[1],\qquad
 \Fil^{m+1}C\,\varepsilon_t\xrightarrow{\ \mathrm{id}\ }\Fil^{m+1}C .
\]
On the Koszul complexes $\eta(\varepsilon_t)=1$ and $\eta(1)=0$.  For the $p^r$-adic
filtration of a complex of $W$-modules, these boundary maps are the
Bocksteins $\beta$ of Definition~\ref{def:f-Bockstein-complex}.

With their coherences, the boundary maps make the associated graded a
coherent cochain complex.  In Raksit's form, it is a module over the dual of the bar
construction of the augmented graded algebra $W_r[t]$,
\[
 D_-:=\mathrm{Bar}(W_r[t])^\vee,\qquad
 \mathrm{Bar}(W_r[t])=W_r\otimes^{L}_{W_r[t]}W_r
 \simeq W_r\oplus W_r(1)[1].
\]
In the homotopy category $D_-$ is the exterior algebra
$W_r[\eta]/\eta^2$, where $\eta$ has complex degree $1$ and weight $1$
and acts by the boundary maps.  We write
$\mathrm{gr}^\bullet(\Fil^\bullet C)$ for the associated graded with
this module structure, and we call
\[
 \mathrm{gr}^\bullet(\Fil^\bullet C)[\bullet]\in
 D(\mathrm{GrMod}_{W_r[d]/d^2})
\]
the \emph{associated cochain complex}, or \emph{Koszul dual}, of
$\Fil^\bullet C$, in the notation of
\cite[Notation~3.22]{AntieauDecalage}.  Regrading is
symmetric monoidal, so it carries $D_-$ to the graded ring $W_r[d]/d^2$
and $D_-$-modules to graded $W_r[d]/d^2$-complexes; in particular
$\eta$ becomes the operator $d$ of complex degree $0$ and internal
degree $1$.
\end{construction}

\begin{proposition}[{Koszul duality;
\cite[Corollary~3.32]{AntieauDecalage} and \cite[Theorem~3.2.15]{RaksitHH}}]
\label{prop:linear-Koszul-equivalence}
Let $\widehat D_{\mathrm{qc}}(\mathbb A^1_{W_r,t}/\Gm)\subset
\Dqc(\mathbb A^1_{W_r,t}/\Gm)$ be the full subcategory of derived
$t$-complete objects, which the Rees equivalence identifies with
complete filtered complexes of $W_r$-modules.  Then
\[
 \mathrm{gr}^\bullet(-)[\bullet]:\
 \widehat D_{\mathrm{qc}}(\mathbb A^1_{W_r,t}/\Gm)
 \xrightarrow{\ \sim\ }
 D(\mathrm{GrMod}_{W_r[d]/d^2})
\]
is an equivalence.
\end{proposition}

\begin{proof}
By Ariotta's $E_1$-page theorem \cite{Ariotta}, in Raksit's form, $\mathrm{gr}^\bullet$
identifies the source with $D_-$-modules in $\mathrm{Gr}(D(W_r))$.  Regrading is a
symmetric monoidal equivalence of $\mathrm{Gr}(D(W_r))$ carrying $D_-$
to the static graded ring $W_r[d]/d^2$, since $D_-$ is the image of the
square-zero ring under the inverse of regrading.  So it identifies
$D_-$-modules with modules over $W_r[d]/d^2$, and these form the
derived category of graded $W_r[d]/d^2$-modules
\cite[Corollary~3.32]{AntieauDecalage}.
\end{proof}

The Hodge filtration of a cochain complex is its stupid filtration,
which lies in the heart of the Beilinson $t$-structure.

\begin{proposition}[{The stupid filtration;
\cite[Example~3.23]{AntieauDecalage} and \cite[Construction~3.3.10,
Proposition~3.3.13 and Remark~3.3.15]{RaksitHH}}]
\label{prop:Koszul-Rees-recovery}
Let $C^\bullet$ be a cochain complex of $W_r$-modules, or of sheaves of
$W_r$-modules on $X_{\et}$, with Hodge filtration $\Fil_H^aC=C^{\geq a}$.
There is a natural equivalence
\[
 \mathrm{gr}^\bullet\bigl(\Rees(\Fil_H^\bullet C)\bigr)[\bullet]
 \simeq C^\bullet
\]
of graded $W_r[d]/d^2$-complexes, with $C^i$ in internal degree $i$
and $d$ the differential $C^i\to C^{i+1}$.
\end{proposition}

We apply this construction to the Hodge-filtered $r$-de Rham--Witt
pullback of an $F$-gauge.

\begin{definition}[Finite de Rham--Witt realization]
\label{def:finite-extraction}
\label{conv:untwisted-tower}
For an $F$-gauge $E$ and $r\geq1$, the \emph{finite de Rham--Witt
realization} of $E$ at level $r$ is
\[
 T_{\dRW_r}(E):=(\sigma^{1-r})^*\Bigl(\mathrm{gr}^\bullet
   \bigl(i_{\dRW_r,+}^*g_{r,*}f_r^*E\bigr)[\bullet]\Bigr)
 \in D(\mathrm{GrMod}_{W_r[d]/d^2}).
\]
\end{definition}

The Frobenius twist $(\sigma^{1-r})^*$ cancels the twist
$(\sigma^{r-1})^*$ of Theorem~\ref{thm:r-nygaard-specializations}(4),
and it is trivial at $r=1$.  Since $\Phi_{\syn}$ applies $\sigma$ to
the coefficients of every chart and fixes $u_i$ and $t_i$, we have
$T_{\dRW_r}(\sigma^*E)\simeq\sigma^*T_{\dRW_r}(E)$.

For $F$-gauge cohomology, $T_{\dRW_r}$ is computed by
Theorem~\ref{thm:r-nygaard-specializations} and
Proposition~\ref{prop:Koszul-Rees-recovery}.

\begin{corollary}[The finite de Rham--Witt comparison]
\label{thm:filtered-comparison}
Let $X/k$ be smooth and quasi-compact and $r\geq1$.  There is a natural
equivalence
\[
 T_{\dRW_r}(\cH_{\syn}(X))
 \simeq R\Gamma(X,W_r\Omega^\bullet_{X/k})
 \quad\text{in }D(\mathrm{GrMod}_{W_r[d]/d^2}),
\]
with $R\Gamma(X,W_r\Omega^i_{X/k})$ in internal degree $i$ and
the de Rham--Witt differential as the operator $d$.
\end{corollary}

\begin{proof}
By Theorem~\ref{thm:r-nygaard-specializations}(1) and
Proposition~\ref{prop:sheaf-fibres},
\[
 i_{\dRW_r,+}^*\cH_{\Nyg_r}(X)\simeq(\sigma^{r-1})^*
 R\Gamma\bigl(X,\Rees(\Fil_H^\bullet W_r\Omega^\bullet_{X/k})\bigr),
\]
where the Rees object is a sheaf on $X_{\et}$.  The functor
$R\Gamma(X,-)$ is exact and $W_r$-linear, so it commutes with
$\mathrm{gr}^\bullet$ and with regrading.
Proposition~\ref{prop:Koszul-Rees-recovery} on $X_{\et}$ identifies the
associated cochain complex of this sheaf with $W_r\Omega^\bullet_{X/k}$ and its
differential.  The Frobenius twist $(\sigma^{1-r})^*$ of
Definition~\ref{def:finite-extraction} cancels $(\sigma^{r-1})^*$.
\end{proof}

At $r=\infty$, the corresponding functor is the de Rham--Witt
realization $T_{\dRW}$ of \S\ref{sec:representing-objects}, whose
reductions to level $r$ are the $T_{\dRW_r}$
(Theorem~\ref{thm:presentable-equivalence}).  For $X$ smooth proper,
Corollary~\ref{cor:geometric-raynaud-realization} identifies
$T_{\dRW}(\cH_{\syn}(X))$ with $R\Gamma(X,W\Omega^\bullet_{X/k})$.

\section{Ekedahl's derived equivalence}
\label{sec:ekedahl-monoidal}
\label{sec:comparison}

Our goal in this section is to prove that the de Rham--Witt
realization is an equivalence (Theorem~\ref{thm:intro-equivalence})
geometrically.  Its symmetric monoidal structure is the subject of
Section~\ref{sec:duality-and-tensor}.  The finite
de Rham--Witt realizations $T_{\dRW_r}$ of \S\ref{sec:hodge-koszul} see
one level at a time.  To assemble them into an object of $D(R)$ we need the
operators that relate different levels, and a way to pass to
$r=\infty$.  Both come from representability.  Each $T_{\dRW_r}$ is
corepresented by perfect $F$-gauges $G_{r,n}$ built from the $r$-Hodge
stack $k^{\mathrm{Hodge}_r}$ (\S\ref{sec:hodge-stack}).  Restriction and
Frobenius are induced by the maps
$\mathsf R_r,\mathsf F_r\colon k^{\NN_r}\to k^{\NN_{r+1}}$ that embed a
chain as the right and the left sub-chain of a longer one, and
Verschiebung is the transpose of Frobenius for the duality of these
objects.  These operators satisfy the relations of the Raynaud ring and
recover the classical ones (\S\ref{sec:geometric-verschiebung}).  The
colimits $G_n$ of the $G_{r,n}$ over $r$ carry a right action of the
Raynaud ring.  The de Rham--Witt realization $T_{\dRW}$ is the functor
they corepresent, and its left adjoint $S_{\mathrm{FG}}$ is computed by Ekedahl's
functor $S$ (\S\ref{sec:universal-kernel}).  The equivalence then reduces to
two facts, that $T_{\dRW}S_{\mathrm{FG}}$ is Raynaud completion and that $T_{\dRW}$ is
conservative.  Both are finite Koszul calculations on the Nygaard
charts, and finiteness criteria on the two sides then identify the
perfect objects (\S\ref{sec:perfect-coefficients}).

\subsection{The \texorpdfstring{$r$}{r}-Hodge stack and the corepresenting objects}
\label{sec:hodge-stack}

The finite de Rham--Witt realizations are computed on the following
derived stack (Lemma~\ref{lem:T-r-Hodge}).  Although $f_r$ is not
finite, its restriction to this stack is finite and oriented
(Lemma~\ref{lem:e-finite-oriented}), which makes $T_{\dRW_r}$
corepresentable by perfect $F$-gauges.

\begin{construction}[The $r$-Hodge stack]
\label{con:universal-Hodge-extraction}
We call the derived fibre product
\[
 k^{\mathrm{Hodge}_r}:=k^{\NN_r}\times^{R}_{k^{\Nyg_r}}B\mathbb G_{m,W_r}
\]
of the contraction $g_r$ and the inclusion $i_{\mathrm{Hodge}_r}$ of
the origin $B\mathbb G_{m,W_r}=V(u,t)$ of $k^{\Nyg_r}$ the
\emph{$r$-Hodge stack}.  Its projections $i'_{\mathrm{Hodge}_r}$ and
$g'_r$, and the composite $e_r:=f_r\circ i'_{\mathrm{Hodge}_r}$, fit into
the diagram
\[
\begin{tikzcd}[column sep=large, row sep=large]
 k^{\syn}
 & k^{\NN_r} \arrow[l, "f_r"'] \arrow[d, "g_r"']
 & k^{\mathrm{Hodge}_r} \arrow[l, "i'_{\mathrm{Hodge}_r}"'] \arrow[d, "g'_r"]
   \arrow[ll, bend right=30, "e_r"']\\
 & k^{\Nyg_r}
 & B\mathbb G_{m,W_r}, \arrow[l, "i_{\mathrm{Hodge}_r}"]
\end{tikzcd}
\]
in which the square is a derived pullback.
\end{construction}

\begin{remark}[The geometry of $k^{\mathrm{Hodge}_r}$]
\label{rem:hodge-geometry}
Recall that the $i$-th chart of $k^{\NN_r}$ is
$k^{\Nyg}_i=\Spf(B_i)/\Gm$, where
\[
 B_i=\bigl(W^{(i)}[u_i,t_i]/(u_it_i-p)\bigr)^{\wedge}_p,\qquad
 |u_i|=1,\quad |t_i|=-1,
\]
and that $g_r$ sends $u$ to $p^iu_i$ and $t$ to $p^{r-1-i}t_i$ there.
Over $k^{\Nyg}_i$, the $r$-Hodge stack $k^{\mathrm{Hodge}_r}$ is therefore the quotient by
$\Gm$ of the spectrum of the $1$-truncated animated ring
$B_i/^{L}(p^iu_i,p^{r-1-i}t_i)$, the quotient taken in the Koszul sense.
For $r\geq2$ the two elements do not form a regular sequence, since
$p^{r-1}t_0$ already vanishes on $B_0/(u_0)=k[t_0]$.  Thus
$k^{\mathrm{Hodge}_r}$ is classical only for $r=1$, where it is
$B\mathbb G_{m,k}$.
Globally, $k^{\mathrm{Hodge}_r}$ is the derived fibre over
$B\mathbb G_{m,W_r}$ of the resolution $g_r$ of the
$A_{r-1}$-singularity.  Its underlying reduced stack is the
exceptional chain of $r-1$ projective lines over $k$, modulo $\Gm$,
where the $j$-th line is glued from the locus $u_{j-1}=0$ in
$k^{\Nyg}_{j-1}$ and the locus $t_j=0$ in $k^{\Nyg}_j$ along
$u_j=t_{j-1}^{-1}$.  The underlying space is therefore a zigzag of $r$ closed points and
$r-1$ open points.  The closed points are the chart origins
$u_i=t_i=0$, each a copy of $B\mathbb G_{m,k}$.  The open points are the
free $\Gm$-orbits on the lines.  The one on the $j$-th line is the open
substack $\Spec\bigl(W/^{L}(p^j,p^{r-j})\bigr)$ where $u_j$ is a unit,
with $\pi_0=W_{\min(j,r-j)}$, and it specializes to the origins of
$k^{\Nyg}_{j-1}$ and $k^{\Nyg}_j$.
\end{remark}

\begin{lemma}[The finite de Rham--Witt realization on the $r$-Hodge stack]
\label{lem:T-r-Hodge}
For every $F$-gauge $E$ there is a natural equivalence in
$D(\mathrm{GrMod}_{W_r})$
\[
 T_{\dRW_r}(E)\simeq(\sigma^{1-r})^*\bigl((g'_r)_*e_r^*E\bigr)[\bullet],
\]
and $T_{\dRW_r}$ preserves colimits.  In particular
$T_{\dRW_1}(E)\simeq E_{\mathrm{Hodge}}[\bullet]$, where
$E_{\mathrm{Hodge}}:=E/^{L}(u,t)$ is the fibre of the underlying gauge
of $E$ at the Hodge point.
\end{lemma}

\begin{proof}
By Definition~\ref{def:finite-extraction} it suffices to identify the
associated graded of $i_{\dRW_r,+}^*g_{r,*}f_r^*E$ with
$(g'_r)_*e_r^*E$.  Restricting to $u=0$ and passing to the associated
graded is pullback along $i_{\mathrm{Hodge}_r}$.  Since $(u,t)$ is a
regular sequence on $k^{\Nyg_r}$ with zero locus $B\mathbb G_{m,W_r}$,
this pullback is the derived quotient $N\mapsto N/^{L}(u,t)$, that is,
tensor product with the Koszul complex
\[
 \Ksz(u,t)=\bigl[\cO\xrightarrow{(-t,\,u)}
 \cO(-1)\oplus\cO(1)
 \xrightarrow{(u,\,t)}\cO\bigr]
\]
in degrees $-2,-1,0$.  Base change in the square of
Construction~\ref{con:universal-Hodge-extraction} identifies
$i_{\mathrm{Hodge}_r}^*g_{r,*}$ with
$(g'_r)_*i'^*_{\mathrm{Hodge}_r}$, since on the affine charts of $g_r$
pushforward is restriction of scalars and the \v{C}ech formula
\eqref{eq:g-pushforward} assembles the charts.  As
$e_r=f_r\circ i'_{\mathrm{Hodge}_r}$, this gives the equivalence.
Finally, tensor product with $\Ksz(u,t)$, $f_r^*$ and $g_{r,*}$ commute
with colimits, and hence so does $T_{\dRW_r}$.  For $r=1$ we have
$g_1=\mathrm{id}_{k^{\Nyg}}$ and $f_1=j_{\Nyg}$, so
$e_1^*E=E/^{L}(u,t)$.
\end{proof}

To represent $T_{\dRW_r}$ by $F$-gauges, we push forward along $e_r$
and use duality.  Both rest on the following lemma.

\begin{lemma}[Finiteness and orientation of $e_r$]
\label{lem:e-finite-oriented}
The map $e_r$ is finite, and $e_{r,*}\cO$ is a perfect $F$-gauge.  Let
$e_r^!$ denote the right adjoint of $e_{r,*}$, which exists since
$e_r$ is affine, so that $e_{r,*}$ preserves colimits
\cite[Corollary~5.5.2.9]{LurieHTT}.  The ordered equations $(u,t)$ give an
orientation, that is, a natural equivalence of functors
$\Dqc(k^{\syn})\to\Dqc(k^{\mathrm{Hodge}_r})$
\[
 e_r^!\simeq e_r^*[-2].
\]
\end{lemma}

\begin{proof}
The closed immersion $i'_{\mathrm{Hodge}_r}$ in $e_r=f_r\circ
i'_{\mathrm{Hodge}_r}$ is the base change along $g_r$ of
$i_{\mathrm{Hodge}_r}$, the zero locus of the regular sequence $(u,t)$ on
$k^{\Nyg_r}$ (Construction~\ref{con:dRW-pullbacks}).  So
$i'_{\mathrm{Hodge}_r,*}\cO\simeq g_r^*\Ksz(u,t)$, which over the chart
$k^{\Nyg}_i$ is the Koszul algebra
\begin{equation}\label{eq:Hodge-extraction-chart-Koszul}
 B_i/^{L}(p^iu_i,p^{r-1-i}t_i)
 \simeq
 [B_i(-1)\xrightarrow{\ p^iu_i\ }B_i]
 \otimes_{B_i}
 [B_i(1)\xrightarrow{\ p^{r-1-i}t_i\ }B_i]
\end{equation}
on generators $\varepsilon_u,\varepsilon_t$ of weights $1,-1$.  The two
factors are the pullbacks along $g_r$ of
$\Ksz(u)=[\cO(-1)\xrightarrow{u}\cO]$ and
$\Ksz(t)=[\cO(1)\xrightarrow{t}\cO]$, whose tensor product is
$\Ksz(u,t)$, and by abuse of notation we denote these pullbacks again by
$\Ksz(u)$ and $\Ksz(t)$.

\emph{Finiteness.}  On the chart $k^{\Nyg}_0$ the first equation is
$g_r^*u=u_0$, and on the chart $k^{\Nyg}_{r-1}$ the second is
$g_r^*t=t_{r-1}$.  So the $r$-Hodge stack $k^{\mathrm{Hodge}_r}$ is
empty over the two ends $(k^{\Nyg}_0)_{u_0\neq0}$ and
$(k^{\Nyg}_{r-1})_{t_{r-1}\neq0}$, and $e_r$ is finite by
Lemma~\ref{lem:syntomic-projection-fibre}.  The same lemma shows that
after pullback along $j_{\Nyg}$, $e_{r,*}\cO$ is the sum of the Koszul
algebras \eqref{eq:Hodge-extraction-chart-Koszul}, pushed forward along
the isomorphisms $f_{r,i}$, so it is perfect.

\emph{Orientation.}  Since $e_r$ is affine, $\Dqc(k^{\mathrm{Hodge}_r})$
is the category of $e_{r,*}\cO$-modules in $\Dqc(k^{\syn})$, with
$e_r^*E=e_{r,*}\cO\otimes E$ and
$e_r^!E=R\mathcal Hom(e_{r,*}\cO,E)\simeq(e_{r,*}\cO)^\vee\otimes E$.
So it suffices to construct an equivalence of $e_{r,*}\cO$-modules
\[
 (e_{r,*}\cO)^\vee\simeq e_{r,*}\cO[-2].
\]

The $r$-Hodge stack is the derived zero locus of the section $(u,t)$ of
$\mathcal V=\cO(1)\oplus\cO(-1)$, pulled back from $k^{\Nyg_r}$ along
$g_r$.  For a section of a vector bundle of rank two, the Koszul algebra
is self-dual up to the determinant of the bundle and the shift $[-2]$:
multiplication followed by projection to the top exterior power of the
dual bundle is a perfect pairing, whether or not the section is
regular.  Here this top exterior power $\det(\mathcal V^\vee)$ is
trivialized by $\varepsilon_u\varepsilon_t$, which has weight zero.  On $k^{\Nyg_r}$, where
$(u,t)$ is regular, this self-duality is the orientation
\[
 i_{\mathrm{Hodge}_r}^!\simeq i_{\mathrm{Hodge}_r}^*[-2]
\]
of the regular immersion $i_{\mathrm{Hodge}_r}$.

Since $f_r$ is \'etale but not finite, we push this pairing forward
after pullback along $j_{\Nyg}$, where $e_{r,*}\cO$ is the sum of the
Koszul algebras \eqref{eq:Hodge-extraction-chart-Koszul}.  The sum of
their pairings descends to a trace
$\operatorname{tr}_{e_r}:e_{r,*}\cO\to\cO_{k^{\syn}}[2]$, since the
descent datum carries the ordered generators $\varepsilon_u,\varepsilon_t$
of the chart $i+1$ to those of the chart $i$, as
$g_{r,i+1}\circ\alpha_i=g_{r,i}$.  The adjoint of $(x,y)\mapsto\operatorname{tr}_{e_r}(xy)$ is
$e_{r,*}\cO$-linear and an equivalence after pullback along $j_{\Nyg}$,
hence the required equivalence, under which the counit
$e_{r,*}e_r^!\cO\to\cO$ becomes $\operatorname{tr}_{e_r}[-2]$.
\end{proof}

At $r=\infty$, the corresponding map
$e_\infty:k^{\mathrm{Hodge}_\infty}\to k^{\syn}$ of
Construction~\ref{con:infinite-nygaard-stacks} is neither finite nor
oriented in the sense of Lemma~\ref{lem:e-finite-oriented}.  Over
$k^{\Nyg}$ it has one piece for each of the infinitely many charts, so it
is not finite.  On each piece only the single equation $p^jt_j$ remains
(Remark~\ref{rem:genuine-infinite-nygaard}), so the self-duality would
involve a single Koszul generator.  Moreover, the dual of the
infinite sum of the pieces is their product, so the pairings on the
pieces do not assemble into a perfect pairing as in
Lemma~\ref{lem:e-finite-oriented}.  For this reason
Verschiebung is constructed at finite level
(\S\ref{sec:geometric-verschiebung}) and passed to $r=\infty$
afterwards.

We now corepresent $T_{\dRW_r}$ by perfect $F$-gauges, computing its
weight components on the $r$-Hodge stack $k^{\mathrm{Hodge}_r}$.

\begin{proposition}[Perfect corepresenting objects]
\label{thm:oriented-finite-extraction}
For every integer $n$ there is a perfect $F$-gauge $G_{r,n}$ with
natural equivalences
$T_{\dRW_r}(E)^n\simeq\RHom_{k^{\syn}}(G_{r,n},E)$ for $F$-gauges $E$.
Concretely, one may take
\[
 G_{r,n}:=(\sigma^{r-1})^*\,
 e_{r,*}\cO_{k^{\mathrm{Hodge}_r}}(-n)[-n-2].
\]
Then the equivalence is $W_r$-linear for the action of $W_r$ on
$G_{r,n}$ through $g'_r$ by endomorphisms of $F$-gauges, the
\emph{right $W_r$-action}, and
$G_{r,n}^{\vee}\simeq G_{r,-n}[2]$.
\end{proposition}

\begin{proof}
By Lemma~\ref{lem:e-finite-oriented}, $e_{r,*}\cO$ is perfect, and hence
so is $G_{r,n}$.  Since $\sigma^*$ is an autoequivalence with
$T_{\dRW_r}(\sigma^*E)\simeq\sigma^*T_{\dRW_r}(E)$, we may replace $E$ by
$(\sigma^{1-r})^*E$ and drop the twist.  Then
\[
 \begin{aligned}
 T_{\dRW_r}(E)^n
 &\simeq\bigl((g'_r)_*e_r^*E\bigr)^n[n]
 \simeq R\Gamma\bigl(k^{\mathrm{Hodge}_r},e_r^*E(n)\bigr)[n]\\
 &\simeq\RHom_{k^{\mathrm{Hodge}_r}}\bigl(\cO(-n),e_r^!E\bigr)[n+2]
 \simeq\RHom_{k^{\syn}}\bigl(e_{r,*}\cO(-n),E\bigr)[n+2].
 \end{aligned}
\]
The first equivalence is Lemma~\ref{lem:T-r-Hodge}, with $[n]$ from
regrading.  The second takes the weight-$n$ part, since $\cO(n)$ is
pulled back along $g'_r$, and by the projection formula for the affine
map $e_r$ its target is also $R\Gamma(k^{\syn},E\otimes e_{r,*}\cO(n))[n]$.
The third is
Lemma~\ref{lem:e-finite-oriented}, and the fourth is adjunction.  All
four are $W_r$-linear for the action through $g'_r$, which adjunction
carries to precomposition with the action on $\cO(-n)$.

For the duality, the internal form of the adjunction
$e_{r,*}\dashv e_r^!\simeq e_r^*[-2]$ gives
$(e_{r,*}\cO(-n))^\vee\simeq e_{r,*}\cO(n)[-2]$, hence
\[
 G_{r,n}^\vee
 \simeq(\sigma^{r-1})^*\bigl(e_{r,*}\cO(-n)\bigr)^\vee[n+2]
 \simeq(\sigma^{r-1})^*e_{r,*}\cO(n)[n]
 =G_{r,-n}[2].
\]
\end{proof}

\begin{lemma}[Breuil--Kisin twists at finite level]
\label{lem:BK-finite}
For $E\in\Dqc(k^{\syn})$ and $a\in\mathbb Z$ there is an equivalence
$T_{\dRW_r}(E\{a\})\simeq T_{\dRW_r}(E)\{a\}$, natural in $E$.
Consequently $G_{r,n}\simeq G_{r,0}\{-n\}[-n]$ for every $n$.
\end{lemma}

\begin{proof}
The line bundle $\cO\{a\}$ is trivial on the two copies of $\Spf W$
and restricts to a weight twist on $k^{\Nyg}$
\cite[\S4.3]{BhattFGauges}.  So the twist
$\{a\}$ shifts the weights, $(E\{a\})^m=E^{m+a}$, and leaves $u$, $t$,
the two ends and the gluing unchanged, and likewise on the charts of
$k^{\NN_r}$.  By Convention~\ref{conv:shearing}, regrading turns $(a)$
into $(a)[-a]=\{a\}$, which gives the equivalence.  For the last
claim, the same description shows that $\cO(-n)=e_r^*\cO\{-n\}$ on the
$r$-Hodge stack $k^{\mathrm{Hodge}_r}$, so the projection formula gives
$e_{r,*}\cO(-n)\simeq e_{r,*}\cO\otimes\cO\{-n\}$.  As $\sigma^*$
fixes $\cO\{-n\}$, this is $G_{r,n}\simeq G_{r,0}\{-n\}[-n]$.
\end{proof}

\subsection{Restriction, Frobenius, and Verschiebung}
\label{sec:geometric-verschiebung}

The finite de Rham--Witt realizations of different levels are related
by three operators.  Restriction and Frobenius are pullbacks along the
restriction map $\mathsf R_r$ and the Frobenius map $\mathsf F_r$ of
Definition~\ref{def:level-maps}, and Verschiebung is the transpose of
Frobenius for the duality of the corepresenting objects of
Proposition~\ref{thm:oriented-finite-extraction}.  We compute all three
on the Koszul complexes of the charts, check that they satisfy the
relations of the Raynaud ring and commute with restriction, and
identify them with the classical operators on the de Rham--Witt
complex.  We first recall the Raynaud ring.

\begin{definition}[{The Raynaud ring and its finite reductions; \cite[I, \S1]{IllusieRaynaud83}}]
\label{def:raynaud-ring}
The \emph{Raynaud ring} is the semilinear graded ring
\[
 R=W_\sigma\langle F,V,d\rangle/
 (FV=VF=p,\ FdV=d,\ d^2=0),
\]
where $Fa=\sigma(a)F$, $Va=\sigma^{-1}(a)V$, $da=ad$,
$\deg F=\deg V=0$, and $\deg d=1$.
In $D(R)$ the internal degree and the complex degree
are independent gradings, as in the conventions of the introduction.
We write $R(-n)$ for the free left module generated in
internal degree $n$.  For finite reduction put
\[
 R_r=R/(V^rR+dV^rR),
\]
a $(W_r[d]/d^2,R)$-bimodule, so that $R_r\otimes_R^{L}M$ is a graded
$W_r[d]/d^2$-complex for every $M\in D(R)$.
\end{definition}

Computations with $R$ and with its finite reductions $R_r$ reduce to
computations with $W$-modules through explicit bases.  Every element of
$R$ is a unique $W$-linear combination of powers of $V$ and $F$, with
$d$ placed to the left of the powers of $V$ and to the right of the
powers of $F$, and $R_r$ truncates the coefficients according to the
power of $V$.

\begin{remark}[{Normal forms; \cite[I, \S1]{IllusieRaynaud83}}]\label{lem:raynaud-normal-form}
The ring $R$ is concentrated in internal degrees $0,1$.  As left
$W$-modules, respectively $W_r$-modules,
\[
 \begin{aligned}
 R^0&=\bigoplus_{j>0}WV^j\oplus\bigoplus_{\ell\geq0}WF^\ell,&
 R^1&=\bigoplus_{j>0}WdV^j\oplus\bigoplus_{\ell\geq0}WF^\ell d,\\
 R_r^0&=\bigoplus_{1\leq j<r}W_{r-j}V^j\oplus\bigoplus_{\ell\geq0}W_rF^\ell,&
 R_r^1&=\bigoplus_{1\leq j<r}W_{r-j}dV^j\oplus\bigoplus_{\ell\geq0}W_rF^\ell d.
 \end{aligned}
\]
\end{remark}

\begin{definition}[{The simple functor; \cite[\S0.3]{Ekedahl86Diagonal}}]
\label{def:simple-functor}
An object $M\in D(R)$ represented by a complex of graded $R$-modules in
a finite internal range is a double complex, with the differential $d$
in the internal direction and its own differential $\partial_M$ in the
complex direction.  On such objects, Ekedahl's \emph{simple functor}
$\mathbf s$, with values in $D(W)$, totalizes the two differentials,
placing $M^{i,q}$ in degree $i+q$ with differential
$d+(-1)^i\partial_M$.  As the internal range is finite, $\mathbf s$
preserves quasi-isomorphisms.  The same totalization applies to any
graded complex in a finite internal range with a square-zero operator of
internal degree one.  For
a smooth proper $k$-scheme $X$, $\mathbf s\,R\Gamma(X,W\Omega^\bullet_{X/k})
\simeq R\Gamma_{\crys}(X/W)$ \cite[II, Th\'eor\`eme~1.4]{Illusie79}.
\end{definition}

We now construct the three operators on the system
$\{T_{\dRW_r}(E)\}_r$.  Restriction and Frobenius come from the two
embeddings of the level-$r$ chain into the level-$(r+1)$ chain.
Verschiebung goes in the
opposite direction, from level $r$ to level $r+1$, and is not a
pullback.  It is defined instead by transposing Frobenius through the
duality of the corepresenting objects.

\begin{construction}[Restriction, Frobenius, and Verschiebung on $T_{\dRW_r}$]
\label{def:intrinsic-geometric-V}
We construct three natural transformations
\[
 \pi_r:T_{\dRW_{r+1}}\to T_{\dRW_r},\qquad
 F_r:T_{\dRW_{r+1}}\to\sigma_*T_{\dRW_r},\qquad
 V_r:\sigma_*T_{\dRW_r}\to T_{\dRW_{r+1}}.
\]
Let $E$ be an $F$-gauge, and recall that on a common chart
$\mathsf R_r$ multiplies the $u$-equation by $p$ and $\mathsf F_r$
multiplies the $t$-equation by $p$ (Definition~\ref{def:level-maps}).
\begin{itemize}
\item \emph{Restriction.}  The two squares of
 Proposition~\ref{prop:level-maps-compatible} give a map of
 $r$-gauges
 \[
  \mathsf R_r^*\,g_{r+1,*}f_{r+1}^*E
  \longrightarrow g_{r,*}\mathsf R_r^*f_{r+1}^*E
  \simeq g_{r,*}f_r^*\sigma^*E,
 \]
 the first arrow being the Beck--Chevalley map of the commutative
 square $g_{r+1}\circ\mathsf R_r=\mathsf R_r\circ g_r$, and the
 equivalence by
 $f_{r+1}\circ\mathsf R_r=\Phi_{\syn}\circ f_r$.  Restricting to $u=0$,
 where $\mathsf R_r$ is the closed immersion
 $\mathbb A^1_{W_r,t}/\Gm\to\mathbb A^1_{W_{r+1},t}/\Gm$, and passing
 to associated cochain complexes gives the \emph{restriction} $\pi_r$.
\item \emph{Frobenius.}  The same steps with $\mathsf F_r$, which on
 $u=0$ applies $\sigma$ to the coefficients and sends $t$ to $pt$, give
 the \emph{Frobenius} $F_r$.
\item \emph{Verschiebung.}  By
 Proposition~\ref{thm:oriented-finite-extraction} and the Yoneda lemma,
 $\pi_r$ and $F_r$ are, in weight $n$, precomposition with maps
 $G_{r,n}\to G_{r+1,n}$.  Dualizing the Frobenius map
 $G_{r,-n}\to G_{r+1,-n}$ through $G_{r,n}^\vee\simeq G_{r,-n}[2]$ and
 shifting by $[-2]$ gives a map $G_{r+1,n}\to G_{r,n}$, and
 precomposition with it is the \emph{Verschiebung} $V_r$.
\end{itemize}
The Frobenius twists of Definition~\ref{def:finite-extraction},
together with $T_{\dRW_r}(\sigma^*E)\simeq\sigma^*T_{\dRW_r}(E)$, make
all three $W$-linear with the stated sources and targets.  Together with
the differential $d$ of $T_{\dRW_r}(E)$, which raises internal degree by
one, they are the operators $\pi,F,V,d$ on the system
$\{T_{\dRW_r}(E)\}_r$.
\end{construction}

\begin{remark}[The twist convention]
\label{rem:twist-one-minus-r}
The twists $(\sigma^{1-r})^*$ in $T_{\dRW_r}$ and $(\sigma^{r-1})^*$ in
$G_{r,n}$ measure the length of the chain.  The coefficients of
$k^{\Nyg_r}$ are those of the $0$-th chart, which carries the
Hodge--Tate locus $u\neq0$ in Bhatt's normalization, while
$T_{\dRW_r}$ is read off on the $(r-1)$-st chart, $r-1$ Frobenius
gluings away (Lemma~\ref{lem:contraction-open-loci}).  As
$f_{r+1}\circ\mathsf R_r=\Phi_{\syn}\circ f_r$
(Proposition~\ref{prop:level-maps-compatible}), these twists make
$\pi_r$ $W$-linear.  Up to a twist by $\sigma$ independent of $r$, they
are the normalization $f'_r=\Phi_{\syn}^{-r}\circ f_r$ of
Construction~\ref{con:infinite-nygaard-stacks}, so
$T_{\dRW}=R\varprojlim_rT_{\dRW_r}$ is the $r=\infty$ counterpart of
the finite realizations, and the remaining $\sigma$ is the
$(\sigma^{-1})^*$ of Theorem~\ref{thm:hodge-witt}(2).
\end{remark}

By the Yoneda lemma, natural transformations between the functors
$T_{\dRW_r}$ are the same as maps between the corepresenting objects
$G_{r,n}$, which are twists of $e_{r,*}\cO$.  In particular $\pi$, $F$
and $V$ are natural in $E$.  For Frobenius and Verschiebung these maps
are geometric.  The Frobenius maps
$\mathsf F_r$ preserve the origins, so they induce maps of $r$-Hodge
stacks, again denoted $\mathsf F_r$, with commutative squares
\[
\begin{tikzcd}[column sep=large, row sep=large]
 k^{\syn} \arrow[d, "\Phi_{\syn}"']
 & k^{\mathrm{Hodge}_r} \arrow[l, "e_r"'] \arrow[r, "g'_r"]
   \arrow[d, "\mathsf F_r"]
 & B\mathbb G_{m,W_r} \arrow[d, "\mathsf F_r|_{u=t=0}"]\\
 k^{\syn}
 & k^{\mathrm{Hodge}_{r+1}} \arrow[l, "e_{r+1}"] \arrow[r, "g'_{r+1}"']
 & B\mathbb G_{m,W_{r+1}}.
\end{tikzcd}
\]
Since $e_r$ is finite and $e_{r+1}$ is affine, $\mathsf F_r$ is finite,
and the orientations of $e_r$ and $e_{r+1}$
(Lemma~\ref{lem:e-finite-oriented}) give $\mathsf F_r^!\cO\simeq\cO$.
Pullback of functions along $\mathsf F_r$ and the \emph{trace}
$\operatorname{tr}_{\mathsf F_r}:\mathsf F_{r,*}\cO\simeq\mathsf
F_{r,*}\mathsf F_r^!\cO\to\cO$, the counit, give maps of $F$-gauges
\[
 F_r:e_{r+1,*}\cO(n)\to\sigma_*e_{r,*}\cO(n),\qquad
 V_r:\sigma_*e_{r,*}\cO(n)\to e_{r+1,*}\cO(n).
\]
Counits compose, so $\operatorname{tr}_{e_{r+1}}\circ
e_{r+1,*}\operatorname{tr}_{\mathsf F_r}=\operatorname{tr}_{e_r}$, and
the trace is linear over $\cO_{k^{\mathrm{Hodge}_{r+1}}}$.  Hence
$\operatorname{tr}_{e_{r+1}}(V_r(x)\,y)=\operatorname{tr}_{e_r}(x\,F_r(y))$,
so $V_r$ is the transpose of $F_r$ for the trace pairings.  Under the Yoneda correspondence, $F$ and $V$ of
Construction~\ref{def:intrinsic-geometric-V} are $F_r$ and $V_r$.
The underlying gauge of $e_{r,*}\cO$, its pullback along $j_{\Nyg}$,
is the sum of the Koszul algebras of the charts,
\[
 j_{\Nyg}^*e_{r,*}\cO\simeq\bigoplus_{i=0}^{r-1}
 f_{r,i,*}\bigl(B_i/^{L}(p^iu_i,\,p^{r-1-i}t_i)\bigr)
\]
(Lemma~\ref{lem:syntomic-projection-fibre}), and we compute $F_r$ and
$V_r$ on each summand.

\begin{proposition}[Local formulas for Frobenius and Verschiebung]
\label{prop:geometric-V-local}
Fix $0\leq i<r$, let $k^{\Nyg}_i=\Spf(B_i)/\Gm$ with
$B_i=\bigl(W^{(i)}[u_i,t_i]/(u_it_i-p)\bigr)^\wedge_p$ be the $i$-th
chart, on which $g_r^*t=p^{r-1-i}t_i$, and consider the $i$-th summand
of the underlying gauge $j_{\Nyg}^*e_{r,*}\cO$.  As in the proof of
Lemma~\ref{lem:e-finite-oriented}, let $\Ksz(u)$ and $\Ksz(t)$ be the
pullbacks along $g_r$ of the Koszul complexes on $k^{\Nyg_r}$, so that
$\Ksz(t)$ is $B_i/^{L}(p^{r-1-i}t_i)$ at level $r$ and
$B_i/^{L}(p^{r-i}t_i)$ at level $r+1$.
Then on the $i$-th summand the maps $F_r$ and $V_r$ are the identity
on $\Ksz(u)$, and on $\Ksz(t)$ they are
\[
 F_r:\quad
 \begin{CD}
 B_i @>{p^{r-i}t_i}>> B_i\\
 @V{p}VV @VV{1}V\\
 B_i @>{p^{r-1-i}t_i}>> B_i
 \end{CD}
 \qquad
 V_r:\quad
 \begin{CD}
 B_i @>{p^{r-1-i}t_i}>> B_i\\
 @V{1}VV @VV{p}V\\
 B_i @>{p^{r-i}t_i}>> B_i,
 \end{CD}
\]
with rows in degrees $-1,0$, after trivializing the line of
$\varepsilon_t$ and absorbing the Frobenius twist.  On the chart
$k^{\Nyg}_r$ of level $r+1$, $V_r$ vanishes.
\end{proposition}

\begin{proof}
Since $\mathsf F_r:k^{\Nyg_r}\to k^{\Nyg_{r+1}}$ fixes $u$ and sends $t$
to $pt$ (Definition~\ref{def:level-maps}), pullback along $\mathsf F_r$
fixes $\Ksz(u)$ and sends $1$ to $1$ and $\varepsilon_t$ to
$p\varepsilon_t$ on $\Ksz(t)$, which is the first square.  The
orientation pairs $1$ with $\varepsilon_t$, so the transpose exchanges
the two coefficients, which is the second square.  The chart
$k^{\Nyg}_r$ is omitted by $\mathsf F_r$, so the Frobenius map has no
component there and its transpose vanishes.

It remains to glue these descriptions on the overlap
$(k^{\Nyg}_{r-1})_{t_{r-1}\neq0}$.  There the equation $t_{r-1}$ is invertible,
so $B_{r-1}/^{L}t_{r-1}$ is zero and its mapping space to the target is
contractible.  Hence the gluing with the zero map on $k^{\Nyg}_r$ is
unique, and it agrees with the gluing supplied by the global transpose.
Writing $\varepsilon_t$ and $\varepsilon_t'$ for the generators of the
source and the target, with $\partial\varepsilon_t=t_{r-1}$ and
$\partial\varepsilon_t'=pt_{r-1}$, an explicit gluing homotopy is
\begin{equation}\label{eq:V-overlap-homotopy}
 h(1)=t_{r-1}^{-1}\varepsilon_t',\qquad h(\varepsilon_t)=0,
 \qquad \partial h+h\partial=V_r.
\end{equation}
Tensoring with $\Ksz(u)$ gives the formulas on the full chart algebra
\eqref{eq:Hodge-extraction-chart-Koszul}.
\end{proof}

\begin{remark}[The Raynaud relations]
\label{rem:local-raynaud-relations}
Write $\varepsilon_t$ and $\varepsilon_t'$ for the generators of
$\Ksz(t)$ at levels $r$ and $r+1$.  By
Proposition~\ref{prop:geometric-V-local}, $F_r$ multiplies
$\varepsilon_t'$ by $p$ and fixes $1$, while $V_r$ sends
$\varepsilon_t$ to $\varepsilon_t'$ and multiplies $1$ by $p$.  In each
degree exactly one of them is multiplication by $p$, so
$F_rV_r=V_rF_r=p$ on the charts of level $r$.  On the chart
$k^{\Nyg}_r$ of level $r+1$, where $V_r$ vanishes, $V_rF_r=p$ holds
through the nullhomotopy of $p=u_rt_r$ given by multiplication by
$u_r\varepsilon_t'$, as $\partial\varepsilon_t'=t_r$ there.  On the
overlap with the previous chart $u_r=t_{r-1}^{-1}$, so this is the
homotopy \eqref{eq:V-overlap-homotopy} applied to $F_r(1)=1$.  The differential $d$ is induced by the boundary map
$\eta$ of Construction~\ref{con:enhanced-gr}, with
$\eta(\varepsilon_t)=1$ and $\eta(1)=0$, which moves the degree
$-1$ generator to the degree $0$ one.  Since $F_r$ is $p$ only in degree
$-1$ and $V_r$ only in degree $0$, the two factors of $p$ never meet in
$F_r\eta V_r$:
\[
 F_r\eta V_r(\varepsilon_t)=F_r\eta(\varepsilon_t')=F_r(1)=1
 =\eta(\varepsilon_t),
 \qquad F_r\eta V_r(1)=0=\eta(1).
\]
In the other orders they do meet, and
$\eta F_r=pF_r\eta$ and $V_r\eta=p\eta V_r$.  After regrading
these identities are the Raynaud relations $FV=VF=p$, $FdV=d$,
$dF=pFd$ and $Vd=pdV$.
\end{remark}

At $r=\infty$ we do not repeat these computations.  The operators on
$T_{\dRW}$ reduce to those on the $T_{\dRW_r}$
(Theorem~\ref{thm:presentable-equivalence}), so they satisfy the same
relations, and \S\ref{sec:representing-objects} assembles them into a
coherent action of $R$.

\begin{lemma}[Restriction and the operator tower]
\label{lem:restriction-operator-compatibility}
For every $F$-gauge $E$ and $r\geq2$, restriction commutes with $F$,
$V$, and $d$:
\[
 \sigma_*(\pi_{r-1})\circ F_r=F_{r-1}\circ\pi_r,\qquad
 \pi_r\circ V_r=V_{r-1}\circ\sigma_*(\pi_{r-1}),\qquad
 \pi_r\circ d=d\circ\pi_r,
\]
naturally in $E$ and compatibly with composites of restrictions.
\end{lemma}

\begin{proof}
Restriction and Frobenius are pullbacks along the maps $\mathsf R$ and
$\mathsf F$ of $r$-Hodge stacks, and $\mathsf R\circ\mathsf F=\mathsf
F\circ\mathsf R$ on the genuine and the naive stacks
(Definition~\ref{def:level-maps}).  This gives the first identity.  On
the locus $u=0$, $\mathsf R$ is the closed immersion
$\mathbb A^1_{W_r,t}/\Gm\to\mathbb A^1_{W_{r+1},t}/\Gm$, which fixes
$t$, so restriction is induced by a map of filtered objects and
commutes with the differential of the associated cochain complexes.
This gives the third.

The second identity is not formal, since Verschiebung is the trace
along $\mathsf F$ rather than a pullback.  It follows from base change,
once we know that the square of $r$-Hodge stacks
\[
\begin{tikzcd}[column sep=large]
 k^{\mathrm{Hodge}_{r-1}} \arrow[r, "\mathsf F_{r-1}"] \arrow[d, "\mathsf R_{r-1}"']
 & k^{\mathrm{Hodge}_r} \arrow[d, "\mathsf R_r"]\\
 k^{\mathrm{Hodge}_r} \arrow[r, "\mathsf F_r"]
 & k^{\mathrm{Hodge}_{r+1}}
\end{tikzcd}
\]
is derived cartesian and that the relative orientations of
$\mathsf F_{r-1}$ and $\mathsf F_r$ correspond.  The images of
$\mathsf R_r$ and $\mathsf F_r$ in the chain of length $r+1$ meet in its
middle sub-chain, so it suffices to check this over a common chart.
There an $r$-Hodge stack is the fibre product of the derived zero loci
of the two equations $g^*u$ and $g^*t$, and $\mathsf R$ changes only the
first factor, $\mathsf F$ only the second.  So the square is the product
of a square with invertible horizontal maps and one with invertible
vertical maps, hence derived cartesian.  The relative orientation of
$\mathsf F$ lives on the second factor
(Proposition~\ref{prop:geometric-V-local}), where $\mathsf R$ is the
identity, so the orientations correspond.  Composites of restrictions follow by iteration.
\end{proof}

For a smooth proper $k$-scheme $X$, we will compare $\pi$, $F$, $V$
and $d$ with the classical operators on
$R\Gamma(X,W_r\Omega^\bullet_{X/k})$, through Ekedahl's functor
(Corollary~\ref{cor:geometric-raynaud-realization}).

\subsection{The de Rham--Witt realization and Ekedahl's functor}
\label{sec:universal-kernel}
\label{sec:representing-objects}

We pass from the objects $G_{r,n}$ of
Proposition~\ref{thm:oriented-finite-extraction} to $r=\infty$, in
three steps.  Definition~\ref{def:representing-objects} and
Proposition~\ref{prop:universal-G} define the colimits $G_n$ and
compute their underlying gauges.  The maps between the $G_n$ satisfy
the Raynaud relations only up to homotopy, so
Construction~\ref{con:ladder} replaces $G_n$ by $S(R(-n))$, the value
of Ekedahl's functor $S$, on which $R$ acts by right multiplication on the nose
(Lemma~\ref{lem:free-ladder} and Construction~\ref{con:right-action-G}).
This action defines $T_{\dRW}$ and its left adjoint $S_{\mathrm{FG}}$
(Definition~\ref{def:geometric-adjunction}), and
Lemma~\ref{thm:ladder-is-quasi-inverse} identifies $S_{\mathrm{FG}}$
with Ekedahl's functor $S$.  From now on in this section, $A:=(W[u,t]/(ut-p))^\wedge_p$
denotes the coordinate ring of $k^{\Nyg}$.

\begin{definition}[The corepresenting objects]
\label{def:representing-objects}
Put
\[
 G_n:=\colim_r G_{r,n}\quad\text{in }\Dqc(k^{\syn}),
\]
where the transition maps $G_{r,n}\to G_{r+1,n}$ are those whose
precomposition is restriction
(Construction~\ref{def:intrinsic-geometric-V}).%
\footnote{Ekedahl's projective generators $G_r$
\cite[\S II.2, Example~v)]{Ekedahl86Diagonal} corepresent
$M\mapsto M^r$, whereas our $G_n$ corepresent
$E\mapsto T_{\dRW}(E)^n$ and are the images $S_{\mathrm{FG}}(R(-n))$ of free
$R$-modules (Lemma~\ref{lem:free-ladder}).}
The right $W_r$-actions
on the $G_{r,n}$ give a \emph{right $W$-action} on $G_n$.  By
Proposition~\ref{thm:oriented-finite-extraction},
\[
 \RHom(G_n,E)\simeq R\varprojlim_rT_{\dRW_r}(E)^n ,
\]
the limit taken along the restriction maps $\pi_r$.
\end{definition}

The right $W$-action is special to the corepresenting objects.  A
general $F$-gauge carries no action of $W$: for instance,
$\operatorname{End}(\cO_{k^{\syn}})=W^{\sigma=1}=\mathbf Z_p$.  On $G_n$, by contrast,
$W$ acts by endomorphisms of $F$-gauges, by $\sigma^b$ on the $b$-th
summand of the underlying gauge (Proposition~\ref{prop:universal-G}),
compatibly with the Frobenius twist between adjacent summands.  In fact
the whole endomorphism ring is large.  By the Yoneda lemma and
Lemma~\ref{lem:path-output} below,
\[
 \operatorname{End}(G_n)^{\mathrm{op}}\simeq
 R\varprojlim_rT_{\dRW_r}(G_n)^n
 \simeq\widehat R^0
 =\bigl(W_\sigma\langle F,V\rangle/(FV-p,\,VF-p)\bigr)^\wedge_V,
\]
the \emph{complete Dieudonn\'e ring}, where $Fa=\sigma(a)F$ and
$Va=\sigma^{-1}(a)V$ for $a\in W$.  So $\widehat R^0$ acts on $G_n$ on the
right, and the right $W$-action is its restriction to $W\subset\widehat R^0$.  Explicitly, the elements of
$\widehat R^0$ are the sums
$\sum_{j\geq1}a_jV^j+\sum_{\ell\geq0}b_\ell F^\ell$ with $a_j,b_\ell\in W$
and $b_\ell\to0$, since $p^rF^\ell=V^rF^{\ell+r}$ lies in $V^rR^0$
(Remark~\ref{lem:raynaud-normal-form}).

\begin{proposition}[The underlying gauge of $G_n$]\label{prop:universal-G}
The underlying gauge of $G_n$ is
\begin{equation}\label{eq:universal-G-chart}
 j_{\Nyg}^*G_n
 \simeq
 \left(\bigoplus_{b\geq0}
       \Ksz_{A}(p^bt)(-n-1)\right)^\wedge_p[-n-1],
\end{equation}
where $(-)^\wedge_p$ is derived $p$-completion.  On the summand
indexed by $b$ the right $W$-action is $a\mapsto\sigma^b(a)$.  The
descent datum identifies the de Rham pullback of the summand $b$ with
the Hodge--Tate pullback of the summand $b-1$, and the de Rham pullback
of the summand $b=0$ vanishes.
\end{proposition}

\begin{proof}
By \eqref{eq:Hodge-extraction-chart-Koszul}, after identifying each
chart with $k^{\Nyg}$ through $f_{r,i}$,
\[
 j_{\Nyg}^*G_{r,n}[n+2]\simeq\bigoplus_{b=0}^{r-1}
 \bigl(\Ksz_A(p^{r-1-b}u)\otimes\Ksz_A(p^bt)\bigr)(-n),
\]
where the summand $b$ comes from the chart $i=r-1-b$.  Its right
$W_r$-action is $\sigma^b$, combining the structural map $\lambda_i$ of
$B_i$ with the Frobenius twist $(\sigma^{r-1})^*$.  The transition map
$G_{r,n}\to G_{r+1,n}$ is the transpose of restriction, which on
$\Ksz(u)$ sends $\varepsilon_u$ to $p\varepsilon_u$ and fixes $1$, as
$\mathsf R_r$ multiplies the $u$-equation by $p$
(Definition~\ref{def:level-maps}).  So it
preserves $b$, it is the identity on $\Ksz_A(p^bt)$, and on $\Ksz(u)$
it is the map $\Ksz_A(p^iu)\to\Ksz_A(p^{i+1}u)$ sending $1$ to $p$ and
$\varepsilon_u$ to $\varepsilon_u$.  Since $j_{\Nyg}^*$ commutes with
colimits, it suffices to compute, for fixed $b$,
\[
 \colim_i\Ksz_A(p^iu)\simeq\bigl[A(-1)\xrightarrow{\ u\ }A[1/p]\bigr],
\]
where $x$ in degree zero of the $i$-th term is identified with
$x/p^i$.  After derived $p$-completion $A[1/p]$ vanishes, and only
$A(-1)[1]$ remains.  Tensoring with $\Ksz_A(p^bt)(-n)[-n-2]$ gives
\eqref{eq:universal-G-chart}, with right $W$-action $\sigma^b$ on
the summand $b$.

It remains to identify the descent datum.  Inverting $t$ turns
$\Ksz_A(p^bt)$ into $W/p^b$, and inverting $u$, so that $t=p/u$, turns
it into $W/p^{b+1}$.  By Lemma~\ref{lem:syntomic-projection-fibre},
the descent datum carries the chart $i$ to the chart $i+1$, that is,
the summand $b$ to the summand $b-1$.  It therefore identifies the de
Rham pullback $W/p^b$ of the summand $b$ with the Hodge--Tate pullback
$W/p^b$ of the summand $b-1$, and its coefficient Frobenius matches
$\sigma^b$ with $\sigma^{b-1}$.  For $b=0$ the de Rham pullback is
$W/p^0=0$.
\end{proof}

Since $G_n$ is defined as a colimit, the operators of
\S\ref{sec:geometric-verschiebung} give maps between the $G_n$, but
their relations hold only up to unspecified homotopies.  So they do not
yet make $T_{\dRW}(E)=\bigoplus_n\RHom(G_n,E)$ a module over $R$.  To
obtain a right action of $R$ on the family $\{G_n\}_n$ with all its
coherences, we replace $G_n$ by a model built from complexes.  The free
modules $R(-n)$ carry the right action of $R$ by right multiplication,
whose relations hold on the nose, and Ekedahl's construction, recalled
next, is a dg functor from complexes of graded $R$-modules to $F$-gauges, so
applying it to the free modules gives the desired action.  We show that it sends $R(-n)$ to $G_n$, read off the resulting
right action, extend $n\mapsto G_n$ to all $M\in D(R)$, and
check that the extension is again computed by Ekedahl's construction.

\begin{construction}[Ekedahl's functor]\label{con:ladder}
Let $M\in D(R)$ be represented by a complex of graded $R$-modules
supported in a finite interval of internal degrees.  In gauge weight $m$ put
\begin{equation}\label{eq:ladder-complex}
 S(M)^m=
 \mathbf s\left[\cdots\longrightarrow\sigma_*M^{m-1}
       \xrightarrow{dV}M^m\xrightarrow dM^{m+1}
       \longrightarrow\cdots\right],
\end{equation}
the simple complex (Definition~\ref{def:simple-functor}) of the row with
differential $d_S$, which is $d$ away from the turn
$\sigma_*M^{m-1}\xrightarrow{dV}M^m$.  The maps between adjacent weights are
\[
\begin{tikzcd}[column sep=2.4em, row sep=2.6em, ampersand replacement=\&]
 S(M)^{m+1} \arrow[d, shift right=0.7ex, "t"'] \& \cdots \arrow[r]
 \& \sigma_*M^{m-1} \arrow[r, "d"]
   \arrow[d, shift right=0.7ex, "p"']
 \& \sigma_*M^{m} \arrow[r, "dV"]
   \arrow[d, shift right=0.7ex, "V"']
 \& M^{m+1} \arrow[r, "d"]
   \arrow[d, shift right=0.7ex, "1"']
 \& \cdots\\
 S(M)^{m} \arrow[u, shift right=0.7ex, "u"'] \& \cdots \arrow[r]
 \& \sigma_*M^{m-1} \arrow[r, "dV"']
   \arrow[u, shift right=0.7ex, "1"']
 \& M^{m} \arrow[r, "d"']
   \arrow[u, shift right=0.7ex, "F"']
 \& M^{m+1} \arrow[r]
   \arrow[u, shift right=0.7ex, "p"']
 \& \cdots
\end{tikzcd}
\]
On the entry of internal degree $i$ the maps $u$ and $t$ are
\[
 \begin{array}{c|ccc}
  & i<m & i=m & i>m\\\hline
  u & 1 & F & p\\
  t & p & V & 1
 \end{array}
\]
After derived
$p$-completion, $u$ is the identity in weights above the range of $M$
and $t$ in weights below it, so $S(M)^{\infty}\simeq\sigma_*\mathbf
s(M)^\wedge_p$ and $S(M)^{-\infty}\simeq\mathbf s(M)^\wedge_p$.  The
identity $\varphi^*S(M)^{\infty}\simeq\mathbf s(M)^\wedge_p\simeq
S(M)^{-\infty}$ makes $S(M)$ an
$F$-gauge.  We call $S$ \emph{Ekedahl's functor}, since in homotopy
categories it is his complex of $F$-gauges
\cite[Definition~II.3.1]{Ekedahl86Diagonal}.
\end{construction}

The relations of $R$ make $u$ and $t$ chain maps with $ut=tu=p$.  At
the turn this uses $FdV=d$, $VF=p$ and $Vd=pdV$, and away from it the
maps are identities or multiplication by $p$.  Since the interval of internal degrees is finite, this construction
preserves quasi-isomorphisms.  It is a dg functor on such complexes,
and hence induces an exact $\infty$-functor on this range.

\begin{example}[Dieudonn\'e modules and the domino $U_0$]
\label{ex:S-examples}
\begin{enumerate}
\item Let $M$ be a Dieudonn\'e module, regarded as a graded $R$-module
 in internal degree $0$ with $d=0$.  Only the entry of internal degree
 $0$ contributes to \eqref{eq:ladder-complex}, and it is $M$ in weights
 $m\leq0$ and $\sigma_*M$ in weights $m\geq1$.  So $S(M)$ is
 \[
 \begin{tikzcd}[column sep=2.6em]
  \cdots \arrow[r, shift left, "p"]
  & M \arrow[l, shift left, "1"] \arrow[r, shift left, "p"]
  & M \arrow[l, shift left, "1"] \arrow[r, shift left, "F"]
  & \sigma_*M \arrow[l, shift left, "V"] \arrow[r, shift left, "1"]
  & \sigma_*M \arrow[l, shift left, "p"] \arrow[r, shift left, "1"]
  & \cdots \arrow[l, shift left, "p"]
 \end{tikzcd}
 \]
 in weights $-1,0,1,2$, with $u$ pointing right and $t$ pointing left.
 The Frobenius and Verschiebung of $M$ appear only between weights $0$
 and $1$, so $S(M)$ has Hodge--Tate weights $0$ and $1$.
\item Let $U_0:=(R/R(F,Fd))^\wedge$ be the elementary domino of Illusie
 and Raynaud \cite[I, \S2]{IllusieRaynaud83}.  It is killed by $p$, with
 \[
  (U_0)^0=\prod_{j\geq0}k\,V^j,\qquad (U_0)^1=\prod_{j\geq0}k\,dV^j,
 \]
 and $d:V^j\mapsto dV^j$ is bijective.  Hence the row
 \eqref{eq:ladder-complex} is acyclic in every weight except $1$, where
 it is the turn $\sigma_*(U_0)^0\xrightarrow{dV}(U_0)^1$,
 $V^j\mapsto dV^{j+1}$.  This map is injective with cokernel $k\,d$.  So
 $S(U_0[1])$ is
 \[
 \begin{tikzcd}[column sep=2.6em]
  \cdots \arrow[r, shift left, "0"]
  & 0 \arrow[l, shift left, "0"] \arrow[r, shift left, "0"]
  & k \arrow[l, shift left, "0"] \arrow[r, shift left, "0"]
  & 0 \arrow[l, shift left, "0"] \arrow[r, shift left, "0"]
  & \cdots \arrow[l, shift left, "0"]
 \end{tikzcd}
 \]
 in weights $0,1,2$, that is, the pushforward $e_{1,*}\cO(-1)$ from the
 Hodge point.  Both
 computations are also in \cite[Example~2.14]{ZhangHN}.
\end{enumerate}
\end{example}

\begin{lemma}[Ekedahl's functor and Breuil--Kisin twists]
\label{lem:S-BK}
For $M$ as in Construction~\ref{con:ladder} and $a\in\mathbb Z$ there is
a natural equivalence of $F$-gauges $S(M\{a\})\simeq S(M)\{a\}$.
\end{lemma}

\begin{proof}
By the proof of Lemma~\ref{lem:BK-finite}, it suffices to compare
weights.  In \eqref{eq:ladder-complex}, the entry of internal degree $i$
of $M\{a\}$ is the entry of internal degree $i+a$ of $M$, and
$\mathbf s(N\{a\})=\mathbf s(N)$, so $S(M\{a\})^m\simeq S(M)^{m+a}$ with
the same maps $u$ and $t$.  Both sides have ends $\mathbf s(M)$ and
$\sigma_*\mathbf s(M)$, with the identity gluing.
\end{proof}

\begin{lemma}[Ekedahl's functor on a free module]\label{lem:free-ladder}
There is a natural equivalence of $F$-gauges $S(R(-n))\simeq G_n$,
compatible with the right $W$-actions.
\end{lemma}

\begin{proof}
Since $R(-n)=R\{-n\}[-n]$, Lemmas~\ref{lem:S-BK} and~\ref{lem:BK-finite}
reduce us to $n=0$.  There are two internal positions, $0$ and $1$.
Normal forms give a decomposition of the underlying gauge of $S(R)$
into the \emph{Koszul pairs} and the \emph{contractible pairs}
\[
 \bigl[A\,e_b\xrightarrow{p^bt}A(-1)\,h_b\bigr]
 \quad(b\geq0),\qquad
 \bigl[A(-1)\,e_{-j}\xrightarrow{1}A(-1)\,h_{-j}\bigr]
 \quad(j\geq1),
\]
all in complex degrees $0,1$.
Here $e_b=F^b$ sits in weight $0$, while
$e_{-j}=V^{j-1}\in\sigma_*R^0$ and $h_b=F^bd$,
$h_{-j}=dV^j\in R^1$ sit in weight $1$.
For example $ue_b=F^{b+1}$, $te_{-j}=V^j$, and the turning
differential sends $e_{-j}$ to $dV^j$.  These identities verify
the decomposition with its $u,t$-maps.
The right $W$-actions on the Koszul pair $b$ and the contractible
pair $j$ are $\sigma^b$ and $\sigma^{-j}$, respectively.

Project away the contractible pairs and take derived
$p$-completion.  The Koszul pair $b$ becomes
$\Ksz_{A}(p^bt)(-1)[-1]$, and with the usual shift sign its
identification sends $e_b\mapsto-\varepsilon_t$ and $h_b\mapsto1$.

It remains to see that the projection is an equivalence of
$F$-gauges.  On the de Rham locus, the pair labelled by $F^b,F^bd$ for $b\geq1$
matches the pair $b-1$ on the Hodge--Tate locus.
The pair $b=0$, labelled by $1,d$, matches the first
contractible pair there, and the negative pairs
$V^j,dV^j$ match the shifted negative pairs.
Projection therefore gives exactly the reindexing between the two
loci of Proposition~\ref{prop:universal-G}, with its Frobenius
twist.  So the projection is an equivalence of $F$-gauges.
Applying $\{-n\}[-n]$, for general $n$ the Koszul pairs are
$[A(-n)e_b\xrightarrow{p^bt}A(-n-1)h_b]$ in complex degrees $n,n+1$,
and the identification sends $e_b\mapsto(-1)^{n+1}\varepsilon_t$ and
$h_b\mapsto1$.
\end{proof}

\begin{construction}[The right $R$-action on the corepresenting objects]
\label{con:right-action-G}
For $x\in R^k$, right multiplication by $x$ is a map
$R(-n-k)\to R(-n)$, and Ekedahl's functor turns it into a map
$G_{n+k}\to G_n$ (Lemma~\ref{lem:free-ladder}).  These maps form a
\emph{right $R$-action} on the family $\{G_n\}_n$, with all
coherences transported from the dg functor $S$.  Explicitly, include
the Koszul pairs in $S(R(-n))$, apply right multiplication, and project
back.  This gives
\[
 F:\ e_b\mapsto e_{b+1},\quad h_b\mapsto p\,h_{b+1},\qquad
 V:\ e_b\mapsto p\,e_{b-1},\quad h_b\mapsto h_{b-1}\quad(b\geq1),
\]
and $V$ is zero on the pair $b=0$.
For right $d:R(-n-1)\to R(-n)$,
normal forms give
\[
 e_b^{(n+1)}\longmapsto h_b^{(n)},\qquad
 e_{-j}^{(n+1)}\longmapsto p^{j-1}u h_{-j}^{(n)},\qquad
 h_c^{(n+1)}\longmapsto0\quad(c\in\mathbb Z).
\]
Thus right $d$ preserves the contractible pairs, and its
projection is the boundary map $\eta$ on $\Ksz(t)$, as a map of
$F$-gauges.  In the shifted Koszul coordinates it is
$(-1)^{n+2}\eta:G_{n+1}\to G_n$, since
$e_b^{(n+1)}=(-1)^{n+2}\varepsilon_t$ and $h_b^{(n)}=1$.
Since right $V$
is zero on the pair $b=0$, the relation $FV=p$ holds there through
the homotopy $h_0\mapsto ue_0$, which exists because $ut=p$.
\end{construction}

\begin{definition}[The Morita adjunction]
\label{def:geometric-adjunction}
By Construction~\ref{con:right-action-G}, the graded object
$G=(G_n)_n$ is a right $R$-module object of $\Dqc(k^{\syn})$.  The
relative tensor product over $R$ and the graded mapping complex
\cite[\S4.8.4]{LurieHA} give an adjunction
\[
 S_{\mathrm{FG}}=G\otimes_R(-):D(R)\rightleftarrows\Dqc(k^{\syn}):
 T_{\dRW}=\RHomgr(G,-).
\]
Here $G\otimes_RM$ pairs $G_n$ with the internal degree $n$ part of
$M$, and $\RHomgr(G,E)^n=\RHom(G_n,E)$ is the mapping complex in
$\Dqc(k^{\syn})$.  Since $R(-n)$ is free on a generator of internal
degree $n$, $S_{\mathrm{FG}}(R(-n))=G_n$.  We call $T_{\dRW}(E)$ the
\emph{de Rham--Witt realization} of $E$.
\end{definition}

If the $G_n$ were compact generators with graded endomorphism ring $R$,
Morita theory, in the form of Schwede and Shipley
\cite[Theorem~7.1.2.1]{LurieHA} for a family of generators, would make
this adjunction an equivalence.  The $G_n$ generate
(Lemma~\ref{lem:N-conservative}), but they are not compact, since their
underlying gauges are infinite sums (Proposition~\ref{prop:universal-G}).
Correspondingly, Theorem~\ref{thm:presentable-equivalence} identifies
$\Dqc(k^{\syn})$ with the full subcategory $\widehat D(R)$ of complete
objects of $D(R)$.

\begin{lemma}[The Raynaud ring as a tensor algebra]
\label{lem:R-tensor-algebra}
The derived $p$-completion of $R^1\otimes_{R^0}^{L}R^1$ vanishes.  Hence,
after derived $p$-completion, the natural map from the tensor algebra of
the $R^0$-bimodule $R^1$ to $R=R^0\oplus R^1$ is an equivalence.
\end{lemma}

\begin{proof}
Recall that $R^1$ has the $W$-basis $dV^j$ for $j\geq0$ and $F^\ell d$
for $\ell\geq1$ (Remark~\ref{lem:raynaud-normal-form}).  Write
$dV^{-\ell}:=F^\ell d$, so that $dV^{-\ell}\cdot V^\ell=d$ by $FdV=d$.
Then $x\mapsto dx$ extends to an isomorphism of right $R^0$-modules
\[
 R^0[V^{-1}]\xrightarrow{\ \sim\ }R^1,\qquad V^j\longmapsto dV^j\quad(j\in\mathbb Z),
\]
where $R^0[V^{-1}]$ is the localization at the normal element $V$.  As
a right $R^0$-module, $R^0[V^{-1}]$ is the filtered colimit of $R^0$
along left multiplication by $V$, hence flat, so
\[
 R^1\otimes_{R^0}^{L}R^1\simeq R^0[V^{-1}]\otimes_{R^0}R^1\simeq R^1[V^{-1}].
\]
Left multiplication by $F$ is bijective on $R^1$, as $F\cdot dV^j=dV^{j-1}$
\cite[I~(1.1)]{IllusieRaynaud83}, and it commutes with $V$.  So $F$ and
$V$, hence $p=FV$, act invertibly on $R^1[V^{-1}]$, and its derived
$p$-completion vanishes.  The higher tensor powers of $R^1$ contain
$R^1\otimes_{R^0}^{L}R^1$ as a factor, so their completions vanish as
well, which gives the second claim.
\end{proof}

\begin{lemma}[Comparison with Ekedahl's functor]\label{thm:ladder-is-quasi-inverse}
For a derived $p$-complete $M\in D(R)$ in a finite internal
range, there is a natural equivalence $S_{\mathrm{FG}}(M)\simeq S(M)$.
\end{lemma}

\begin{proof}
Both functors are exact on complexes in a fixed finite internal range
and agree on the free modules $R(-n)$, by Lemma~\ref{lem:free-ladder}
and Definition~\ref{def:geometric-adjunction}.  So it suffices to
resolve $M$, naturally and within a finite internal range, by modules
induced from $R^0$.  As the target is derived $p$-complete,
$S_{\mathrm{FG}}$ factors through derived $p$-completion, and we read
the tensor products below in derived $p$-complete modules.  By
Lemma~\ref{lem:R-tensor-algebra}, $R$ is then the tensor algebra of
$R^1$ over $R^0$, so $M$ has the two-term resolution of a module over a
tensor algebra,
\[
 R\widehat\otimes_{R^0}^{L}R^1
    \widehat\otimes_{R^0}^{L}M
 \longrightarrow
 R\widehat\otimes_{R^0}^{L}M
 \longrightarrow M ,
\]
whose first arrow is multiplication minus the action on $M$, and which
is natural in $M$.  If $M$ lies in internal degrees $[a,b]$, its terms
lie in $[a,b+2]$, as $R$ sits in internal degrees $0$ and $1$, and
resolving them over $R^0$, which sits in internal degree zero,
preserves this range.  The resulting induced free terms may have
infinite rank, and $S_{\mathrm{FG}}$ and $S$ agree on them, since both
commute with direct sums in a fixed internal range: $S_{\mathrm{FG}}$
as a left adjoint, and $S$ because, after derived $p$-completion,
$S(M)^m$ is a finite totalization of the entries of $M$.  Both also
commute with finite totalization, which proves the comparison.
\end{proof}

\subsection{Proof of Ekedahl's derived equivalence}
\label{sec:length-one-adjunction}
\label{sec:finite-universal-calculation}

We now prove that the adjunction of
Definition~\ref{def:geometric-adjunction} restricts to an equivalence
onto $\widehat D(R)$.  Write
$\eta:\mathrm{id}\to T_{\dRW}S_{\mathrm{FG}}$ and
$\epsilon:S_{\mathrm{FG}} T_{\dRW}\to\mathrm{id}$ for its unit and
counit.  By the triangle identities it suffices to show two things:
that $T_{\dRW}S_{\mathrm{FG}}$ is Raynaud completion, with $\eta$ the
completion unit, and that $T_{\dRW}$ is conservative.

The first is a Koszul calculation on the charts.  After recalling
Raynaud completion (Definition~\ref{def:raynaud-completion} and
Proposition~\ref{prop:Raynaud-completion}), we compute $T_{\dRW_r}$ on
the corepresenting objects (Lemma~\ref{lem:path-output}), identify $T_{\dRW_r}S_{\mathrm{FG}}$
with the finite reduction $R_r\otimes_R^{L}(-)$ together with its
operators (Proposition~\ref{prop:finite-kernel-calculation}), and pass to $r=\infty$
(Lemma~\ref{lem:limit-action-comparison}).  The second uses the closed
origin and the open substack $\Spf(W)$ of $k^{\syn}$
(Lemma~\ref{lem:N-conservative}).
Theorem~\ref{thm:presentable-equivalence} assembles the two.  We follow
Li and Yang for Raynaud completion.

\begin{definition}[{Raynaud completion; \cite[I~(1.3)]{IllusieRaynaud83} and \cite[Construction~2.24]{LiYang}}]
\label{def:raynaud-completion}
For $M\in D(R)$ put
$\widehat M:=R\varprojlim_r(R_r\otimes_R^{L}M)$.
The system $\{R_r\}$, with its left operators between levels, is a
pro-object with a commuting right $R$-action, so $\widehat M$
carries a graded $R$-action.  An object is \emph{complete} if the
natural map $M\to\widehat M$ is an equivalence, and
$\widehat D(R)\subset D(R)$ denotes
the full subcategory of complete objects.
\end{definition}

\begin{proposition}[{Raynaud completion as a localization;
\cite[Propositions~2.28--2.30]{LiYang}}]
\label{prop:Raynaud-completion}
Completion is a localization of $D(R)$ with essential image
$\widehat D(R)$.  In particular $\widehat M$ is complete for every
$M\in D(R)$, and complete objects are closed under retracts.
\end{proposition}

By \cite[I, Proposition~(3.2)]{IllusieRaynaud83}, the right $R$-module $R_r$ has
the resolution
\begin{equation}\label{eq:Raynaud-finite-resolution}
 R(-1)\xrightarrow{\ x\mapsto(F^rx,-F^rdx)\ }
 R(-1)\oplus R
 \xrightarrow{\ (y,z)\mapsto dV^ry+V^rz\ }R
\end{equation}
in complex degrees $-2,-1,0$.  In particular $R_r\otimes_R^{L}(-)$ commutes with
limits.

Both $T_{\dRW_r}$ and $S_{\mathrm{FG}}$ preserve colimits, so
$T_{\dRW_r}S_{\mathrm{FG}}$ is determined by its values
$T_{\dRW_r}(G_n)$ on the free modules $R(-n)$ together with right
multiplication.  We first compute these values.

\begin{lemma}[The finite realizations of the corepresenting objects]
\label{lem:path-output}
There is an equivalence in $D(\mathrm{GrMod}_{W_r[d]/d^2})$
\[
 T_{\dRW_r}(G_n)\simeq R_r(-n),
\]
compatible with the right action of $R$.  In particular
$T_{\dRW_r}(G_n)$ is concentrated in complex degree zero.
\end{lemma}

\begin{proof}
Let $M$ be a complex of graded $R$-modules in a finite internal range.
We first compute the $r$-gauge $E:=g_{r,*}f_r^*S(M)$, and then its fibre
over $u=0$, whose associated cochain complex is $T_{\dRW_r}(S(M))$
(Definition~\ref{def:finite-extraction}).

By Construction~\ref{con:f-pullback}, $f_r^*S(M)$ is the genuine
$r$-gauge with $r$ rows, each a copy of $S(M)$, in which the de Rham end
$S(M)^{-\infty}$ of each row is glued to the Hodge--Tate end
$S(M)^{\infty}$ of the next.  By the \v{C}ech formula
\eqref{eq:g-pushforward}, $E^m$ therefore consists of the tuples
$(x_0,\dots,x_{r-1})$ in $S(M)^m$ with $t^\infty x_i=u^\infty x_{i+1}$,
where $t^\infty$ and $u^\infty$ are the composites of $t$ and $u$ from
weight $m$ to the two ends.  On the entry of internal degree $j$ they
are (Construction~\ref{con:ladder})
\[
 t^\infty=p^{m-j-1}V,\ \ u^\infty=1\quad(j<m),\qquad
 t^\infty=1,\ \ u^\infty=p^{j-m}F\quad(j\geq m).
\]
On each entry one of the two is the identity, so the maps in the
\v{C}ech formula are surjective in each degree, and the fibre product
is computed by these tuples.  Below the turn a tuple is thus determined by $x_0$, as
$x_{i+1}=p^{m-j-1}Vx_i$, and from the turn on by $x_{r-1}$, as
$x_i=p^{j-m}Fx_{i+1}$.  So $E^m$ is a single row, made of the first copy
below the turn and of the last copy from the turn on, and its
differential at the turn is $dV$ applied to $x_{r-1}=V^{r-1}x_0$, that
is, $dV^r$.  Since $u$ and $t$ act on the $i$-th row by $p^iu_i$ and
$p^{r-1-i}t_i$ (Proposition~\ref{prop:contraction-map}), they are $1$
and $p^r$ below the turn, $p^r$ and $1$ above it, and $F^r$ and $V^r$
at the turn, where the gluings contribute $F^{r-1}$ and $V^{r-1}$.  Up
to Frobenius twists of coefficients, $E$ is therefore
\[
\begin{tikzcd}[column sep=2.4em, row sep=2.6em, ampersand replacement=\&]
 E^{m+1} \arrow[d, shift right=0.7ex, "t"'] \& \cdots \arrow[r]
 \& \sigma^r_*M^{m-1} \arrow[r, "d"]
   \arrow[d, shift right=0.7ex, "p^r"']
 \& \sigma^r_*M^{m} \arrow[r, "dV^r"]
   \arrow[d, shift right=0.7ex, "V^r"']
 \& M^{m+1} \arrow[r, "d"]
   \arrow[d, shift right=0.7ex, "1"']
 \& \cdots\\
 E^{m} \arrow[u, shift right=0.7ex, "u"'] \& \cdots \arrow[r]
 \& \sigma^r_*M^{m-1} \arrow[r, "dV^r"']
   \arrow[u, shift right=0.7ex, "1"']
 \& M^{m} \arrow[r, "d"']
   \arrow[u, shift right=0.7ex, "F^r"']
 \& M^{m+1} \arrow[r]
   \arrow[u, shift right=0.7ex, "p^r"']
 \& \cdots
\end{tikzcd}
\]
that is, the diagram of Construction~\ref{con:ladder} with $p,F,V$
replaced by $p^r,F^r,V^r$, as befits $ut=p^r$ on $k^{\Nyg_r}$.

Now let $M=R(-n)$, so that $S(M)=G_n$ (Lemma~\ref{lem:free-ladder}).
Then $F^r$, $V^r$ and $p^r$ are injective on $M$, and
$R_r\otimes_RM=M/(V^rM+dV^rM)$.  By
Proposition~\ref{prop:Koszul-Rees-recovery}, it suffices to identify
$E/u$, as a $t$-filtered object, with the Hodge filtration of
$R_r\otimes_RM$ by internal degree, compatibly with right
multiplication.  Dividing by $u$ kills the entries below the turn and
reduces those above it modulo $p^r$, so the $m$-th filtered piece of
$E/u$ is
\[
 \mathbf s\Bigl[\sigma^r_*M^{m-1}/F^rM^{m-1}\xrightarrow{\ dV^r\ }
 M^m/p^rM^m\xrightarrow{\ d\ }M^{m+1}/p^rM^{m+1}
 \xrightarrow{\ d\ }\cdots\Bigr],
\]
with $t$ equal to $V^r$ at the turn and to $1$ above it.  The
projections onto $M^i/(V^rM^i+dV^rM^{i-1})$ kill $V^r$ and $dV^r$, so
they define a map of filtered objects to the Hodge filtration of
$R_r\otimes_RM$, natural in $M$.  As both filtrations vanish in high
weights, it suffices to check this map on graded pieces.  Here $t$ is
injective, since at the turn it is
$V^r:\sigma^r_*M^m/F^rM^m\to M^m/p^rM^m$, which is injective as
$p^r=V^rF^r$ and $V^r$ is injective.  So dividing by $t$ is the naive
quotient, which kills the entries above the turn and leaves
\[
 \sigma^r_*M^{m-1}/F^rM^{m-1}\xrightarrow{\ dV^r\ }M^m/V^rM^m
\]
in internal degrees $m-1$ and $m$.  By the exactness of
\eqref{eq:Raynaud-finite-resolution}, its only cohomology is
$M^m/(V^rM^m+dV^rM^{m-1})$, in degree zero, which is the $m$-th graded
piece of the Hodge filtration.  Here dividing by $t$ removes $V^r$, and
$dV^r$ is the differential that remains.
\end{proof}

\begin{remark}[The canonical class]
\label{rem:canonical-class}
We normalize the equivalence of Lemma~\ref{lem:path-output} by the
canonical maps $G_{r,n}\to G_n$.  Let $c\in T_{\dRW_r}(G_n)^n$ be the
class of this map (Definition~\ref{def:representing-objects}).  As it
includes a finite stage into the colimit, it projects each chart onto
its $u$-generator (proof of Proposition~\ref{prop:universal-G}).  Under
the orientation of Lemma~\ref{lem:e-finite-oriented}, this projection
becomes the unit of $\Ksz(u)$, and the identity of the two copies of
$\Ksz(t)$ becomes their diagonal class, which on the last chart is the
generator $1$ of $R(-n)$ at the turn, up to sign
(Lemma~\ref{lem:free-ladder}).  The sign does not depend on $r$, as $c$
and $1$ are both compatible with restriction, and we fix the sign of the
equivalence of Proposition~\ref{thm:oriented-finite-extraction} in
weight $n$ so that it is $+1$.  As the equivalence of
Lemma~\ref{lem:path-output} respects $d$ and the right action of $R$,
the map $x\mapsto x_*c$, for $x\in R^k$ acting as $G_{n+k}\to G_n$, is
then the projection $R(-n)\to R_r(-n)$ followed by the inverse of that
equivalence.  So it is an isomorphism $R_r(-n)\simeq T_{\dRW_r}(G_n)$
of graded $W_r[d]/d^2$-modules and of right $R$-modules under which $c$
corresponds to $1$, and the right action of $d$ corepresents the
operator $d$.  We use this isomorphism from now on.
\end{remark}

\begin{proposition}[{The composite \texorpdfstring{$T_{\dRW_r}S_{\mathrm{FG}}$}{T\_dRW\_r S\_FG}}]
\label{prop:finite-kernel-calculation}
There is a natural equivalence
\[
 T_{\dRW_r}S_{\mathrm{FG}}
 \simeq R_r\otimes_R^{L}(-):D(R)\longrightarrow D(\mathrm{GrMod}_{W_r[d]/d^2}).
\]
It respects the transition maps and the operators $F,V,d$ between
finite levels.  Under this equivalence the unit of
Definition~\ref{def:geometric-adjunction}, projected to level $r$, is
$M\to R_r\otimes_R^{L}M$.
\end{proposition}

\begin{proof}
Both sides preserve colimits, so they are determined by their values
on the free modules $R(-n)$ together with the right action of $R$
\cite[Theorem~4.8.4.1]{LurieHA}.  On these they agree as right
$R$-modules by Lemma~\ref{lem:path-output}, through the projected unit
$x\mapsto x_*c$ (Remark~\ref{rem:canonical-class}).  These values are concentrated in complex degree zero,
so the $\infty$-category they span is the nerve of an ordinary category
\cite[Propositions~2.3.4.5 and~2.3.4.18]{LurieHTT}, and a diagram of
them is determined by the underlying maps of modules.

It remains to compare the operators between levels: restriction,
Frobenius and its transpose $V$, as Lemma~\ref{lem:path-output} already
respects $d$.  On
$T_{\dRW_r}(G_n)$ they are natural in the $F$-gauge, so they commute
with the right action of $R$, which is postcomposition.  On $R_r(-n)$
they are left multiplication by $1,F,V$, which is right $R$-linear.
Since $R_r(-n)$ is generated by $1=c$ as a right $R$-module, it
suffices to check
\[
 \pi(1)=1,\qquad F(1)=F,\qquad V(1)=V.
\]
In the proof of Lemma~\ref{lem:path-output}, a class at the turn is
represented on the last chart.  Restriction forgets the first chart, so
$\pi(1)=1$.  Frobenius forgets the last chart, and the gluing replaces
the representative $x$ on the last chart by $Fx$ on the one before, so
$F(1)=F$.  Finally,
$V(1)\cdot F=V(F(1))=p$ by Remark~\ref{rem:local-raynaud-relations},
and $V$ is the only element $z$ of $R_{r+1}^0$ with $zF=p$, so
$V(1)=V$.  Extending from free modules gives the equivalence, and
the statement on the unit holds because $c$ corresponds to $1$
(Remark~\ref{rem:canonical-class}).
\end{proof}

We now pass to $r=\infty$.

\begin{lemma}[{The composite \texorpdfstring{$T_{\dRW}S_{\mathrm{FG}}$}{T\_dRW S\_FG}}]\label{lem:limit-action-comparison}
For every $F$-gauge $E$, the identification
$T_{\dRW}(E)^n\simeq R\varprojlim_rT_{\dRW_r}(E)^n$ of
Definition~\ref{def:representing-objects} carries the right action of
$R$ on $\{G_n\}_n$ to the operators $F,V,d$ of
Construction~\ref{def:intrinsic-geometric-V}.  Consequently there is an
equivalence of functors
\begin{equation}\label{eq:NL-completion}
 T_{\dRW}S_{\mathrm{FG}}\simeq\widehat{(-)}
\end{equation}
under which the unit $\eta$ is the completion unit.
\end{lemma}

\begin{proof}
By Lemma~\ref{lem:restriction-operator-compatibility}, $F$, $V$ and $d$
are natural transformations of the tower $\{T_{\dRW_r}\}_r$.  By the
Yoneda lemma they correspond to
$\Zp$-linear maps between the $G_n$, as do the elements of $R$.  These
mapping spaces are discrete, since
\[
 \Map_{\Dqc(k^{\syn})}(G_m,G_n)
 \simeq\varprojlim_r\Map_{\Dqc(k^{\syn})}(G_{r,m},G_n)
\]
and each term on the right is discrete by Lemma~\ref{lem:path-output}.
So these objects span the nerve of an ordinary category
\cite[Propositions~2.3.4.5 and~2.3.4.18]{LurieHTT}, and each such map,
and every coherence among such maps and their finite sums, is
determined by the image of $\mathrm{id}_{G_n}$, which is $1$ at
every finite level by Remark~\ref{rem:canonical-class}.  These images agree
by Proposition~\ref{prop:finite-kernel-calculation}, which proves the
first assertion.  The inverse limit of
Proposition~\ref{prop:finite-kernel-calculation}, including its
statement on the unit, gives \eqref{eq:NL-completion}.
\end{proof}

\begin{lemma}[Conservativity]\label{lem:N-conservative}
The functor $T_{\dRW}:\Dqc(k^{\syn})\to D(R)$ is conservative.
\end{lemma}

\begin{proof}
Let $j_\Delta:\Spf(W)\hookrightarrow k^{\syn}$ be the inclusion of the
open substack (Definition~\ref{def:syntomification}), and consider the
localization sequence
\[
 \ker(j_\Delta^*)\longrightarrow\Dqc(k^{\syn})
 \xrightarrow{\ j_\Delta^*\ }\Dqc(\Spf W),
\]
where $\Dqc(\Spf W)$ is the category of derived $p$-complete
$W$-complexes.  Under Ekedahl's functor, $j_\Delta^*$ corresponds to the
$p$-completed simple functor (Definition~\ref{def:simple-functor}),
since $S(M)^{-\infty}\simeq\mathbf s(M)^\wedge_p$
(Construction~\ref{con:ladder}).  We show
that $T_{\dRW}$ detects objects at both ends.  Let $E$ be an $F$-gauge
with $T_{\dRW}(E)=0$.

The kernel consists of the objects supported at the Hodge point, and we
probe it with the skyscrapers $G_{1,n}$, which come from the domino
$U_0$.  Indeed, since $e_{1,*}\cO(-n)=e_{1,*}\cO(-1)\{1-n\}$,
Example~\ref{ex:S-examples}(2), Lemma~\ref{thm:ladder-is-quasi-inverse}
and Lemma~\ref{lem:S-BK} give
\[
 G_{1,n}=e_{1,*}\cO(-n)[-n-2]\simeq S_{\mathrm{FG}}\bigl(U_0(1-n)[-2]\bigr),
\]
so $\RHom(G_{1,n},E)\simeq\RHom_R(U_0(1-n)[-2],T_{\dRW}(E))=0$ for every
$n$.  Let $F$ be the fibre of $E\to j_{\Delta,*}j_\Delta^*E$, which lies
in the kernel.  As $j_\Delta^*G_{1,n}=0$, also $\RHom(G_{1,n},F)=0$.
Since $F$ vanishes on the de Rham and Hodge--Tate loci, along which
$k^{\syn}$ is glued, we may compute after pullback along $j_{\Nyg}$.
As $F$ is derived $p$-complete, it suffices to show $F/p=0$, and
$\RHom(G_{1,n},F/p)=0$ for all $n$ as well.  The vanishing on the
formal open loci says that $(F/p)[1/u]=(F/p)[1/t]=0$, so the cohomology
of $F/p$ is $u$-power and $t$-power torsion.  By
Proposition~\ref{thm:oriented-finite-extraction} and
Lemma~\ref{lem:T-r-Hodge}, $\RHom(G_{1,n},F/p)\simeq
T_{\dRW_1}(F/p)^n$ is a shift of the weight-$n$ part of
$(F/p)_{\mathrm{Hodge}}=(F/p)/^{L}(u,t)$, so its vanishing for all $n$
says that $(F/p)_{\mathrm{Hodge}}=0$.  Let $C$ be the cofibre of $u$ on
$F/p$.  Then $t$ is invertible on $C$, whose cohomology is $t$-power
torsion, so $C=0$.  Thus $u$ is invertible on $F/p$, whose cohomology
is $u$-power torsion, so $F/p=0$.  Hence
$E\simeq j_{\Delta,*}j_\Delta^*E$.

We probe the quotient with $j_\Delta^*G_n\simeq\mathbf s(R(-n))^\wedge_p$.  The
only cohomology of $\mathbf s(R(-n))$ is
$R^1/dR^0\simeq\bigoplus_{\ell\geq1}W/p^\ell$, spanned by the
$F^\ell d$, in degree $n+1$, so
\[
 j_\Delta^*G_n\simeq\Bigl(\bigoplus_{\ell\geq1}W/p^\ell\Bigr)^\wedge_p[-n-1],
\]
which contains $W/p[-n-1]$, the summand $\ell=1$, as a retract.  As
$\RHom(G_n,E)=0$ and $E\simeq j_{\Delta,*}j_\Delta^*E$, adjunction shows
that $\RHom_W(W/p,j_\Delta^*E)$ vanishes, so $j_\Delta^*E=0$ by derived
$p$-completeness.  Hence $E=0$.

On perfect $F$-gauges, the kernel is the category of finite-length
$F$-gauges, and the correspondence above matches it with Ekedahl's
$\mathbf s$-acyclic diagonal dominoes
\cite[Definition~III.2.4]{Ekedahl86Diagonal}.
\end{proof}

\begin{theorem}[Ekedahl's derived equivalence]
\label{thm:presentable-equivalence}
\label{prop:affine-Cech-Raynaud-counit}
The de Rham--Witt realization is an equivalence of stable $\infty$-categories
\[
 T_{\dRW}:\Dqc(k^{\syn})\xrightarrow{\sim}\widehat D(R),
\]
with inverse the restriction of $S_{\mathrm{FG}}$.  For every $F$-gauge
$E$ and every $r\geq1$ there is an equivalence
\begin{equation}\label{eq:raynaud-reduction}
 R_r\otimes_R^{L}T_{\dRW}(E)\xrightarrow{\sim}T_{\dRW_r}(E),
\end{equation}
natural in $E$ and compatible with $\pi$, $F$, $V$ and $d$.
\end{theorem}

\begin{proof}
By \eqref{eq:NL-completion} the unit $\eta$ is the completion unit, so
$\eta_M$ is an equivalence for complete $M$.  By the triangle identity
$T_{\dRW}(\epsilon_E)\circ\eta_{T_{\dRW}(E)}=\mathrm{id}$, the object
$T_{\dRW}(E)$ is a retract of its completion, hence complete
(Proposition~\ref{prop:Raynaud-completion}).  So $\eta_{T_{\dRW}(E)}$,
and with it $T_{\dRW}(\epsilon_E)$, is an equivalence, and $\epsilon_E$
is an equivalence by conservativity (Lemma~\ref{lem:N-conservative}).

Finally, applying $T_{\dRW_r}$ to $\epsilon_E$ and
Proposition~\ref{prop:finite-kernel-calculation} to $M=T_{\dRW}(E)$
gives
\[
 R_r\otimes_R^{L}T_{\dRW}(E)\simeq
 T_{\dRW_r}S_{\mathrm{FG}}T_{\dRW}(E)
 \xrightarrow[\sim]{\ T_{\dRW_r}(\epsilon_E)\ }T_{\dRW_r}(E),
\]
compatibly with $\pi$, $F$, $V$ and $d$.
\end{proof}

\begin{corollary}[Breuil--Kisin twists at $r=\infty$]
\label{lem:Koszul-BK-shift}
For $E\in\Dqc(k^{\syn})$ and $a\in\mathbb Z$,
$T_{\dRW}(E\{a\})\simeq T_{\dRW}(E)\{a\}$.
\end{corollary}

\begin{proof}
The equivalences of Lemma~\ref{lem:BK-finite} are internal shifts, so
they are compatible with $\pi$, $F$, $V$ and $d$.  Hence
\eqref{eq:raynaud-reduction} and completeness give the claim.
\end{proof}

\subsection{Perfect \texorpdfstring{$F$}{F}-gauges and \texorpdfstring{$F$}{F}-gauge cohomology}
\label{sec:perfect-coefficients}

In this subsection, we show that Ekedahl's derived equivalence
(Theorem~\ref{thm:presentable-equivalence}) identifies perfect
$F$-gauges with $D_c^b(R)$ (Corollary~\ref{thm:gauge-raynaud-equivalence})
and sends the $F$-gauge cohomology of a smooth proper $k$-scheme to its
de Rham--Witt cohomology
(Corollary~\ref{cor:geometric-raynaud-realization}).

\begin{definition}[{Coherent complexes; \cite[Definition~0.5.13]{Ekedahl85}}]
\label{def:coherent-complexes}
\label{rem:coherence-criterion}
A bounded complex $M\in D(R)$ is \emph{coherent} if it is complete and
$R_1\otimes_R^{L}M$ has finite-dimensional total cohomology after
forgetting the internal grading.  We write $D_c^b(R)\subset D(R)$ for
the full subcategory of coherent complexes.
\end{definition}

By \cite[Proposition~III.1.1]{Ekedahl85}, these are the bounded
complexes whose cohomology modules are coherent in the sense of Illusie
and Raynaud \cite[I, (3.9)]{IllusieRaynaud83}.  In particular a coherent
complex has finite internal range, and it is derived $p$-complete, as
the derived limit of the complexes $R_r\otimes_R^{L}M$ killed by $p^r$.

\begin{lemma}[Nakayama's lemma on $k^{\syn}$]
\label{lem:graded-perfect-lifting}
Let $E$ be an $F$-gauge, with fibre $E_{\mathrm{Hodge}}$ at the Hodge
point (Lemma~\ref{lem:T-r-Hodge}).
\begin{enumerate}
\item If $E$ is perfect, its underlying gauge is derived
 $(u,t)$-complete.
\item If the underlying gauge of $E$ is bounded above and derived
 $(u,t)$-complete, and $E_{\mathrm{Hodge}}$ has finite-dimensional
 total cohomology, then $E$ is perfect.
\item Let $f$ be a map of $F$-gauges whose underlying gauges are
 derived $(u,t)$-complete.  If $f_{\mathrm{Hodge}}$ is an equivalence,
 so is $f$.
\end{enumerate}
\end{lemma}

\begin{proof}
Perfectness and equivalences can be checked on underlying gauges, that
is, after pullback along the \'etale cover $j_{\Nyg}$, and an underlying
gauge is a graded $A$-module.  For (1), in each weight of $A(-n)$ the
$(u,t)$-adic and $p$-adic filtrations are cofinal.  Part (3) is derived
Nakayama \cite[Tag~0G1U]{StacksProject}, whose proof applies to graded
modules.  For (2), let $Q$ be the underlying gauge of $E$.  Complete
Nakayama shows that the top cohomology of $Q$ sits in the top degree
$b$ of $Q/^{L}(u,t)$.  A homogeneous basis of $H^b(Q/^{L}(u,t))$ lifts
to a map from a finite graded free module to $Q$, whose cofibre has one
fewer nonzero cohomology group after reduction modulo $(u,t)$.  After
finitely many steps the reduction vanishes, so the cofibre does too,
and $Q$ is a finite extension of finite free modules.
\end{proof}

\begin{lemma}[Uniform weight bounds]
\label{lem:finite-weight-bounds}
Let $E$ be an $F$-gauge whose underlying gauge is, up to retract, a
bounded complex of completed direct sums of free graded $A$-modules
$A(-m)$ with $a\leq m\leq b$.  Then $T_{\dRW_r}(E)^n=0$ unless
$a-1\leq n\leq b+1$, and the complex degrees of $T_{\dRW_r}(E)$ lie in
a range independent of $r$.  In particular, $T_{\dRW}(E)$ is derived
$p$-complete and bounded in both gradings.  This applies to perfect
$E$.
\end{lemma}

\begin{proof}
On $A(-m)$, the map $u$ is an isomorphism in weights at least $m$ and
$t$ in weights at most $m$, so on the underlying gauge of $E$, $u$ is an
isomorphism in weights above $b$ and $t$ in weights below $a$.  Compute
$T_{\dRW_r}(E)$ by the \v{C}ech complex of the chain of $r$ Nygaard
charts \eqref{eq:g-pushforward}, tensored with $\Ksz(u)\otimes\Ksz(t)$.
Let $n\geq b+2$.  For $0\leq i\leq r-2$, restriction from the
$(i+1)$-st chart to the $i$-th overlap, its locus $u_{i+1}\neq0$, is an
isomorphism in weights $n-1$ and $n$ compatible with the Koszul
factors, so these two terms cancel.  The remaining term is that of the
$0$-th chart, where $\Ksz(u)=[\cO(-1)\xrightarrow{u_0}\cO]$ is acyclic
in weight $n$.  Symmetrically, for $n\leq a-2$ the $i$-th chart cancels
against the $i$-th overlap through $t_i\neq0$, and $\Ksz(t)$ is acyclic
on the last chart.  So $T_{\dRW_r}(E)^n=0$ unless $a-1\leq n\leq b+1$.
Before regrading, the complex degrees lie in a range independent of
$r$, since $E$ is bounded and the Koszul and \v{C}ech complexes have
length at most three.  The regrading shifts weight $n$ by $[n]$, and
only $a-1\leq n\leq b+1$ occur, so the same holds after regrading.
Both bounds pass to
$T_{\dRW}(E)^n\simeq R\varprojlim_rT_{\dRW_r}(E)^n$
(Definition~\ref{def:representing-objects}), with one more degree for
the derived limit, which is derived $p$-complete as a limit of
complexes killed by $p^r$.  A perfect $F$-gauge has underlying gauge a
retract of a finite complex of finite free graded $A$-modules.
\end{proof}

\begin{corollary}[Perfect $F$-gauges and coherent complexes of $R$-modules]
\label{thm:gauge-raynaud-equivalence}
The equivalence of Theorem~\ref{thm:presentable-equivalence}
restricts to an equivalence
\[
 T_{\dRW}:\Perf(k^{\syn})\xrightarrow{\sim}D_c^b(R),
\]
and for $M\in D_c^b(R)$ the inverse $S_{\mathrm{FG}}(M)$ is
$S(M)$, with $S$ Ekedahl's functor of Construction~\ref{con:ladder}.
\end{corollary}

\begin{proof}
Let $E$ be perfect and put $M=T_{\dRW}(E)$.  Then $M$ is complete by
Theorem~\ref{thm:presentable-equivalence} and bounded by
Lemma~\ref{lem:finite-weight-bounds}.  By
\eqref{eq:raynaud-reduction} and Lemma~\ref{lem:T-r-Hodge},
$R_1\otimes_R^{L}M\simeq T_{\dRW_1}(E)\simeq E_{\mathrm{Hodge}}[\bullet]$,
which has finite-dimensional total cohomology.  So $M\in D_c^b(R)$ by
Definition~\ref{def:coherent-complexes}.

Conversely, let $M\in D_c^b(R)$.  By
the discussion after Definition~\ref{def:coherent-complexes}, $M$ is
derived $p$-complete with
finite internal range, so $S_{\mathrm{FG}}(M)\simeq S(M)$ is a bounded
$F$-gauge by Lemma~\ref{thm:ladder-is-quasi-inverse}.  Its underlying
gauge is derived $(u,t)$-complete: below the internal range of $M$ it
has $u=p$ and $t=1$, and above it $u=1$ and $t=p$, so the derived limits
along $u$ towards weight $-\infty$ and along $t$ towards weight
$+\infty$ are limits along multiplication by $p$, so they vanish by
derived $p$-completeness, and this is derived $(u,t)$-completeness
\cite[Tags~091S and 091Q]{StacksProject}.  Its fibre at the Hodge point
is, up to regrading, $T_{\dRW_1}S_{\mathrm{FG}}(M)\simeq R_1\otimes_R^{L}M$
(Lemma~\ref{lem:T-r-Hodge} and
Proposition~\ref{prop:finite-kernel-calculation}), which has
finite-dimensional total cohomology.  So $S_{\mathrm{FG}}(M)$ is perfect
by Lemma~\ref{lem:graded-perfect-lifting}(2).
\end{proof}

We now compute $T_{\dRW}$ on the $F$-gauge cohomology $\cH_{\syn}(X)$.
The point is that Ekedahl's functor turns a $p$-torsion-free graded
$R$-module into the Rees object of its Nygaard filtration.  For a
$p$-torsion-free Dieudonn\'e module $M$ in internal degree $0$, this
filtration is $M\supset VM\supset pVM\supset\cdots$, and we recover
Example~\ref{ex:S-examples}(1).  For $W\Omega^\bullet_{X/k}$, the
weight-$m$ row of $S$ is the modification of $W\Omega^\bullet_{X/k}$
introduced by Nygaard \cite[Definition~1.1]{Nygaard81}, and it becomes
the $m$-th Nygaard filtered piece.

\begin{lemma}[Dieudonn\'e complexes and the Nygaard filtration]
\label{lem:saturated-ladder}
Let $C$ be a $p$-torsion-free, derived $p$-complete graded $R$-module in
a finite internal range, with Nygaard filtration
\[
 (\Fil^{\geq m}_{\Nyg}C)^i=
 \begin{cases}
 p^{m-i-1}VC^i,&i<m,\\
 C^i,&i\geq m,
 \end{cases}
\]
as in \cite[Construction~8.1.1]{BLM21}.  Then there is a natural
equivalence
\[
 S_{\mathrm{FG}}(C)\simeq\Rees(\Fil^{\geq\bullet}_{\Nyg}C),
\]
where the Rees object is an $F$-gauge through the divided Frobenius.
In particular, for $X$ smooth over $k$,
$\cH_{\syn}(X)\simeq R\Gamma(X,S(W\Omega^\bullet_{X/k}))$, with $S$
applied sheafwise.
\end{lemma}

\begin{proof}
By Lemma~\ref{thm:ladder-is-quasi-inverse} it suffices to identify
$S(C)$.  Consider the maps
\[
 S(C)^{m,i}\longrightarrow(\Fil^{\geq m}_{\Nyg}C)^i,\qquad
 x\longmapsto
 \begin{cases}
 p^{m-i-1}Vx,&i<m,\\
 x,&i\geq m.
 \end{cases}
\]
For $i<m$ the source is $\sigma_*C^i$, so the map is $W$-linear, and it
is a bijection onto $p^{m-i-1}VC^i$ as $C$ is $p$-torsion-free.  These
maps commute with the row differentials, by $Vd=pdV$ below the turn and
by the definition of the turn $dV$.  They also commute with $u$ and
$t$, which act on the Rees object by $p$ and by inclusion: away from
$i=m$ the power of $p$ absorbs the factors $1$ and $p$, and at $i=m$
the relevant identities are $VF=p$ and $t=V$.  Finally, the two
$F$-gauge structures agree.  Below the internal range of $C$ the map is
the identity, and above it the divided Frobenius sends $p^{m-i-1}Vx$ to
$x$, since $FV=p$.

Now let $X$ be smooth over $k$.  The formulas above are compatible with
\'etale localization.  Moreover, $W\Omega^\bullet_{X/k}$ is saturated
\cite[Theorem~4.4.12]{BLM21}, and it is derived $p$-complete as the
limit of the $W_r\Omega^\bullet_{X/k}$, so they apply to
$W\Omega^\bullet_{X/k}$.  Its Nygaard filtration is complete with the
graded pieces of \cite[Proposition~8.2.1]{BLM21}, and these
characterize Bhatt's Nygaard filtration
\cite[Theorem~3.2.1]{BhattFGauges}.  So
Theorem~\ref{thm:ordinary-input}(1) identifies its derived global
sections with $\cH_{\Nyg}(X)$.  The $F$-gauge structure of
$\cH_{\syn}(X)$ induces the divided Frobenius on this filtration
\cite[Proposition~4.4.2]{BhattFGauges}, and all these identifications
are multiplicative.
\end{proof}

\begin{corollary}[The de Rham--Witt realization of $F$-gauge cohomology]\label{cor:geometric-raynaud-realization}
Let $X$ be a smooth proper $k$-scheme.  Then
\[
 T_{\dRW}(\cH_{\syn}(X))
 \simeq R\Gamma(X,W\Omega^\bullet_{X/k}),
\]
and its finite reductions are
$R_r\otimes_R^{L}T_{\dRW}(\cH_{\syn}(X))\simeq
R\Gamma(X,W_r\Omega^\bullet_{X/k})$, compatibly with $\pi$, $F$, $V$
and $d$.
\end{corollary}

\begin{proof}
We compare sheaves, which identifies all four operators, Verschiebung
included, together with their coherences.  By
Lemma~\ref{lem:saturated-ladder}, $\cH_{\syn}(X)$ is the derived global
sections of $S(W\Omega^\bullet_{X/k})$.  By
Proposition~\ref{prop:finite-kernel-calculation} and
Lemma~\ref{lem:path-output}, the finite de Rham--Witt realization of
$S(W\Omega^\bullet_{X/k})$ is $R_r\otimes_R^{L}W\Omega^\bullet_{X/k}$,
with the overlap component of $V$.  By
Proposition~\ref{prop:sheaf-fibres}, this reduction is sheafwise
concentrated in complex degree zero, where
\[
 H^0(R_r\otimes_R^{L}W\Omega^\bullet_{X/k})
 =W\Omega^\bullet_{X/k}/(V^rW\Omega^\bullet_{X/k}+dV^rW\Omega^\bullet_{X/k})
 =W_r\Omega^\bullet_{X/k}.
\]
On this quotient the operators of the system $\{R_r\}$, which
Proposition~\ref{prop:finite-kernel-calculation} matches with
$\pi,F,V,d$, induce the classical restriction, Frobenius, Verschiebung
and differential.  As the form degree has finite range, these
identifications can be totalized over an \'etale resolution and passed
to derived global sections.  Finally,
$W\Omega^\bullet_{X/k}=\varprojlim_rW_r\Omega^\bullet_{X/k}$, and
derived global sections commute with this limit.
\end{proof}

\section{Multiplicative structures and duality}
\label{sec:duality-and-tensor}

In this section, we show that Ekedahl's derived equivalence respects
tensor products and duality.  The key input in both cases is the cup
product on the finite de Rham--Witt realizations $T_{\dRW_r}$, which
comes from the lax monoidal structure of pushforward.  In
\S\ref{sec:corrected-monoidality} we observe that cup products obey
Ekedahl's rules for the star product, and deduce that $T_{\dRW}$ is
symmetric monoidal by reducing to level one, where both sides are
monoidal (Theorem~\ref{thm:symmetric-monoidal-equivalence}).  In
\S\ref{sec:geometric-ekedahl-duality} we use cup products to identify
the duality of perfect $F$-gauges with Ekedahl's dualizing functor
(Theorem~\ref{cor:duality-comparison}), and deduce Poincar\'e duality
for de Rham--Witt cohomology from Poincar\'e duality for $F$-gauge cohomology
(Corollary~\ref{thm:geometric-Ekedahl-duality}).  These complete the
proof of Theorem~\ref{thm:intro-equivalence} and prove
Theorem~\ref{thm:intro-duality}.

As in Section~\ref{sec:ekedahl-monoidal}, we write
$A=(W[u,t]/(ut-p))^\wedge_p$ for the coordinate ring of $k^{\Nyg}$ in
this section.

\subsection{Tensor products}
\label{sec:corrected-monoidality}

In this subsection, we show that Ekedahl's derived equivalence is
symmetric monoidal, for the usual tensor product of $F$-gauges and
Ekedahl's star product of complexes of $R$-modules.  As an application, we
compare the K\"unneth formulas on the two sides
(Corollary~\ref{cor:geometric-kunneth}).  We begin by recalling the
star product.

\begin{definition}[{The star product; \cite[Definition~I.3.1]{Ekedahl85}}]
\label{def:star-product}
Let $M,N,P$ be graded $R$-modules.  A family of $W$-bilinear maps
$\beta:M^i\times N^j\to P^{i+j}$ is \emph{$(F,V,d)$-bilinear} if, for
$x\in M^i$ and $y\in N^j$,
\[
\begin{aligned}
 F\beta(x,y)&=\beta(Fx,Fy),\\
 V\beta(x,Fy)&=\beta(Vx,y),&
 V\beta(Fx,y)&=\beta(x,Vy),\\
 d\beta(x,y)&=\beta(dx,y)+(-1)^{i}\beta(x,dy).
\end{aligned}
\]
The \emph{star product} $M\star N$ is the graded $R$-module
corepresenting $(F,V,d)$-bilinear maps, so that
\[
 \Hom_R(M\star N,P)\simeq
 \{(F,V,d)\text{-bilinear maps }M\times N\to P\}
\]
naturally in $P$, where $\Hom_R$ denotes maps of graded $R$-modules.
We write $x\star y$ for the image of $(x,y)$ under the universal
$(F,V,d)$-bilinear map $M\times N\to M\star N$.  The unit of $\star$ is
$W$ in internal degree zero, with $F=\sigma$, $V=p\sigma^{-1}$ and
$d=0$.
\end{definition}

The star product extends a classical construction.  A Dieudonn\'e
module over $k$ is a module over $R^0$, the semilinear ring generated
by $F$ and $V$ over $W$.  For Dieudonn\'e modules, Antieau and Nikolaus
call a $W$-bilinear map \emph{$(V,F)$-bilinear} if it satisfies the
relations above involving $F$ and $V$, and the box product
$\boxtimes_W$ corepresents such maps, with unit $(W,\sigma,p\sigma^{-1})$
\cite[\S4]{AntieauNikolaus}.  The star product is thus the box product
for graded modules with $d$, the remaining relation being the Leibniz
rule.

\begin{proposition}[The star product in internal degree zero]
\label{prop:star-box-degree-zero}
For graded $R$-modules $M,N$ concentrated in nonnegative internal
degrees, the universal $(F,V,d)$-bilinear map induces a natural
isomorphism
\[
 (M\star N)^0\simeq M^0\boxtimes_W N^0.
\]
\end{proposition}

\begin{proof}
The generators of $M\star N$ in internal degree zero are the
$x\star y$ with $x\in M^0$ and $y\in N^0$.  Among them, the relations
are $W$-bilinearity and the three relations involving $F$ and $V$.
The relations involving $d$ lie in positive internal degree, and
multiplication by $R$ does not lower the internal degree.  This is the
presentation of the box product.
\end{proof}

In positive internal degree the star product of two Dieudonn\'e
modules can be nonzero, so it carries more than the box product.

\begin{definition}[{Derived and completed products; \cite[\S I.4]{Ekedahl85} and \cite[\S3.2]{LiYang}}]
\label{def:derived-star}
The symmetric monoidal structure
$(D(R),\star^{L})$ is constructed from the
star product on free modules by monoidal cocompletion and
stabilization \cite[Constructions~3.27--3.32]{LiYang}.  On homotopy
categories it extends Ekedahl's derived star product, which he defined
when one input is bounded above \cite[Proposition~I.4.5]{Ekedahl85}
\cite[Remark~3.33]{LiYang}.  Completion is a symmetric monoidal
localization, with product $M\widehat\star^{L}N:=\widehat{\,M\star^{L}N\,}$
on $\widehat D(R)$ \cite[Proposition~3.35]{LiYang}.
\end{definition}

On complexes of Dieudonn\'e modules, the derived box product
$\boxtimes_W^{L}$ is computed by free $R^0$-resolutions, and the
analogue of completion is derived $V$-adic completion,
\[
 C^\wedge_V:=R\varprojlim_r\cofib\bigl(V^r:\sigma^r_*C\to C\bigr),
\]
with transition maps given by $V$ on the source and the identity on
the target.  The Frobenius twist makes $V^r$ $W$-linear.

\begin{proposition}[Derived and completed products in internal degree zero]
\label{prop:derived-star-box}
For complexes $M,N$ of graded $R$-modules concentrated in nonnegative
internal degrees, there are natural equivalences
\[
 (M\star^{L}N)^0\simeq M^0\boxtimes_W^{L}N^0,\qquad
 (M\widehat\star^{L}N)^0\simeq\bigl(M^0\boxtimes_W^{L}N^0\bigr)^\wedge_V,
\]
where the superscript $0$ takes internal degree zero and keeps the
complex degree.
\end{proposition}

\begin{proof}
Resolve $M$ and $N$ by free graded $R$-modules with generators in
nonnegative internal degrees.  In internal degree zero these become
free $R^0$-resolutions of $M^0$ and $N^0$, and
Proposition~\ref{prop:star-box-degree-zero} applies termwise.  This
proves the first equivalence.

For the second, let $K$ be such a complex.  Since the relation $dV^r$
has positive internal degree, $(R_r)^0=R^0/V^rR^0$.  Left
multiplication by $V^r$ is injective on $R^0$
(Remark~\ref{lem:raynaud-normal-form}), so $R^0/V^rR^0$ has a two-term
free resolution, and
\[
 (R_r\otimes_R^{L}K)^0\simeq
 \cofib\bigl(V^r:\sigma^r_*K^0\to K^0\bigr),
\]
compatibly with the transition maps.  Taking internal degree zero
commutes with inverse limits, so $(\widehat K)^0\simeq(K^0)^\wedge_V$.
Apply this to $K=M\star^{L}N$.
\end{proof}

For classical Dieudonn\'e modules $M,N$ in internal degree zero, this
recovers the completed box product in the heart:
$H^0\bigl((M\widehat\star^{L}N)^0\bigr)\simeq H^0\bigl((M\boxtimes_WN)^\wedge_V\bigr)$,
which is the completed box product of
\cite[Proposition~4.14]{AntieauNikolaus}.  Indeed $(-)^\wedge_V$ preserves
complexes in degrees $\leq0$, and on them $H^0(C^\wedge_V)$ depends
only on $H^0C$, since the quotients by $V^r$ have surjective
transition maps.

\begin{example}[The star product $k\star k$]
\label{ex:star-box-dV}
Regard $k$ as a graded $R$-module in internal degree zero, with
$F=V=d=0$.  Put $e=1\star1$.  The only relations on $e$ are
$pe=Fe=de=0$, so $k\star k=R/(Rp+RF+Rd)$, and the normal forms of
Remark~\ref{lem:raynaud-normal-form} give
\[
 (k\star k)^0=\bigoplus_{j\geq0}k\,V^je,\qquad
 (k\star k)^1=\bigoplus_{j\geq1}k\,dV^je,
\]
with $F(dV^je)=dV^{j-1}e$ and $V(dV^je)=0$ in degree one.  The part of
internal degree zero is the box product $k\boxtimes_Wk$, and its
$V$-adic completion $\prod_{j\geq0}k\,V^je$ is the completed box
product.
Completing the whole star product gives, in cohomological degree zero,
\[
 \prod_{j\geq0}k\,V^je\xrightarrow{\ d\ }\prod_{j\geq1}k\,dV^je,
\]
the elementary domino $U_1=\widehat R/\widehat R(F,d)$ of Illusie and
Raynaud \cite[I, \S2]{IllusieRaynaud83}.  This is the degree zero part
of Ekedahl's computation of $k\,\widehat\star^{L}k$, whose cohomology is
$U_1$, $U_0^{\oplus2}$ and $U_{-1}$ in degrees $0$, $-1$ and $-2$
\cite[\S III.6]{Ekedahl85}.  So the star product also records the
differentials $d(V^je)=dV^je$ of the new $V$-terms, which the box
product forgets.
\end{example}

\begin{remark}[{Free modules and level one; \cite[Chapter~I]{Ekedahl85}}]
\label{rem:star-free-level-one}
The star product of free graded $R$-modules is again free, since
$R\star R$ is free with generators in internal degrees $0,1$.  At level
one the star product becomes the tensor product over $k$: for graded
$R$-modules $M,N$, the map $x\otimes y\mapsto x\star y$ induces an
isomorphism
\begin{equation}\label{eq:level-one-star}
 (R_1\otimes_RM)\otimes_k(R_1\otimes_RN)\xrightarrow{\sim}
 R_1\otimes_R(M\star N).
\end{equation}
\end{remark}

On the side of $F$-gauges, the finite realizations carry cup pairings.

\begin{construction}[Cup pairings]
\label{con:cup-pairings}
In the description
$T_{\dRW_r}(E)\simeq(\sigma^{1-r})^*\bigl((g'_r)_*e_r^*E\bigr)[\bullet]$ of
Lemma~\ref{lem:T-r-Hodge}, the functors $e_r^*$, $(\sigma^{1-r})^*$ and
regrading are symmetric monoidal, and $(g'_r)_*$ is lax symmetric
monoidal as the right adjoint of $(g'_r)^*$.  So $T_{\dRW_r}$ is lax
symmetric monoidal, and we call its structure maps
\[
 T_{\dRW_r}(E)\otimes_{W_r}^{L}T_{\dRW_r}(E')
 \longrightarrow T_{\dRW_r}(E\otimes E')
\]
the \emph{cup pairings}.  For $r=1$ the map $g'_1$ is the identity, so
$T_{\dRW_1}(E)\simeq E_{\mathrm{Hodge}}[\bullet]$
(Lemma~\ref{lem:T-r-Hodge}) is symmetric monoidal.
\end{construction}

\begin{lemma}[Operators and cup pairings]
\label{lem:V-cup-projection}
For $E,E'\in\Dqc(k^{\syn})$ and $x$ of internal degree $i$, the cup
pairings satisfy
\[
\begin{aligned}
 \pi_r(x\cup y)&=\pi_rx\cup\pi_ry,&
 F_r(x\cup y)&=F_rx\cup F_ry,\\
 V_r(x)\cup y&=V_r(x\cup F_ry),&
 x\cup V_r(y)&=V_r(F_rx\cup y),\\
 d(x\cup y)&=dx\cup y+(-1)^ix\cup dy,
\end{aligned}
\]
where the argument of $V_r$ lies in $\sigma_*T_{\dRW_r}$ and the other
factor in $T_{\dRW_{r+1}}$.  In the limit over $r$, the cup products
$T_{\dRW}(E)\otimes_W^{L}T_{\dRW}(E')\to T_{\dRW}(E\otimes E')$ are
$(F,V,d)$-bilinear (Definition~\ref{def:star-product}).
\end{lemma}

\begin{proof}
Restriction and Frobenius are pullbacks along the restriction map
$\mathsf R_r$ and the Frobenius map $\mathsf F_r$, so they are
multiplicative.  The operator $d$ is the boundary map $\eta$ on
$\Ksz(t)$, a derivation, and regrading turns it into the graded Leibniz
rule.  Verschiebung is the trace along the finite map $\mathsf F_r$
(\S\ref{sec:geometric-verschiebung}), and the trace is linear over
$\cO_{k^{\mathrm{Hodge}_{r+1}}}$:
\[
 \operatorname{tr}_{\mathsf F_r}\bigl(x\cdot\mathsf F_r^*y\bigr)
 =\operatorname{tr}_{\mathsf F_r}(x)\cdot y .
\]
Applied to $e_{r+1}^*E$ and $e_{r+1}^*E'$ and pushed forward as in
Construction~\ref{con:cup-pairings}, this projection formula gives the
first identity for $V_r$, and graded symmetry gives the second, since
$V$ has degree zero.  The cup pairings commute with $\pi_r$, so they
pass to the limit over $r$.
\end{proof}

\begin{theorem}[Symmetric monoidal equivalence]
\label{thm:symmetric-monoidal-equivalence}
There is a symmetric monoidal refinement of
Theorem~\ref{thm:presentable-equivalence},
\[
 \bigl(\Dqc(k^{\syn}),\otimes\bigr)
 \xrightarrow[T_{\dRW}]{\sim}
 \bigl(\widehat D(R),\widehat\star^{L}\bigr).
\]
It restricts to a symmetric monoidal equivalence
$\Perf(k^{\syn})\simeq D_c^b(R)$
(Corollary~\ref{thm:gauge-raynaud-equivalence}).
\end{theorem}

\begin{proof}
We make $S_{\mathrm{FG}}$ symmetric monoidal on graded free modules and
then extend it.

\emph{Free modules.}  Let $P,Q$ be graded free $R$-modules generated in
a bounded range of internal degrees.  Then so is $P\star Q$
(Remark~\ref{rem:star-free-level-one}), and the pairing
$(x,y)\mapsto\eta(x)\cup\eta(y)$ is $(F,V,d)$-bilinear
(Lemma~\ref{lem:V-cup-projection}).  It induces a map
\[
 \beta_{P,Q}:P\star Q\longrightarrow
 T_{\dRW}\bigl(S_{\mathrm{FG}}(P)\otimes S_{\mathrm{FG}}(Q)\bigr),
\]
unique up to homotopy since $P\star Q$ is free, with adjoint
$\gamma_{P,Q}:S_{\mathrm{FG}}(P\star Q)\to S_{\mathrm{FG}}(P)\otimes
S_{\mathrm{FG}}(Q)$.  By \eqref{eq:raynaud-reduction},
Proposition~\ref{prop:finite-kernel-calculation} and
\eqref{eq:level-one-star}, $R_1\otimes_R\beta_{P,Q}$ is the level-one
cup pairing, an equivalence (Construction~\ref{con:cup-pairings}).
Both sides of $\beta_{P,Q}$ are bounded above.  For the target, by
Lemmas~\ref{thm:ladder-is-quasi-inverse} and~\ref{lem:free-ladder} and
\eqref{eq:universal-G-chart}, the underlying gauges of
$S_{\mathrm{FG}}(P)$ and $S_{\mathrm{FG}}(Q)$ are two-term complexes of
completed direct sums of free graded $A$-modules with generator weights
in a finite interval, so that of
$S_{\mathrm{FG}}(P)\otimes S_{\mathrm{FG}}(Q)$ is a bounded complex of
the same kind, and
Lemma~\ref{lem:finite-weight-bounds} applies.
So $\beta_{P,Q}$ becomes an equivalence after completion, by Ekedahl's
Nakayama lemma \cite[Corollary~I.1.1.3]{Ekedahl85}.  As $\eta$ is the
completion unit \eqref{eq:NL-completion}, $T_{\dRW}(\gamma_{P,Q})$ is an
equivalence, hence so is $\gamma_{P,Q}$ (Lemma~\ref{lem:N-conservative}).

\emph{Coherence.}  A map of graded free modules sends the summand
generated in internal degree $n$ into those generated in degrees $n-1$
and $n$, and $S_{\mathrm{FG}}$, $\star$ and $\otimes$ preserve direct
sums.  So the $\gamma_{P,Q}$ extend to all graded free $P,Q$.  The
completion of such a $P$ is a module, as the maps
$R_{r+1}\otimes_RP\to R_r\otimes_RP$ are surjective, so by
Theorem~\ref{thm:presentable-equivalence} the mapping spaces between the
$S_{\mathrm{FG}}(P)$ are discrete.  Since $\gamma_{P,Q}$ identifies
$S_{\mathrm{FG}}(P)\otimes S_{\mathrm{FG}}(Q)$ with
$S_{\mathrm{FG}}(P\star Q)$, and $P\star Q$ is again graded free
\cite[Proposition~I.3.2]{Ekedahl85}, the mapping spaces from tensor
products of such objects are discrete as well.  So the $\infty$-operads involved are
nerves of ordinary categories
\cite[Propositions~2.3.4.5 and~2.3.4.18]{LurieHTT}, and $S_{\mathrm{FG}}$
is non-unital symmetric monoidal on graded free modules once
associativity and symmetry hold, and these may be checked on elements $x\star y\star z$,
where they are the associativity and graded commutativity of cup
products.

\emph{Extension and unit.}  Li and Yang obtain $\star^{L}$ from $\star$
on graded free modules by monoidal cocompletion and stabilization, with
$W$ as a quasi-unit \cite[Proposition~3.30]{LiYang}.  As $S_{\mathrm{FG}}$
preserves colimits, it extends to a non-unital symmetric monoidal
functor on $D(R)$ \cite[\S4.8.1]{LurieHA}.  The Nygaard filtration of
$W$ is $p$-adic, so $S_{\mathrm{FG}}(W)\simeq\cO_{k^{\syn}}$ by
Lemmas~\ref{thm:ladder-is-quasi-inverse} and~\ref{lem:saturated-ladder}.
So $S_{\mathrm{FG}}$ carries the quasi-unit $W$ to $\cO_{k^{\syn}}$,
and tensoring with $\cO_{k^{\syn}}$ is the identity, so $S_{\mathrm{FG}}$
is a quasi-unital map.  It is therefore symmetric monoidal, since
quasi-unital commutative algebras in $\infty$-categories and their
quasi-unital maps are unital \cite[Corollary~5.4.4.7]{LurieHA}.  By \eqref{eq:NL-completion} the
localization it induces is Raynaud completion, which is symmetric
monoidal with product $\widehat\star^{L}$.  The restriction to perfect
objects follows from Corollary~\ref{thm:gauge-raynaud-equivalence}.
\end{proof}

We now deduce Ekedahl's K\"unneth formula for de Rham--Witt cohomology
from the K\"unneth formula for $F$-gauge cohomology, which can be
checked at the Hodge point.

\begin{corollary}[{K\"unneth formula for de Rham--Witt cohomology; compare \cite[Theorem~II.1.1]{Ekedahl85} and \cite[Theorem~3.36]{LiYang}}]
\label{cor:geometric-kunneth}
Let $X$ and $Y$ be smooth proper over $k$.  The K\"unneth map
$\cH_{\syn}(X)\otimes\cH_{\syn}(Y)\to\cH_{\syn}(X\times_kY)$ is an
equivalence, and $T_{\dRW}$ carries it to an equivalence
\[
 R\Gamma(X,W\Omega^\bullet_{X/k})
 \widehat\star^{L}
 R\Gamma(Y,W\Omega^\bullet_{Y/k})
 \simeq
 R\Gamma(X\times_kY,W\Omega^\bullet_{X\times_kY/k}).
\]
\end{corollary}

\begin{proof}
The two projections and multiplication give a natural map
\[
 \cH_{\syn}(X)\otimes\cH_{\syn}(Y)
 \longrightarrow\cH_{\syn}(X\times_kY).
\]
Both sides are perfect \cite[Remark~4.2.3]{BhattFGauges}.  At the
origin $B\mathbb G_{m,k}$ of $k^{\Nyg}$, the Hodge comparison of
Theorem~\ref{thm:ordinary-input}(6), which is natural and
multiplicative, identifies this map in weight $n$ with coherent
K\"unneth:
\[
 \bigoplus_{i+j=n}
 R\Gamma(X,\Omega^i_{X/k})\otimes_k R\Gamma(Y,\Omega^j_{Y/k})
 \xrightarrow{\sim}
 R\Gamma(X\times_kY,\Omega^n_{X\times_kY/k}).
\]
This is an equivalence by the decomposition of forms on a
product and the affine \v{C}ech formula over the field $k$.
The cone is perfect with zero fibre at the Hodge point, so it vanishes
by Lemma~\ref{lem:graded-perfect-lifting}(1) and~(3).
Apply Corollary~\ref{cor:geometric-raynaud-realization} and
Theorem~\ref{thm:symmetric-monoidal-equivalence}.
\end{proof}

\begin{remark}[Failure of the K\"unneth formula at finite level]
\label{rem:finite-level-not-monoidal}
To recover the Hodge-filtered $r$-de Rham--Witt complex, the finite
realization pulls back along $e_r:k^{\mathrm{Hodge}_r}\to k^{\syn}$ and
then pushes forward along $g'_r:k^{\mathrm{Hodge}_r}\to B\mathbb G_{m,W_r}$
(Lemma~\ref{lem:T-r-Hodge}).  The pullback is symmetric monoidal, but
the pushforward is only lax symmetric monoidal.  For $r=1$ the map
$g'_1$ is the identity, so $T_{\dRW_1}$ is symmetric monoidal.  For
$r\geq2$ the cup pairings are not equivalences in general.  For
instance, for the $R$-module $M=k$ of
Example~\ref{ex:star-box-dV} and $E=S_{\mathrm{FG}}(M)$, the cup pairing
of $T_{\dRW_r}(E)$ in complex and internal degree zero is the map
$k\to\bigoplus_{0\leq j<r}k\,V^je$ with image $ke$, by
Theorems~\ref{thm:presentable-equivalence}
and~\ref{thm:symmetric-monoidal-equivalence}.  At $r=\infty$, the
de Rham--Witt realization $T_{\dRW}$ is again symmetric monoidal, for
the completed star product
(Theorem~\ref{thm:symmetric-monoidal-equivalence}).

This explains a classical phenomenon.  For rings $A$ and $B$,
$W(A)\widehat\boxtimes W(B)\cong W(A\otimes B)$
\cite[Theorem~4.16]{AntieauNikolaus}, but for $k$-algebras and $r\geq2$
the map $W_r(A)\otimes_{W_r}W_r(B)\to W_r(A\otimes_kB)$ is not
surjective in general.  For instance, the image of $W_2(k[x])\otimes_{W_2}W_2(k[y])$
misses $V[xy]$, since it is spanned by products of $[x^n]$, $V[x^n]$,
$[y^m]$ and $V[y^m]$, and the nonzero ones have a weight with an
integral coordinate for the $\mathbb Z[1/p]^2$-grading of
\cite[I, \S2]{Illusie79}.
\end{remark}

\subsection{Duality}
\label{sec:geometric-ekedahl-duality}

In this subsection, we show that Ekedahl's derived equivalence carries
the duality of perfect $F$-gauges to Ekedahl's dualizing functor
(Theorem~\ref{cor:duality-comparison}), and deduce Ekedahl's
Poincar\'e duality for de Rham--Witt cohomology
\cite[Theorem~III.2.9]{Ekedahl84} from Poincar\'e duality for $F$-gauge cohomology
\cite[Theorem~4.5.3]{BhattFGauges}
(Corollary~\ref{thm:geometric-Ekedahl-duality}).

Duality exchanges restriction with its transpose.  On de Rham--Witt
complexes, the transpose is the injective map
$\rho:W_r\Omega^\bullet_{X/k}\to W_{r+1}\Omega^\bullet_{X/k}$ induced by
multiplication by $p$, which makes $\{W_r\Omega^\bullet_{X/k}\}_r$ a
direct system \cite[\S III.2]{Ekedahl84}.

\begin{lemma}[The transpose of restriction]
\label{lem:scalar-transfer-projection}
The oriented transpose of restriction is a natural map
$\rho_r:T_{\dRW_r}\to T_{\dRW_{r+1}}$.  Under
\eqref{eq:raynaud-reduction} it is induced by the map
$\rho_r:R_r\to R_{r+1}$, $[x]\mapsto[px]$
\cite[(0.4.2)]{Ekedahl84}.  For $E,E'\in\Dqc(k^{\syn})$ it
satisfies the projection formula
\[
 \rho_r(x)\cup y=\rho_r(x\cup\pi_r y)
\]
as maps out of $T_{\dRW_r}(E)\otimes_{W_{r+1}}^{L}T_{\dRW_{r+1}}(E')$, with
$T_{\dRW_r}(E)$ viewed as a $W_{r+1}$-complex by reduction.
\end{lemma}

\begin{proof}
On charts, $\rho_r$ is described as Verschiebung is in
Construction~\ref{def:intrinsic-geometric-V}, with restriction in place
of Frobenius.  The restriction map $\mathsf R_r$ embeds the level-$r$
chain as the right sub-chain of the level-$(r+1)$ chain.  There
$\rho_r$ has entries $1$ and $p$ on the Koszul resolution of the
$u$-equation, on the first chart of the longer chain it is zero, and on
the first overlap it is the homotopy \eqref{eq:V-overlap-homotopy} with
$u$ in place of $t$.

As $\rho_r$ is natural, it commutes with the right action of $R$ on the
$G_n$, so as in the proof of
Proposition~\ref{prop:finite-kernel-calculation} it is determined by
its value on the class $c$ of Remark~\ref{rem:canonical-class}, which
corresponds to $1$.  As $c$ is dual to the unit of $\Ksz(u)$, $\rho_r$ sends it to $p$ times the class
$c$ at level $r+1$, that is, to $p$.  The zero on the first chart is consistent with
this, since there $p$ is null-homotopic through $t_0\varepsilon_u$, as
$\partial(t_0\varepsilon_u)=p$.  So $\rho_r$ is the map
$R_r\to R_{r+1}$, $[x]\mapsto[px]$.  The projection formula follows as
in the proof of Lemma~\ref{lem:V-cup-projection}, with restriction in
place of Frobenius.
\end{proof}

To dualize the pushforward along $g_r$ we need its relative
dualizing complex.

\begin{proposition}[The contraction is proper and crepant]
\label{prop:contraction-crepant}
The contraction $g_r$ is proper, and it is crepant: writing $g_r^!$
for the right adjoint of $g_{r,*}$, there is a canonical
identification $g_r^!\cO_{k^{\Nyg_r}}\simeq\cO_{k^{\NN_r}}$.  It is the
identity over $(k^{\Nyg_r})_{u\neq0}$, where $g_r$ is
an isomorphism (Lemma~\ref{lem:contraction-open-loci}).
\end{proposition}

\begin{proof}
Identify each $W^{(i)}$ with $W$ by $a^{(i)}\mapsto\sigma^i(a)$, so that
the coefficient maps in the gluings and in $g_r$ become the identity.
By Proposition~\ref{prop:contraction-map}, $g_r$ is then obtained from
the toric resolution of the $A_{r-1}$-singularity
$\Spec W[z,u,t]/(ut-z^r)$, with charts $\Spec W[u_i,t_i]$ where
\[
 z=u_it_i,\qquad u=z^iu_i,\qquad t=z^{r-1-i}t_i,
\]
by base change along $z\mapsto p$, $p$-adic completion and passage to
the $\Gm$-quotient.  The resolution is proper, as its fan subdivides
the original one without changing the support, and so is $g_r$.

It is also crepant.  By adjunction for the hypersurface $ut=z^r$, the
relative dualizing module of $W[z,u,t]/(ut-z^r)$ over $W$ is generated
by the residue
\[
 \frac{du}{u}\wedge dz=z\,d\log u\wedge d\log z,
\]
which has weight zero since $|u|=1$ and $|z|=0$.  On the $i$-th chart $z=u_it_i$ and
$u=u_i^{i+1}t_i^i$, with exponent matrix of determinant one, so
$d\log u\wedge d\log z=d\log u_i\wedge d\log t_i$ and the generator
becomes $du_i\wedge dt_i$, which has neither zeros nor poles.  So the
right adjoint of pushforward along the resolution takes $\cO$ to $\cO$,
without a shift since source and target have the same relative
dimension over $W$.

We pass to $g_r$ in three steps.  First, $z-p$ is a non-zero-divisor
on the domain $W[z,u,t]/(ut-z^r)$ and on each chart, where it equals
$u_it_i-p$.  So the base change along $z=p$ is Tor-independent and has
Tor dimension one.  The base is Noetherian and $\cO$ is bounded below,
so the right adjoint of pushforward commutes with this base change
\cite[Tag~0AA8]{StacksProject}.  Second, $p^n$ is a non-zero-divisor on
$W[u,t]/(ut-p^r)$ and on each chart $W[u_i,t_i]/(u_it_i-p)$, so the
same lemma applies to the reduction modulo $p^n$.  Third, base change
maps compose, so the identifications modulo $p^n$ and modulo $p^{n+1}$
are compatible, and their limit is an identification on the $p$-adic
completions.  It is $\Gm$-equivariant, as the generator has weight zero,
so it descends to the $\Gm$-quotient and gives
$g_r^!\cO_{k^{\Nyg_r}}\simeq\cO_{k^{\NN_r}}$.  On the $0$-th chart
$u=u_0$, so this is the identity over $(k^{\Nyg_r})_{u\neq0}$.
\end{proof}

We can now dualize one level at a time.

\begin{proposition}[Finite duality]
\label{prop:finite-extraction-preserves-duals}
\label{def:finite-duality-pairing}
For $E\in\Perf(k^{\syn})$, let $b_r$ be the cup pairing followed by
evaluation,
\[
 b_r:T_{\dRW_r}(E^\vee)\otimes_{W_r}^{L}T_{\dRW_r}(E)
 \longrightarrow T_{\dRW_r}(E^\vee\otimes E)\longrightarrow T_{\dRW_r}(\cO)=W_r,
\]
where $T_{\dRW_r}(\cO)=W_r$ by Corollary~\ref{cor:geometric-raynaud-realization}
for $X=\Spec k$.  The adjoint of $b_r$ is a natural equivalence
$T_{\dRW_r}(E^\vee)\simeq\RHomgr_{W_r}(T_{\dRW_r}(E),W_r)$.
\end{proposition}

\begin{proof}
Put $\mathcal E=f_r^*((\sigma^{1-r})^*E)$, a perfect genuine $r$-gauge,
and $\cO=\cO_{B\mathbb G_{m,W_r}}$.  The regular sequence $(u,t)$ gives
$(i_{\mathrm{Hodge}_r,*}\cO)^\vee\simeq i_{\mathrm{Hodge}_r,*}\cO[-2]$.
With the projection formula and duality for the proper crepant map
$g_r$ (Proposition~\ref{prop:contraction-crepant} and
\cite[Tag~0GEW]{StacksProject}, applied modulo $p^n$), this gives
\begin{align*}
 &i_{\mathrm{Hodge}_r,*}\RHomgr_{W_r}(i_{\mathrm{Hodge}_r}^*g_{r,*}\mathcal E,W_r)\\
 &\qquad\simeq
 R\mathcal Hom_{\cO_{k^{\Nyg_r}}}\bigl(g_{r,*}(\mathcal E\otimes g_r^*i_{\mathrm{Hodge}_r,*}\cO),\cO_{k^{\Nyg_r}}\bigr)[2]\\
 &\qquad\simeq
 g_{r,*}\bigl(\mathcal E^\vee\otimes(g_r^*i_{\mathrm{Hodge}_r,*}\cO)^\vee\bigr)[2]\\
 &\qquad\simeq g_{r,*}(\mathcal E^\vee\otimes g_r^*i_{\mathrm{Hodge}_r,*}\cO)\\
 &\qquad\simeq i_{\mathrm{Hodge}_r,*}i_{\mathrm{Hodge}_r}^*g_{r,*}\mathcal E^\vee,
\end{align*}
using in the second step that $\mathcal E\otimes
g_r^*i_{\mathrm{Hodge}_r,*}\cO$ is perfect.  These are equivalences of
$i_{\mathrm{Hodge}_r,*}\cO$-modules, that is, of complexes on
$B\mathbb G_{m,W_r}$, and regrading commutes with duals
(Convention~\ref{conv:shearing}).

Tracing the adjunctions identifies this equivalence with cup product,
followed by evaluation and by the trace
$g_{r,*}\cO_{k^{\NN_r}}\simeq g_{r,*}g_r^!\cO_{k^{\Nyg_r}}\to\cO_{k^{\Nyg_r}}$.
This trace is inverse to the isomorphism
$\cO_{k^{\Nyg_r}}\simeq g_{r,*}\cO_{k^{\NN_r}}$ of
Remark~\ref{rem:Rg-structure-sheaf}, since both are the identity over
$(k^{\Nyg_r})_{u\neq0}$ (Proposition~\ref{prop:contraction-crepant}),
which detects the scalar.  So it has the normalization of $b_r$.
\end{proof}

\begin{definition}[{The dual finite system; \cite[\S III.2]{Ekedahl84}}]
\label{def:dual-system}
Write $\rho_r:W_r\hookrightarrow W_{r+1}$, $\bar a\mapsto p\widetilde a$,
for any lift $\widetilde a$.  For a complex $M$ of graded $R$-modules,
the \emph{dual system} of $\{R_r\otimes_R^{L}M\}_r$ has terms
$\RHomgr_{W_r}(R_r\otimes_R^{L}M,W_r)$, with operators
$\pi^D,F^D,V^D,d^D$ characterized by
\[
 \begin{aligned}
 \rho_r\bigl((\pi^Dy)(x)\bigr)&=y(\rho_r x),\\
 \rho_r\bigl((F^Dy)(x)\bigr)&=\sigma\bigl(y(V_r x)\bigr),\\
 (V^Dy)(x)&=\rho_r\bigl(\sigma^{-1}(y(F_r x))\bigr),\\
 (d^Dy)(x)&=(-1)^{n+1}y(dx),
 \end{aligned}
\]
where in the first two lines $y$ is at level $r+1$ and $x$ at level
$r$, in the third the levels are reversed, and in the last $y$ has
internal degree $n$.  The dual of the internal-degree $-n$ component
lies in internal degree $n$.
\end{definition}

The maps $\bar a\mapsto p^{-r}\widetilde a$ embed each $W_r$ into
$W[1/p]/W$, compatibly with $\rho_r$, and identify
$\RHomgr_{W_r}(C,W_r)$ with $\RHomgr_W(C,W[1/p]/W)$.  So the dual system is precomposition into
the injective $W$-module $W[1/p]/W$, with $\sigma^{\pm1}$ on values for
$F$ and $V$, and it dualizes the whole diagram with its homotopies.

\begin{lemma}[Duality of the finite system]
\label{lem:finite-system-duality}
For $E\in\Perf(k^{\syn})$, the equivalences of
Proposition~\ref{prop:finite-extraction-preserves-duals}
intertwine $\pi,F,V,d$ on $\{T_{\dRW_r}(E^\vee)\}_r$ with the dual system of
$\{T_{\dRW_r}(E)\}_r\simeq\{R_r\otimes_R^{L}T_{\dRW}(E)\}_r$.
\end{lemma}

\begin{proof}
On the unit $\cO_{k^{\syn}}=S_{\mathrm{FG}}(W)$, \eqref{eq:raynaud-reduction}
gives $R_r\otimes_R^{L}W=W_r$, with restriction the reduction, its
transpose $\rho_r$, $F_r=\sigma\circ\pi_r$, $V_r=\rho_r\circ\sigma^{-1}$
and $d=0$.  Applying the projection formulas of
Lemmas~\ref{lem:V-cup-projection} and~\ref{lem:scalar-transfer-projection}
to the evaluation $E^\vee\otimes E\to\cO$ gives, with levels as in
Definition~\ref{def:dual-system},
\[
 \begin{aligned}
 \rho_r b_r(\pi_r y,x)&=b_{r+1}(y,\rho_r x),\\
 \rho_r b_r(F_r y,x)&=\sigma b_{r+1}(y,V_r x),\\
 b_{r+1}(V_r y,x)&=\rho_r\sigma^{-1}b_r(y,F_r x),\\
 b_r(dy,x)&=(-1)^{n+1}b_r(y,dx),
 \end{aligned}
\]
the last by the Leibniz rule and $d=0$ on $W_r$.  These are the
formulas of Definition~\ref{def:dual-system}.  The projection formulas
hold as module maps with their gluing homotopies, and passage to Hom
into $W[1/p]/W$ preserves the whole diagram.
\end{proof}

\begin{definition}[{Ekedahl's dualizing functor; \cite[\S III.3]{Ekedahl84}}]
\label{def:raynaud-dual}
Let $\check R=\varprojlim_r\Hom_{W_r}(R_r,W_r)$, with transition maps
dual to $\rho_r:R_r\to R_{r+1}$.  It has two left $R$-module
structures, the first from the right action of $R$ on each $R_r$ and
the second from the operators $F,V,d$ on the system $\{R_r\}$.
\emph{Ekedahl's dualizing functor} is $\Dual_R(M)=\RHom_R(M,\check R)$,
formed for the first structure, with $R$ acting on the result through
the second.
\end{definition}

\begin{proposition}[{Finite reductions of the dual; \cite[Proposition~III.4.2]{Ekedahl84} and \cite[\S3.1]{LiYang}}]
\label{prop:raynaud-duality}
For $M\in D_c^b(R)$, the finite reductions are
$R_r\otimes_R^{L}\Dual_R(M)\simeq\RHomgr_{W_r}(R_r\otimes_R^{L}M,W_r)$, and under this
identification the reduction system of $\Dual_R(M)$ is the dual
system of $\{R_r\otimes_R^{L}M\}_r$.
\end{proposition}

\begin{proof}
The equivalence of finite reductions is
\cite[Proposition~III.4.2]{Ekedahl84}.  Li and Yang show that it
carries the transition maps and $F,V$ of the reduction system to
precomposition with $\rho_r$ and to the maps of
Definition~\ref{def:dual-system}.  They let $d$ act by unsigned
precomposition.  Multiplication by $(-1)^{n(n+1)/2}$ in internal degree
$n$ turns this into the signed action of
Definition~\ref{def:dual-system} and leaves $\pi,F,V$ unchanged.
\end{proof}

Putting the two sides together gives the comparison.

\begin{theorem}[Duality]\label{cor:duality-comparison}
Ekedahl's derived equivalence is compatible with duality: for
$E\in\Perf(k^{\syn})$ and $M\in D_c^b(R)$ there are natural
equivalences
\[
 T_{\dRW}(E^\vee)\simeq\Dual_R(T_{\dRW}(E)),\qquad
 S(\Dual_RM)\simeq S(M)^\vee.
\]
In particular, on $D_c^b(R)$ Ekedahl's dualizing functor is the
duality for the completed star product.
\end{theorem}

\begin{proof}
Combine \eqref{eq:raynaud-reduction},
Proposition~\ref{prop:finite-extraction-preserves-duals}, and
Proposition~\ref{prop:raynaud-duality}.
Lemma~\ref{lem:scalar-transfer-projection} identifies the
geometric $\rho_r$ with the map induced by $\rho_r:R_r\to R_{r+1}$,
and Lemma~\ref{lem:finite-system-duality} makes these comparisons
an equivalence of the full inverse systems.
Taking the derived inverse limit therefore gives an
$R$-linear equivalence, both objects being complete.
The second equivalence is the first for $E=S(M)$, which is perfect
with $T_{\dRW}(E)\simeq M$ by
Corollary~\ref{thm:gauge-raynaud-equivalence}.  The last statement
follows from Theorem~\ref{thm:symmetric-monoidal-equivalence}, which
preserves evaluation and coevaluation.
\end{proof}

We now apply the comparison to $F$-gauge cohomology.  Poincar\'e
duality for $F$-gauge cohomology involves a Breuil--Kisin twist, which
$T_{\dRW}$ turns into a shift
(Corollary~\ref{lem:Koszul-BK-shift}).

\begin{corollary}[{Poincar\'e duality for de Rham--Witt cohomology; compare \cite[Theorem~III.2.9]{Ekedahl84}}]
\label{thm:geometric-Ekedahl-duality}
Let $X/k$ be smooth proper of pure dimension $N$.
Poincar\'e duality for $F$-gauge cohomology induces
\[
 \Dual_R(R\Gamma(X,W\Omega^\bullet_{X/k}))(-N)[-N]\xrightarrow{\sim}R\Gamma(X,W\Omega^\bullet_{X/k}).
\]
At finite length this is
\[
 \RHomgr_{W_r}
   \bigl(R\Gamma(X,W_r\Omega^\bullet_{X/k}),W_r\bigr)
   (-N)[-N]
 \xrightarrow{\sim}
 R\Gamma(X,W_r\Omega^\bullet_{X/k}).
\]
These equivalences carry the dual system of
Definition~\ref{def:dual-system} to $\pi,F,V,d$.  In particular,
they exchange Frobenius and Verschiebung under transpose.
\end{corollary}

\begin{proof}
Both sides are additive in $X$, so we may assume $X$ connected.
Then $k'=H^0(X,\cO_X)$ is a finite separable extension of $k$ over
which $X$ is geometrically connected.  The de Rham--Witt complex of
$X$ is the same over $k$ and over $k'$, and the trace of the finite
\'etale extension $W_r(k')/W_r(k)$ identifies
$\RHomgr_{W_r(k')}(-,W_r(k'))$ with $\RHomgr_{W_r(k)}(-,W_r(k))$,
compatibly with $\sigma$, reduction and $\rho_r$.  So we may replace
$k$ by $k'$ and assume $X$ geometrically connected.

Poincar\'e duality for $F$-gauge cohomology, given by the trace pairing, then gives
\[
 \cH_{\syn}(X)^\vee\simeq\cH_{\syn}(X)\{N\}[2N]
\]
in $\Perf(k^{\syn})$ \cite[Theorem~4.5.3]{BhattFGauges}, a result of
Berthelot \cite{Berthelot97} and Tang \cite{TangPrismaticPD}.
Theorem~\ref{cor:duality-comparison} and
Corollary~\ref{lem:Koszul-BK-shift} give
\[
 \Dual_R\bigl(T_{\dRW}\cH_{\syn}(X)\bigr)
 \simeq T_{\dRW}\cH_{\syn}(X)\{N\}[2N]
 =T_{\dRW}\cH_{\syn}(X)(N)[N],
\]
and Corollary~\ref{cor:geometric-raynaud-realization} identifies
$T_{\dRW}\cH_{\syn}(X)$ with $R\Gamma(X,W\Omega^\bullet_{X/k})$.
The finite formula follows in the same way from
Proposition~\ref{prop:finite-extraction-preserves-duals},
Lemma~\ref{lem:BK-finite} and the finite reductions in
Corollary~\ref{cor:geometric-raynaud-realization}, which together with
\eqref{eq:raynaud-reduction} give
$T_{\dRW_r}\cH_{\syn}(X)\simeq R\Gamma(X,W_r\Omega^\bullet_{X/k})$.
Lemma~\ref{lem:finite-system-duality} and
Corollary~\ref{cor:geometric-raynaud-realization} give the operator
compatibility.
\end{proof}

\section{Special values of zeta functions}
\label{sec:zeta}

Let $k=\mathbf F_q$ be a finite field, with $q=p^a$.  For a smooth
proper variety $X$ over $k$, the special values of the zeta function of
$X$ at integers $n$ have two $p$-adic expressions, one through the de Rham--Witt complex of $X$ and
one through its $F$-gauge cohomology.  Milne and Ramachandran proved the
first for every $M\in D_c^b(R)$, assuming that
$q^n$ is at most a simple root of the minimal polynomial of Frobenius
\cite[Theorem~0.1]{MilneRamachandran}, and Illusie proves it
without this assumption \cite[Theorem~5]{IllusieZeta}.  Mondal proves the second for
every perfect $F$-gauge \cite[Theorem~6.5]{MondalZeta}.  In this section, we
show that Ekedahl's derived equivalence matches the invariants entering
the two formulas, so that the two formulas are equivalent
(Theorem~\ref{thm:zeta-comparison}).  We begin by recalling the
invariants, following Illusie \cite[\S\S4--5]{IllusieZeta}.

Throughout this section, $K=W[1/p]$, $\bar k$ is an algebraic closure
of $k$, and $\gamma\in\mathrm{Gal}(\bar k/k)$ is the
Frobenius $x\mapsto x^q$.  As elsewhere in this paper, we write
$M\in D(R)$ and $E\in\Dqc(k^{\syn})$, where Illusie writes $P$ and $M$.

\begin{definition}[Base change]
\label{def:zeta-base-change}
Let $k'$ be a perfect field containing $k$, and let $R_{k'}$ be obtained
from the presentation of $R$ in Definition~\ref{def:raynaud-ring} by
replacing $W$ with $W(k')$ and $\sigma$ with its Witt vector Frobenius.
The inclusion $k\subset k'$ induces $R\to R_{k'}$.  The \emph{base change}
of $M\in D(R)$ is the Raynaud completion $M_{k'}$ of
$R_{k'}\otimes_R^{L}M$.  For a coherent $R$-module $M$, it is given by
\[
 (M_{k'})^i\simeq\varprojlim_r\Bigl(W(k')\otimes_W
 \bigl(M/(V^rM+dV^rM)\bigr)^i\Bigr)
\]
\cite[I, Theorem~3.8 and IV, (3.6.3)]{IllusieRaynaud83}.  The base
change of $E\in\Dqc(k^{\syn})$ is $E_{k'}:=\rho^*E$, where
$\rho:k'^{\syn}\to k^{\syn}$ is the induced map.
\end{definition}

\begin{definition}[{Stable Bockstein characteristic; \cite[Definition~5.2]{MondalZeta} and \cite[\S2]{IllusieZeta}}]
\label{def:stable-bockstein}
Let $C$ be a complex of $\mathbf Z_p$-modules with an endomorphism $u$,
and put $L=\operatorname{fib}(u)$.  The composite of the boundary maps
$L\to C\to L[1]$ of the triangle $L\to C\xrightarrow{u}C\to L[1]$ is a
map $d_u:L\to L[1]$, and the \emph{Bockstein complex}
$\operatorname{Bock}(C,u)$ is $\cdots\to H^i(L)\xrightarrow{d_u}
H^{i+1}(L)\to\cdots$.  If $C$ is bounded and $C\otimes\mathbf Q_p$
and $L$ are perfect, there is an $r_0$ such that for all $r\geq r_0$ the
complex $\operatorname{Bock}(C,u^r)$ has finite-length cohomology and
\[
 \chi_s(C,u):=\frac1r\sum_i(-1)^i
 \operatorname{length}_{\mathbf Z_p}H^i\operatorname{Bock}(C,u^r)
\]
does not depend on $r$ \cite[Proposition~5.1]{MondalZeta}.  We write $\chi_s(C,u)^\times:=p^{\chi_s(C,u)}$.
\end{definition}

\begin{definition}[{Invariants of complexes of $R$-modules; \cite{MilneRamachandran} and \cite[\S4]{IllusieZeta}}]
\label{def:zeta-raynaud}
Let $M\in D_c^b(R)$ and $n\in\mathbb Z$.
\begin{enumerate}
\item The \emph{zeta function} of $M$ is
 \[
  Z(M,t):=\prod_m\det\bigl(1-\varphi_M^at\mid
  H^m(\mathbf s(M)\otimes_WK)\bigr)^{(-1)^{m+1}},
 \]
 where $\varphi_M$ is the $\sigma$-linear automorphism of
 $\mathbf s(M)\otimes_WK$ given by $p^iF$ on $M^i\otimes_WK$.
\item The \emph{syntomic cohomology} of $M$ is
 $R\Gamma(M,\mathbf Z_p(n)):=\RHom_R(W,M\{n\})$, where
 $M\{n\}=M(n)[-n]$ is Illusie's Tate twist $M^{\bullet+n,\bullet-n}$.
\item The \emph{stable Bockstein characteristic} of $M$ is
 $\chi_s(M,\mathbf Z_p(n))^\times:=\chi_s(C,1-\gamma)^\times$
 (Definition~\ref{def:stable-bockstein}) for
 $C=R\Gamma(M_{\bar k},\mathbf Z_p(n)):=\RHom_{R_{\bar k}}(W(\bar k),M_{\bar k}\{n\})$,
 on which $\gamma$ acts through its action on $\bar k$.  The hypotheses
 of Definition~\ref{def:stable-bockstein} hold by
 \cite[Proposition~4.3]{IllusieZeta}.
\item The \emph{Hodge numbers} of $M$ are
 $h^{i,j}(M):=\dim_kH^j\bigl((R_1\otimes_R^{L}M)^i\bigr)$, and
 \[
  \chi(M,n):=\sum_{j,\,i\leq n}(-1)^{i+j}(n-i)\,h^{i,j}(M).
 \]
\end{enumerate}
\end{definition}

\begin{definition}[{Invariants of $F$-gauges; \cite{MondalZeta} and \cite[\S5]{IllusieZeta}}]
\label{def:zeta-gauge}
Let $E\in\Perf(k^{\syn})$ and $n\in\mathbb Z$.
\begin{enumerate}
\item The \emph{$F$-isocrystal} of $E$ is $E^{\infty}[1/p]$, where
 $E^\infty=j_\Delta^*E$, with the Frobenius
 $\varphi_E:\varphi^*E^{\infty}[1/p]\xrightarrow{\sim}E^{\infty}[1/p]$
 induced by the $F$-gauge structure
 \cite[Proposition~4.3.1]{BhattFGauges}.  The
 \emph{zeta function} of $E$ is
 \[
  Z(E,t):=\prod_m\det\bigl(1-\varphi_E^at\mid
  H^m(E^{\infty}[1/p])\bigr)^{(-1)^{m+1}},
 \]
 where the $a$-fold iterate $\varphi_E^a$ is $K$-linear since
 $\sigma^a=\mathrm{id}$ on $K$ \cite[Definition~3.2]{MondalZeta}.
\item The \emph{syntomic cohomology} of $E$ is
 $R\Gamma(E,\mathbf Z_p(n)):=R\Gamma(k^{\syn},E\{n\})$
 \cite[Definition~2.19]{MondalZeta}.
\item The \emph{stable Bockstein characteristic} of $E$ is
 $\chi_s(E,\mathbf Z_p(n))^\times:=\chi_s(C,1-\gamma)^\times$
 (Definition~\ref{def:stable-bockstein}) for
 $C=R\Gamma(E_{\bar k},\mathbf Z_p(n)):=R\Gamma(\bar k^{\syn},E_{\bar k}\{n\})$,
 on which $\gamma$ acts through its action on $\bar k$
 \cite[Definition~6.4]{MondalZeta}.  The hypotheses of
 Definition~\ref{def:stable-bockstein} hold by
 \cite[Corollary~2.21 and Remark~3.3]{MondalZeta}.
\item The \emph{Hodge numbers} of $E$ are
 $h^{i,j}(E):=\dim_kH^{i+j}\bigl(((j_{\Nyg}\circ i_{\mathrm{Hodge}})^*E)^i\bigr)$,
 and $\chi(E,n)$ is defined from them as in
 Definition~\ref{def:zeta-raynaud} \cite[Definitions~6.2--6.3]{MondalZeta}.
\end{enumerate}
\end{definition}

Write $Z(-,t)=(1-q^nt)^{-\rho_n}Z^*(-,t)$ with
$Z^*(-,q^{-n})\neq0$, as in \cite[(4.2.2)]{IllusieZeta}.  The two
special-value formulas read
\[
 \bigl|Z^*(M,q^{-n})\bigr|_p^{-1}
 =\chi_s(M,\mathbf Z_p(n))^\times q^{\chi(M,n)},\qquad
 \bigl|Z^*(E,q^{-n})\bigr|_p^{-1}
 =\chi_s(E,\mathbf Z_p(n))^\times q^{\chi(E,n)}.
\]
Illusie observes that if Ekedahl's equivalence admits a symmetric
monoidal $\infty$-categorical enhancement compatible with Tate twists,
then it matches the invariants on the two sides, and the two formulas
become equivalent \cite[\S\S6.3--6.4]{IllusieZeta}.
Theorem~\ref{thm:intro-equivalence} provides such an enhancement.  Two
further properties of the de Rham--Witt realization are needed: it
commutes with base change (Lemma~\ref{lem:zeta-base-change}), and it
identifies the $F$-isocrystals
(Lemma~\ref{lem:zeta-rational-frobenius}).

\begin{lemma}[Base change]
\label{lem:zeta-base-change}
Let $k'$ be a perfect field containing $k$.  For every $F$-gauge $E$
there is a natural equivalence
\[
 T_{\dRW}(E_{k'})\simeq T_{\dRW}(E)_{k'},
\]
compatible with the action of automorphisms of $k'$ over $k$.  In
particular, for $M\in D_c^b(R)$ we have $M_{k'}\in D_c^b(R_{k'})$ and
$S_{\mathrm{FG}}(M_{k'})\simeq S_{\mathrm{FG}}(M)_{k'}$.
\end{lemma}

\begin{proof}
Let $a:k'^{\NN_r}\to k^{\NN_r}$, $b:k'^{\Nyg_r}\to k^{\Nyg_r}$ and
$\rho:k'^{\syn}\to k^{\syn}$ be the maps induced by $k\subset k'$.  Since
$\sigma$ on $W(k')$ extends $\sigma$ on $W$, the charts and gluings for
$k'$ are obtained from those for $k$ by extending scalars along the
$p$-completely flat map $W\to W(k')$.  So the square formed by $a$, $b$
and the two contractions $g_r$ is cartesian, and so is its restriction
to the $r$-Hodge stacks.  The syntomic point is not a $W$-stack
(Remark~\ref{rem:coefficient-Frobenius-syn}), but the two maps $f_r$
satisfy $f_r\circ a=\rho\circ f_r$, so the pullbacks of
$E_{k'}=\rho^*E$ and of $E$ to the $r$-Hodge stacks correspond under the
map induced by $a$.  Flat base change along the proper map $g_r$
therefore gives
\[
 T_{\dRW_r}(E_{k'})\simeq W_r(k')\otimes_{W_r}T_{\dRW_r}(E),
\]
compatibly with $\pi$, $F$, $V$ and $d$.  By the normal forms
(Remark~\ref{lem:raynaud-normal-form}),
$(R_{k'})_r\otimes_R^{L}M\simeq W_r(k')\otimes_{W_r}(R_r\otimes_R^{L}M)$,
so by \eqref{eq:raynaud-reduction} the two sides of the claim have the
same finite reductions, compatibly with the transition maps.  Both are
complete (Theorem~\ref{thm:presentable-equivalence} and
Definition~\ref{def:zeta-base-change}), which gives the claim,
naturally in $k'$ and hence compatibly with automorphisms.  Finally, for
$M\in D_c^b(R)$ the $F$-gauge $S_{\mathrm{FG}}(M)_{k'}$ is perfect, so
$M_{k'}\simeq T_{\dRW}(S_{\mathrm{FG}}(M)_{k'})$ is coherent by
Corollary~\ref{thm:gauge-raynaud-equivalence}, and
$S_{\mathrm{FG}}(M_{k'})\simeq S_{\mathrm{FG}}(M)_{k'}$.
\end{proof}

\begin{lemma}[Rational Frobenius]
\label{lem:zeta-rational-frobenius}
Let $M\in D_c^b(R)$ and $E=S_{\mathrm{FG}}(M)$.  There is a natural
equivalence of $F$-isocrystals
\[
 \bigl(E^{\infty}[1/p],\varphi_E\bigr)
 \simeq(\sigma^{-1})^*\bigl(\mathbf s(M)\otimes_WK,\varphi_M\bigr).
\]
In particular, $Z(E,t)=Z(M,t)$.
\end{lemma}

\begin{proof}
By Corollary~\ref{thm:gauge-raynaud-equivalence}, $E\simeq S(M)$, so
$E^{-\infty}=\mathbf s(M)$, $E^{\infty}=\sigma_*\mathbf s(M)$, and the
descent datum is the identity of $\mathbf s(M)$.  The Frobenius
$\varphi_E$ is the descent datum composed with the identifications of
$E^{\pm\infty}[1/p]$ with $E^0[1/p]$ by $t$ and $u$
\cite[Remark~3.4.15]{BhattFGauges}.  On the entry of internal degree
$i$, these are composites of the maps $1$, $p$, $F$ and $V$ of
Construction~\ref{con:ladder}: for $i\geq0$ the composite is $p^iF$,
and for $i<0$ it is the inverse of $p^{|i|-1}V$, which is again $p^iF$
since $FV=p$.  These maps commute with the total differential, as $u$
and $t$ are maps of complexes, so $\varphi_E$ is $\varphi_M$
twisted by $(\sigma^{-1})^*$.  The characteristic polynomials of
$\varphi_M^a$ on cohomology have coefficients in $\mathbf Q_p$
\cite[\S4.2]{IllusieZeta}, so $Z(E,t)=Z(M,t)$.
\end{proof}

\begin{theorem}[Comparison of zeta-value invariants]
\label{thm:zeta-comparison}
Let $M\in D_c^b(R)$, $E=S_{\mathrm{FG}}(M)$ and $n\in\mathbb Z$.
\begin{enumerate}
\item $Z(M,t)=Z(E,t)$.  In particular, the orders and the leading
 coefficients at $t=q^{-n}$ agree.
\item There are natural equivalences
 $R\Gamma(M,\mathbf Z_p(n))\simeq R\Gamma(E,\mathbf Z_p(n))$ and
 $R\Gamma(M_{\bar k},\mathbf Z_p(n))\simeq R\Gamma(E_{\bar k},\mathbf Z_p(n))$,
 the second compatible with $\gamma$.
\item $\chi_s(M,\mathbf Z_p(n))^\times=\chi_s(E,\mathbf Z_p(n))^\times$.
\item $h^{i,j}(M)=h^{i,j}(E)$ for all $i,j$, and $\chi(M,n)=\chi(E,n)$.
\end{enumerate}
Hence the special-value formula of \cite[Theorem~5]{IllusieZeta} for $M$
is equivalent to that of \cite[Theorem~6.5]{MondalZeta} for $E$.
\end{theorem}

\begin{proof}
Assertion~(1) is Lemma~\ref{lem:zeta-rational-frobenius}.

For~(2), we have $T_{\dRW}(E)\simeq M$ since $M$ is complete.  The
equivalence $T_{\dRW}$ is symmetric monoidal
(Theorem~\ref{thm:symmetric-monoidal-equivalence}), so it sends
$\cO_{k^{\syn}}$ to the unit $W$, and it commutes with Breuil--Kisin
twists (Corollary~\ref{lem:Koszul-BK-shift}).  Full faithfulness
(Theorem~\ref{thm:presentable-equivalence}) gives
\[
 R\Gamma(k^{\syn},E\{n\})
 =\RHom(\cO_{k^{\syn}},E\{n\})
 \simeq\RHom_R(W,M\{n\}).
\]
The same argument over $\bar k$, with
$T_{\dRW}(E_{\bar k})\simeq M_{\bar k}$ from
Lemma~\ref{lem:zeta-base-change}, gives the second equivalence, and the
naturality in that lemma gives the compatibility with $\gamma$.
Assertion~(3) follows, since the stable Bockstein characteristic
depends only on the complex and its endomorphism.

For~(4), \eqref{eq:raynaud-reduction} for $r=1$ and
Lemma~\ref{lem:T-r-Hodge} give
\[
 R_1\otimes_R^{L}M\simeq T_{\dRW_1}(E)
 \simeq\bigl((j_{\Nyg}\circ i_{\mathrm{Hodge}})^*E\bigr)[\bullet].
\]
By Convention~\ref{conv:shearing}, the weight-$i$ part of the right side
is $((j_{\Nyg}\circ i_{\mathrm{Hodge}})^*E)^i[i]$, so
\[
 H^j\bigl((R_1\otimes_R^{L}M)^i\bigr)
 \simeq H^{i+j}\bigl(((j_{\Nyg}\circ i_{\mathrm{Hodge}})^*E)^i\bigr).
\]
Our weights agree with Illusie's degrees, since on both sides the
Breuil--Kisin twist $\cO\{-1\}$ has Hodge pullback $k$ in weight $1$
\cite[\S5.2]{IllusieZeta}.  So the Hodge numbers agree, and hence so do
the Euler characteristics $\chi(-,n)$.  The last statement follows
from~(1)--(4) \cite[\S6.4]{IllusieZeta}.
\end{proof}

Let $X$ be a smooth proper variety over $k$.  By
Corollary~\ref{cor:geometric-raynaud-realization}, the de Rham--Witt
realization of $\cH_{\syn}(X)$ is $R\Gamma(X,W\Omega^\bullet_{X/k})$,
and both carry the zeta function, the syntomic cohomology and the Hodge
numbers of $X$ \cite[\S\S4.2--5.2]{IllusieZeta}.  So
Theorem~\ref{thm:zeta-comparison} identifies the two $p$-adic formulas
for the special values of the zeta function of $X$.

The two general formulas were proved by different d\'evissages.
Illusie reduces to elementary objects of $D_c^b(R)$ by Ekedahl's
d\'evissage, using properties of stable Bockstein characteristics but
not Mondal's theorem \cite[\S4.5]{IllusieZeta}, and notes that this
d\'evissage had no counterpart for $F$-gauges
\cite[\S5.5]{IllusieZeta}.  Through Ekedahl's derived equivalence it now
applies to perfect $F$-gauges, so Mondal's formula follows from
Illusie's.



\begin{thebibliography}{BZSV24}

\bibitem[And24]{AndriopoulosTR}
Faidon Andriopoulos.
\newblock {$TR$} and the {$r$}-{N}ygaard filtered prismatic cohomology, 2024.
\newblock Preprint, \href{https://arxiv.org/abs/2412.00914}{arXiv:\allowbreak
  2412.00914}.

\bibitem[AN21]{AntieauNikolaus}
Benjamin Antieau and Thomas Nikolaus.
\newblock Cartier modules and cyclotomic spectra.
\newblock {\em Journal of the American Mathematical Society}, 34(1):1--78, 2021.

\bibitem[Ant24]{AntieauDecalage}
Benjamin Antieau.
\newblock Spectral sequences, d\'ecalage, and the {B}eilinson $t$-structure.
\newblock Preprint, arXiv:2411.09115v1, 2024.

\bibitem[Ari21]{Ariotta}
Stefano Ariotta.
\newblock Coherent cochain complexes and {B}eilinson $t$-structures, with an
  appendix by {A}chim {K}rause.
\newblock arXiv:2109.01017, 2021.

\bibitem[Ber97]{Berthelot97}
Pierre Berthelot.
\newblock Dualit\'e de {P}oincar\'e et formule de {K}\"unneth en
  cohomologie rigide.
\newblock {\em C. R. Acad. Sci. Paris S\'er. I Math.}, 325(5):493--498,
  1997.

\bibitem[Bha23]{BhattFGauges}
Bhargav Bhatt.
\newblock Prismatic {$F$}-gauges, 2023.
\newblock Course notes, January 2023. Available at
  \url{https://www.math.ias.edu/~bhatt/teaching/mat549f22/lectures.pdf}.

\bibitem[BL22]{BhattLurie}
Bhargav Bhatt and Jacob Lurie.
\newblock Absolute prismatic cohomology, 2022.
\newblock Preprint, \href{https://arxiv.org/abs/2201.06120}{arXiv:\allowbreak
  2201.06120}.

\bibitem[BLM21]{BLM21}
Bhargav Bhatt, Jacob Lurie, and Akhil Mathew.
\newblock {\em Revisiting the de {R}ham--{W}itt complex}, volume 424 of {\em
  Ast\'erisque}.
\newblock Soci\'et\'e Math\'ematique de France, 2021.

\bibitem[BMS18]{BMS18}
Bhargav Bhatt, Matthew Morrow, and Peter Scholze.
\newblock Integral {$p$}-adic {H}odge theory.
\newblock {\em Publications Math\'ematiques de l'IH\'ES}, 128:219--397, 2018.

\bibitem[Dev25]{DevalapurkarThesis}
Sanath Devalapurkar.
\newblock {\em Spherochromatism in representation theory and arithmetic
  geometry}.
\newblock PhD thesis, Harvard University, 2025.

\bibitem[Dri24]{Drinfeld}
Vladimir Drinfeld.
\newblock Prismatization.
\newblock {\em Selecta Mathematica. New Series}, 30:Paper No.\ 49, 2024.

\bibitem[Dri25]{DrinfeldBT}
Vladimir Drinfeld.
\newblock Ring stacks conjecturally related to the stacks
  $\mathrm{BT}_n^{G,\mu}$.
\newblock Preprint, arXiv:2510.04958v3, 2025.

\bibitem[Eke84]{Ekedahl84}
Torsten Ekedahl.
\newblock On the multiplicative properties of the de {R}ham--{W}itt
  complex. {I}.
\newblock {\em Arkiv f\"or Matematik}, 22(2):185--239, 1984.

\bibitem[Eke85]{Ekedahl85}
Torsten Ekedahl.
\newblock On the multiplicative properties of the de {R}ham--{W}itt
  complex. {II}.
\newblock {\em Arkiv f\"or Matematik}, 23(1):53--102, 1985.

\bibitem[Eke86]{Ekedahl86Diagonal}
Torsten Ekedahl.
\newblock {\em Diagonal complexes and {$F$}-gauge structures}, volume~18 of
  {\em Travaux en Cours}.
\newblock Hermann, Paris, 1986.

\bibitem[FJ21]{FontaineJannsen}
Jean-Marc Fontaine and Uwe Jannsen.
\newblock Frobenius gauges and a new theory of $p$-torsion sheaves in
  characteristic $p$.
\newblock {\em Documenta Mathematica}, 26:65--101, 2021.

\bibitem[Goe99]{Goerss}
Paul~G. Goerss.
\newblock Hopf rings, {D}ieudonn\'e modules, and {$E_*\Omega^2S^3$}.
\newblock In {\em Homotopy invariant algebraic structures}, volume 239 of {\em
  Contemporary Mathematics}, pages 115--174. American Mathematical Society,
  1999.

\bibitem[Hes96]{Hesselholt96}
Lars Hesselholt.
\newblock On the $p$-typical curves in {Q}uillen's {$K$}-theory.
\newblock {\em Acta Mathematica}, 177(1):1--53, 1996.

\bibitem[HL13]{HopkinsLurie}
Michael Hopkins and Jacob Lurie.
\newblock Ambidexterity in {$K(n)$}-local stable homotopy theory, 2013.
\newblock Preprint. Available at
  \url{https://www.math.ias.edu/~lurie/papers/Ambidexterity.pdf}.

\bibitem[HRW25]{HRW}
Jeremy Hahn, Arpon Raksit, and Dylan Wilson.
\newblock A motivic filtration on the topological cyclic homology of
  commutative ring spectra.
\newblock To appear in {\em Annals of Mathematics}.
  \href{https://arxiv.org/abs/2206.11208}{arXiv:\allowbreak 2206.11208}.

\bibitem[Ill79]{Illusie79}
Luc Illusie.
\newblock Complexe de de {R}ham--{W}itt et cohomologie cristalline.
\newblock {\em Annales scientifiques de l'\'Ecole Normale Sup\'erieure, S\'er.
  4}, 12:501--661, 1979.

\bibitem[Ill26]{IllusieZeta}
Luc Illusie.
\newblock Recent developments on zeta values over finite fields, after
  {S}.~{M}ondal, 2026.
\newblock Notes. Available at
  \url{https://www.imo.universite-paris-saclay.fr/~luc.illusie/Illusie-zeta-values-Mondal.pdf}.

\bibitem[IR83]{IllusieRaynaud83}
Luc Illusie and Michel Raynaud.
\newblock Les suites spectrales associ\'ees au complexe de de {R}ham--{W}itt.
\newblock {\em Publications Math\'ematiques de l'IH\'ES}, 57:73--212, 1983.

\bibitem[LM25]{LahotiManam}
Dhilan Lahoti and Deven Manam.
\newblock Cohomology theories in the moduli of ring stacks, 2025.
\newblock Preprint, \href{https://arxiv.org/abs/2510.09582}{arXiv:\allowbreak
  2510.09582}.

\bibitem[Lur09]{LurieHTT}
Jacob Lurie.
\newblock {\em Higher topos theory}, volume 170 of {\em Annals of
  Mathematics Studies}.
\newblock Princeton University Press, Princeton, NJ, 2009.

\bibitem[Lur17]{LurieHA}
Jacob Lurie.
\newblock Higher algebra, 2017.
\newblock September 2017. Available at
  \url{https://www.math.ias.edu/~lurie/papers/HA.pdf}.

\bibitem[LY26]{LiYang}
S.~Li and Y.~Yang.
\newblock On de {R}ham--{W}itt cohomology of classifying stacks, 2026.
\newblock Preprint, \href{https://arxiv.org/abs/2604.03062v1}{arXiv:\allowbreak
  2604.03062v1}.

\bibitem[Mon25]{MondalZeta}
Shubhodip Mondal.
\newblock Zeta function of {$F$}-gauges and special values, 2025.
\newblock Preprint, \href{https://arxiv.org/abs/2501.05027}{arXiv:\allowbreak
  2501.05027}.

\bibitem[MR15]{MilneRamachandran}
James~S. Milne and Niranjan Ramachandran.
\newblock The $p$-cohomology of algebraic varieties and special values of zeta
  functions.
\newblock {\em J. Inst. Math. Jussieu}, 14:801--835, 2015.

\bibitem[NS18]{NikolausScholze}
Thomas Nikolaus and Peter Scholze.
\newblock On topological cyclic homology.
\newblock {\em Acta Mathematica}, 221(2):203--409, 2018.

\bibitem[Nyg81]{Nygaard81}
Niels~O. Nygaard.
\newblock Slopes of powers of {F}robenius on crystalline cohomology.
\newblock {\em Annales scientifiques de l'\'Ecole Normale Sup\'erieure},
  14(4):369--401, 1981.

\bibitem[Rak26]{RaksitHH}
Arpon Raksit.
\newblock Hochschild homology and the derived de {R}ham complex revisited.
\newblock Preprint, 2020; version~3, January 2026.
\newblock \href{https://arxiv.org/abs/2007.02576v3}{arXiv:\allowbreak 2007.02576v3}.

\bibitem[Rig25]{RiggenbachTrunc}
Noah Riggenbach.
\newblock {$K$}-theory of truncated polynomials.
\newblock {\em Mathematische Zeitschrift}, 310: Paper No. 59, 2025.

\bibitem[Sta]{StacksProject}
The {S}tacks {P}roject {A}uthors.
\newblock The {S}tacks project.
\newblock \url{https://stacks.math.columbia.edu}.

\bibitem[Tan24]{TangPrismaticPD}
Longke Tang.
\newblock Syntomic cycle classes and prismatic {P}oincar\'e duality.
\newblock {\em Compositio Mathematica}, 160:2322--2365, 2024.

\bibitem[Yan26]{YangCanonical}
Yuan Yang.
\newblock {E}kedahl's equivalence explained.
\newblock arXiv:2609.39293v2, 2026.

\bibitem[Zha26a]{ZhangDominoes}
Yuanning Zhang.
\newblock Higher dimensional dominoes in de {R}ham--{W}itt cohomology.
\newblock arXiv:2607.26323, 2026.

\bibitem[Zha26b]{ZhangHN}
Yuanning Zhang.
\newblock Harder--{N}arasimhan filtration of torsion {$F$}-gauges.
\newblock arXiv:2609.27089, 2026.

\end{thebibliography}
\end{document}